\documentclass[10pt,twoside,titlepage]{book}
\usepackage{amsthm,amsmath,amstext,amssymb,amscd}
\usepackage{color,graphicx}
\usepackage{subfig}
\usepackage{appendix}
\usepackage[section,above,below,verbose]{placeins}
\usepackage{hyperref}
\numberwithin{equation}{chapter}
\numberwithin{figure}{chapter}
\numberwithin{table}{chapter}

\theoremstyle{plain}
\newtheorem{thm}{Theorem}[chapter]
\newtheorem{lem}[thm]{Lemma}
\newtheorem{prop}[thm]{Proposition}
\newtheorem*{cor}{Corollary}

\theoremstyle{definition}
\newtheorem{defn}{Definition}[chapter]

\newtheorem{example}{Example}[chapter]

\theoremstyle{remark}

\makeatletter
\@ifundefined{showcaptionsetup}{}{\PassOptionsToPackage{caption=false}{subfig}}
\makeatother

\begin{document}
\frontmatter
\begin{titlepage}
	\centering
	\huge
	Stability and chaos in real polynomial maps\\
	\bigskip
	\normalsize
	Thesis presented by\\
	\smallskip
	\Large
	Ferm\'in Franco-Medrano, B.Sc.\\
	\medskip
	\normalsize
	As a requirement to obtain the degree of\\
	\smallskip
	\large
	Master of Science with specialty in Applied Mathematics\\
	\medskip
	\normalsize
	conferred by\\
	\large
	Centro de Investigaci\'on en Matem\'aticas, A.C.\\
	\medskip
	\normalsize
	Thesis Advisor:\\
	\large
	\medskip
	Francisco Javier Sol\'is Lozano, Ph.D.\\
	\vfill
	\includegraphics[width=0.2\columnwidth]{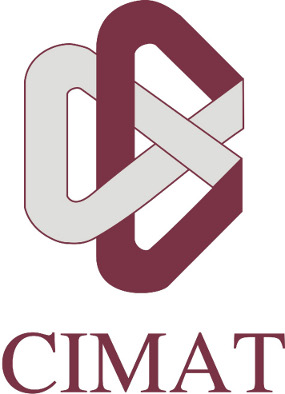}
	\vfill
	\large
	Guanajuato, GT, United Mexican States, June 2013.
\end{titlepage}

\chapter*{}
{\centering
	\huge
	Stability and chaos in real polynomial maps\\
	\bigskip
	\normalsize
	Thesis presented by\\
	\smallskip
	\Large
	Ferm\'in Franco-Medrano, B.Sc.\\
	\medskip
	\normalsize
	As a requirement to obtain the degree of\\
	\smallskip
	\large
	Master of Science with specialty in Applied Mathematics\\
	\medskip
	\normalsize
	conferred by\\
	\large
	Centro de Investigaci\'on en Matem\'aticas, A.C.\\
	\vfill
	\normalsize
	Approved by the Thesis Advisor:\\
	\vfill
	\rule{6.5 cm}{1 pt}\\
	\large
	Francisco Javier Sol\'is Lozano, Ph.D.\\
	\vfill
	\large
	Guanajuato, GT, United Mexican States, June 2013.\\
}
\chapter*{Abstract}

We extend and improve the existing characterization of the dynamics of general quadratic real polynomial maps with coefficients that depend on a single parameter $\lambda$, and generalize this characterization to cubic real polynomial maps, in a consistent theory that is further generalized to real $n$-th degree real polynomial maps. In essence, we give conditions for the stability of the fixed points of any real polynomial map with real fixed points. In order to do this, we have introduced the concept of \emph{Canonical Polynomial Maps} which are topologically conjugate to any polynomial map of the same degree with real fixed points. The stability of the fixed points of canonical polynomial maps has been found to depend solely on a special function termed \emph{Product Distance Function} for a given fixed point. The values of this product distance determine the stability of the fixed point in question, when it bifurcates, and even when chaos arises, as it passes through what we have termed \emph{stability bands}. The exact boundary values of these stability bands are yet to be calculated for regions of type greater than one for polynomials of degree higher than three.
\chapter*{Acknowledgements}

I have both the duty and pleasure of thanking the support of many people and institutions that helped me in the process of conducting my graduate studies. In the first place, I greatly thank the support of the National Council of Science and Technology of the United Mexican States (CONACYT) for their financial support as a national scholar during my studies. I thank my teachers and the personnel of the Center for Mathematical Research (CIMAT) in Guanajuato, for their patient and committed work from which I nurtured and which allowed me to grow in this Mexican institution. In particular, I thank the orientation and guidance of Professor Francisco Javier Sol\'{i}s Lozano, my thesis advisor; as well as Professors L\'{a}zaro Ra\'{u}l Felipe Parada, M\'{o}nica Moreno Rocha and Daniel Olmos Liceaga, who contributed greatly to the improvement of this work with their comments and advice; and Dr. Marcos Aurelio Capistr\'{a}n Ocampo for his support and orientation. I also thank Mr. Jos\'{e} Guadalupe Castro L\'{o}pez and my friends and classmates for their support and invaluable company. To my family and particularly to my mother, I must extend the greatest gratitude of all.
  \tableofcontents
\mainmatter
\chapter{Introduction}
\label{cha:Intro}

In this chapter, we will briefly address the basic concepts of the theory of discrete dynamical systems, so as the concept of chaos that can occur in these systems, as much as necessary to define the relevant problem and the goals of this thesis. Important references in this subject are \cite{DevaneyIntroCDD,Holmgren,Elaydi}, among others, and all of the proofs omitted here can be found therein. This chapter is included with the aim of making the work self-contained, it includes only what is now the standard basic theory of discrete dynamical systems, and may be skipped if the reader is familiar with it.

This chapter has the following objectives: (\emph{i}) familiarize with the fundamentals of the theory of discrete dynamical systems, particularly chaotic ones; (\emph{ii}) show the importance of the parametric dependency in families of functions, particularly bifurcation theory; and (\emph{iii}) present the Feigenbaum sequence and its universality property for unimodal maps as a basis to understand how chaos may arise through period doubling bifurcations.

\section{Elementary concepts}

The main objective of the theory of \emph{discrete dynamical systems} is to understand the \emph{final} or \emph{asymptotic} behavior of an iterative process. If the process is continuous, a differential equation with time being the independent variable, then the theory intends to predict the ultimate behavior of the solutions in the distant future ($t\rightarrow\infty$) or the distant past ($t\rightarrow-\infty$). If the process is discrete, such as the iteration of a real (or complex) function $f$, then the theory hopes to be able to understand the eventual behavior of the points  $x,f(x),f^{2}(x),...,f^{n}(x)$ when $n\in\mathbb{N}$ is large, where $f^{n}(x)$ denotes the $n$-th iteration (i.e., composition of the function with itself). That is, the theory of dynamical systems asks the question: where do points go and what do they do when they get there? In this work, we aim to answer a part of this question for one of the simplest kinds of dynamical systems: functions of a single real variable; in particular, for real polynomial functions, having as a precedent the work of \cite{Solis2004} for quadratic maps.

In the following, let $A\subseteq\mathbb{R}$, $x\in A$ and $f:A\rightarrow\mathbb{R}$.

\begin{defn}[Orbits]
	The \emph{forward orbit} of $x$ is the set of points $x,\, f(x),\, f^{2}(x),\,...$ and is denoted by $O^{+}(x)$. If $f$ is a homeomorphism, we can define the \emph{complete orbit} of $x$, $O(x)$, as the set of points $f^{n}(x)$ for $n\in\mathbb{Z}$, and the \emph{backward orbit} of $x$, $O^{-}(x)$, as the set of points $x,\, f^{-1}(x),\, f^{-2}(x),\,...$.
\end{defn}

Therefore, we can restate the basic goal of discrete dynamical systems as to know and understand all orbits of a map. Orbits can be quite complicated sets, even for simple non-linear maps. However, there are some specially simple orbits which play a central role in the dynamics of the whole system.

\begin{defn}[Periodic Points]
	The point $x$ is a fixed point of $f$ if $f(x)=x$. The point $x$ is a \emph{periodic point of period $n$} if $f^{n}(x)=x$. The smallest positive $n$ for which $f^{n}(x)=x$ is called the \emph{prime period} of $x$. We denote the set of periodic points of period $n$ (not necessarily prime) by $\mathrm{Per}_{n}(f)$, and the set of fixed points by $\mathrm{Fix}(f)$. The set of all iterates of a periodic point forms a \emph{periodic orbit}.
\end{defn}

An important question is the ``range of influence'' of a fixed point. That is to ask, what set of points, if any, will approach or tend to a fixed (or periodic) point under the iteration of $f$.

\begin{defn}[Stable Set or Basin of Attraction]
	Let $p$ be periodic of period $n$. A point $x$ is \emph{forward asymptotic to $p$} if $\lim_{i\rightarrow\infty}f^{in}(x)=p$. The \emph{stable set (or basin of attraction)} of $p$, denoted by $W^{s}(p)$, consists of all points forward asymptotic to $p$.
\end{defn}

If $p$ is not periodic, we can still define its forward asymptotic points requiring that $\vert f^{i}(x)-f^{i}(p)\vert\rightarrow0$ when $i\rightarrow\infty$. Moreover, if $f$ is invertible, we can consider \emph{backward asymptotic points} by taking $i\rightarrow-\infty$ in the last definition. The set of backward asymptotic points to $p$ is called the \emph{unstable set} of $p$ and is denoted by $W^{u}(p)$.

Another useful concept comes from elementary calculus.

\begin{defn}[Critical Points]
	A point $x$ is a \emph{critical point} of $f$ if $f'(x)=0$. The critical point is \emph{non degenerate} if $f''(x)\neq0$. The critical point is \emph{degenerate} if $f''(x)=0$.
\end{defn}

Critical points will prove later to be important in relation to the stable set of a fixed point, and to determine how many attracting orbits a map may have.

Degenerate critical points can be either \emph{maxima}, \emph{minima} or \emph{saddle points}. Non degenerate critical points can only be maxima or minima. A \emph{diffeomorphism} cannot have critical points, but its existence in non-invertible maps is one of the main reasons why this type of maps are more complicated.

The goal of the theory of discrete dynamical systems is then to understand the nature of all orbits and to identify which are periodic, asymptotic, etc. In general, this is an impossible task. For example, if $f(x)$ is a quadratic polynomial, then explicitly finding the periodic points of period $n$ is equivalent to solving the equation $f^{n}(x)=x$, which is a polynomial equation of degree $2^{n}$. Computer calculations do not help in this case either, since rounding errors tend to accumulate and make many periodic points ``invisible'' to the computer. Therefore, qualitative and geometrical techniques must be employed to understand the nature of a system.

\section{Hyperbolicity and criteria for stability}

A key issue once we find periodic points, is to determine the behavior of orbits near these points; such investigation is what is called \emph{stability theory}. Once this is done, we can state whether a periodic point is \emph{attracting} or \emph{repelling} (or neither). To determine this, we can part from the following

\begin{defn}[Stability]
	Let $p\in A$ be a fixed point of $f$. Then,
	\begin{enumerate}
		\item $p$ is said to be \emph{stable} if for any $\varepsilon>0$ there exists a $\delta>0$ such that if $\vert x_0-p\vert<\delta$ then $\vert f^n(x_0)-p\vert<\varepsilon$ for all integers $n\in\mathbb{Z}^+$. Otherwise, the fixed point $p$ is called \emph{unstable}.
		\item $p$ is called \emph{attracting} if there is a $\eta>0$ such that if $\vert x_0-p\vert<\eta$ then $\lim_{n\rightarrow\infty}f^n(x_0)=p$.
		\item $p$ is \emph{asymptotically stable} if it is both stable and attracting.
	\end{enumerate}
\end{defn}

Therefore, intuitively, a fixed point is stable if nearby orbits remain near the fixed point, and you can always find an interval around the fixed point such that orbits starting within the interval will remain arbitrarily close to the fixed point. The opposite is that you cannot do that; i.e. there is an interval around the fixed point such that no matter how close to the fixed point you start the orbit from, the orbit will go outside the interval. The case of an asymptotically fixed point is particularly important, since in this case all the orbits near a fixed point will approach it in the limit. The next step is to be able to determine when a fixed point is stable or unstable. To do this, we will need the

\begin{defn}[Hyperbolic Point]
	\label{def:hyperbolic-fp}
	Let $p$ be a periodic point of period $n$. The point $p$ is \emph{hyperbolic} if $\vert(f^{n})'(p)\vert\neq1$. Otherwise, it is called a \emph{nonhyperbolic} fixed point. The number $(f^{n})'(p)$ is called the \emph{multiplier} of the periodic point.
\end{defn}

\begin{example}
	Consider the diffeomorphism $f(x)=\frac{1}{2}(x^{3}+x)$. There are three fixed points: $x=0,\,1,$ and -1. Note that $f'(0)=1/2$ and $f'(\pm1)=2$. Therefore, each point is hyperbolic.
\end{example}

As it turns out, hyperbolic points are easy to understand in terms of whether they are attracting or \emph{repelling}, as they are always one of the two.

\begin{defn}[Attractor and repellor]
	\label{def:attractor-repellor}
	Let $p$ be a hyperbolic point of period $n$. If $\vert(f^{n})'(p)\vert<1$, the point $p$ is asymptotically stable and is called an attracting periodic point or more shortly, an \emph{attractor}; if $\vert(f^{n})'(p)\vert>1$, $p$ is unstable and it is called a repelling periodic point or, simply, a \emph{repellor}.
\end{defn}

Occasionally, the terms sink and source may also be used to refer to attractors and repellors, respectively, stemming from the continuous dynamical systems analog. Now, if a periodic point is attracting, it must have an interval around it in which it is so; this is stated in the following

\begin{thm}\label{thm:attractor-repellor}
	Let $p$ be a hyperbolic fixed point of $f$. Then if $p$ is an attractor, there exists an open set $U\subseteq W^{s}(p)$, with $p\in U$; on the other hand, if $p$ is a repellor, there is an open set $V\subset W^{u}(p)$, with $p\in V$.
\end{thm}

This theorem explains the choice of the terms ``attracting'' an ``repelling'', since attractors and repellors have around them a basin of attraction and an unstable set, respectively.

It is straightforward to see that $W^s(p)$ is an invariant set under the action of the map. A similar result is true for periodic points of period $n$. In this case, we have an open interval around $p$ that is mapped into itself by $f^{n}$.

\section{The Schwarzian derivative}
\label{sec:SchwarzianDerivative}

Here, we will discuss a mathematical tool that will turn out useful in determining some properties of maps in terms of its periodic points structure.

\begin{defn}[Schwarzian Derivative]
	The \emph{Schwarzian derivative} of a function $f$ at the point $x$ is given by
	\[Sf(x)=\frac{f'''(x)}{f'(x)}-\frac{3}{2}\left(\frac{f''(x)}{f'(x)}\right)^{2}.\]
\end{defn}

In particular, it was first used by \cite{Singer} in \cite{Singer} to address the question of how many attracting periodic points a differentiable map can have \cite{Singer} (see \emph{Singer's Theorem} below). It can also be used to determine the nature of non-hyperbolic periodic points \cite{Elaydi}. Mappings with negative Schwarzian derivative present particular dynamical properties that will interest us in this work.

\begin{thm}
	Let $P(x)$ be a real polynomial. If all roots of $P'(x)$ are real and distinct, then $SP(x)<0$ for all $x$.
\end{thm}

The following theorem is useful when analyzing \emph{topologically conjugate maps} (a concept to be defined further below in \ref{sec:Conjugacy}), which in turn is useful to determine the stability properties of one map in terms of the known such properties of another map (the \emph{conjugate map}).

\begin{thm}
	Suppose $Sf<0$ and $Sg<0$. Then $S(f\circ g)<0$.
\end{thm}

Also, this tells us that iteration of a map does not change the sign of its Schwarzian derivative, i.e. the property is preserved. This immediately takes us to

\begin{cor}
	Let $Sf<0$. Then $Sf^{n}<0$ for all $n>1$.
\end{cor}

The assumption of $Sf<0$ has several surprising implications for the dynamics of a one-dimensional map, as we will see below. Another useful result dealing with the geometry of a map is

\begin{lem}
	\label{lem:maxmin}
	If $Sf<0$, then $f'(x)$ cannot have a positive local minimum or a negative local maximum.
\end{lem}

The proof of this last lemma is straightforward from the definition of the Schwarzian derivative and elementary calculus. Note that this lemma tells us that, if $Sf<0$, between any two consecutive critical points of $f'$, its graph must cross the $x$-axis; this in turn means that there must be a critical point for $f$ between these two points, i.e. there must be a maximum or minimum of $f$ between any pair of its inflection points.

The Schwarzian derivative can also be used to determine whether a fixed point attracts a critical point of a map.

\begin{lem}
	Let $a_1, a_2$ and $a_3$ be fixed points of a continuously differentiable map $g$ with $a_1<a_2<a_3$ and such that $Sg<0$ on the open interval $(a_1,\,a_3)$. If $g'(a_2)\leq 1$, then $g$ has a critical point in $(a_1,\,a_3)$.
\end{lem}

The last lemma is used in the proof of the following important result due to \cite{Singer} (\cite{Singer}).

\begin{thm}[Singer's Theorem]
	\label{thm:Singer}
	Suppose $Sf<0$ ($Sf(x)=-\infty$ is allowed). Suppose $f$ has $n$ critical points. Then $f$ has at most $n+2$ attracting periodic orbits.
\end{thm}

This last theorem allows to find directly from $f$ an upper bound for the number of attracting periodic points. This differs from the information provided in the analysis by \emph{Sarkovskii's theorem} (see section \ref{sec:Sarkovskii} below) in that the latter tells us how many periodic points there are, regardless of whether they are attracting or not and, also, we have first to find some periodic point and its period to be able to apply Sarkovskii's theorem.

\section{Nonhyperbolic fixed points}
\label{sec:nonhyperbolic}

The stability criteria for nonhyperbolic fixed points involve the Schwarzian derivative defined in \ref{sec:SchwarzianDerivative}. We will analyze the cases of the multiplier being equal to 1 and -1 separately. The unstated proofs of the following theorems can be found in the book by \cite{Elaydi}.

\begin{thm}
	\label{thm:nonhyperbolic-pos}
	Let $p$ be a fixed point of a map $f$ such that $f'(p)=1$. Then, if $f'''(p)\neq 0$ and is continuous, the following statements hold
	\begin{enumerate}
		\item If $f''(p)\neq 0$, then $p$ is unstable.
		\item If $f''(p)=0$ and $f'''(p)>0$, then $p$ is unstable.
		\item If $f''(p)=0$ and $f'''(p)<0$, then $p$ is asymptotically stable.
	\end{enumerate}
\end{thm}

Similarly, we have

\begin{thm}
	\label{thm:nonhyperbolic-neg}
	Let $p$ be fixed point of a map $f$ such that $f'(p)=-1$. If $f'''(p)$ is continuous, then the following statements hold:
	\begin{enumerate}
		\item If $Sf(p)<0$, then $p$ is an asymptotically stable.
		\item If $Sf(p)>0$, then $p$ is unstable.
	\end{enumerate}
\end{thm}

It is worth noting that a repellor is an unstable periodic point but not all such points are repellors since, a nonhyperbolic periodic point may also be unstable; e.g. a ``semistable'' nonhyperbolic fixed point is ``attracting from the right (or left)'' but ``repelling from the left (or right, respectively)''.

When dealing with nonhyperbolic points, particularly when they are unstable, it is useful to account for the concept of ``semistability'', separating stability ``from the right'' and ``from the left''.

\begin{defn} [Semistability]
	\label{def:semistability}
	A fixed point $p$ of a map $f$ is said to be semistable from the right (respectively, from the left) if for any $\varepsilon>0$ there exists $\delta>0$ such that if $0<x_0-p<\delta$ (respectively, $0<p-x_0<\delta$) then $\vert f^n(x_0)-p\vert<\varepsilon$ for all $n\in\mathbb{Z}^+$. Moreover, if $lim_{n\rightarrow\infty}f^n(x_0)=p$ whenever $0<x_0-p<\eta$ (respectively, $0<p-x_0<\eta$) for some $\eta>0$, then $p$ is said to be semiasymptotically stable from the right (respectively, from the left).
\end{defn}

It is then easy to prove that

\begin{thm} [Asymptotical semistability]
	\label{thm:semistability}
	Let $f$ be a map, $p$ one of its fixed points and with $f'(p)=1$ and $f''(p)\neq0$. Then $p$ is
	\begin{enumerate}
		\item Semiasymptotically stable from the right if $f''(p)<0$;
		\item Semiasymptotically stable from the left if $f''(p)>0$.
	\end{enumerate}
\end{thm}

\section{Sarkovskii's theorem}
\label{sec:Sarkovskii}

This theorem is of great beauty, as it constructs a new --apparently artificial-- ordering of the natural numbers in order to determine what types of orbits may a map have. Consider then the following ordering of the natural numbers:

\[3\triangleright5\triangleright7\triangleright\cdots\triangleright2\cdot3\triangleright2\cdot5\triangleright2\cdot7\triangleright\cdots\triangleright2^{2}\cdot3\triangleright2^{2}\cdot5\triangleright\cdots\]

\[\triangleright2^{3}\cdot3\triangleright2^{3}\cdot5\triangleright\cdots\cdots\triangleright2^{3}\triangleright2^{2}\triangleright2\triangleright1.\]

That is: first all odd numbers, then two times each odd number, $2^{2}$ times each odd number, $2^{3}$ times each odd number, etc., and finally, the powers of two in descending order. This is \emph{Sarkovskii's ordering} of the natural numbers and it easy to see it is exhaustive .

Having defined the above, we can state

\begin{thm}[Sarkovskii]
	Let $f:\mathbb{R}\rightarrow\mathbb{R}$ continuous. Suppose that $f$ has a periodic point of prime period $k$. If $k\triangleright l$ in Sarkovskii's ordering, then $f$ also has a periodic point of period $l$.
\end{thm}

Thus, Sarkovskii's theorem gives a nearly complete answer to the question of how many periodic points and of which periods can a map have, although it does not tell us their nature in terms of their stability, nor does it tell us how to find such points.

Some important consequences of Sarkovskii's Theorem can be summarized as follows:
\begin{itemize}
	\item If $f$ has periodic points with periods different from a power of 2, then $f$ has infinite periodic points of infinitely different periods.
	\item If $f$ only has a finite number of periodic points, then necessarily all points have periods of powers of 2.
	\item Period 3 is the ``greatest'' period in Sarkovskii's ordering, therefore, it implies the existence of all other periods.
\end{itemize}

Sarkovskii's theorem is without doubt a central result for the theory of discrete dynamical systems.

\section{Bifurcation theory}

We will now consider families of maps instead of single maps, as this is often useful in applications and is also the main object of study of this work. For example, in practice, the coefficients of a polynomial map that models some natural, physical, economic or social discrete process can only be determined up to a certain precision, which gives rise to a family of maps as a set bounded by the uncertainties of the coefficients.

The aim of bifurcation theory is then to study the changes that occur in a map when its parameters are varied. This changes commonly refer qualitatively to its periodic points structure and stability properties, but they can also involve other changes. Consider families of functions in one real variable which depend upon a single real parameter (uniparametric). Consider the bivariate function $G(x,\lambda)=f_{\lambda}(x)$, where for fixed $\lambda\in\mathbb{R}$, $f_{\lambda}(x)\in C^{\infty}$. We will assume also that $G$ depends smoothly on $\lambda$. Some examples of such families of functions are:

\begin{itemize}
	\item $F_{\mu}(x)=\mu x(1-x)$.
	\item $E_{\lambda}(x)=\lambda e^{x}$.
	\item $S_{\lambda}(x)=\lambda\sin(x)$.
	\item $Q_{c}(x)=x^{2}+c$.
\end{itemize}

With this in mind, we will next give a rather informal definition of bifurcation.

\begin{defn}[Bifurcation]
	Let $f_{\lambda}(x)$ be a uniparametric family of functions. Then, we say there is a \emph{bifurcation} at $\lambda_{0}$ if there exists $\epsilon>0$ such that if $a$ and $b$ satisfy $\lambda_{0}-\epsilon<a<\lambda_{0}$ and $\lambda_{0}<b<\lambda_{0}+\epsilon$, then the ``dynamics'' of $f_{a}(x)$ is different from the ``dynamics'' of $f_{b}(x)$. That is, the ``dynamics'' of the function $f_{\lambda} (x)$ \emph{changes} when $\lambda$ passes through the value $\lambda_{0}$.
\end{defn}

The ``changes in the dynamics'' of the above definition refer to qualitative changes in the structure of periodic points of $f_{\lambda}(x)$. To make the above notion precise, we must consider separately the different types of bifurcations that may occur. For real maps, the \emph{saddle-node} and \emph{period doubling} bifurcations are the most common. We will next discuss the concept of saddle-node bifurcation through an example and the period doubling bifurcation is discussed in more detail below.

\begin{example}[Saddle-node or Tangent Bifurcation]
	Let's consider the family of functions $E_{\lambda}(x)=\lambda e^{x}$, with $\lambda>0$. A bifurcation occurs when $\lambda=1/e$. We have three cases:
	\begin{enumerate}
		\item If $\lambda>1/e$ the graph of $E_{\lambda}(x)$ and $y=x$ do not intersect each other, so that $E_{\lambda}$ has no fixed points. $E_{\lambda}^{n}(x)\rightarrow\infty\,\forall x$.
		\item When $\lambda=1/e$, the graph of $E_{\lambda}$ tangentially touches the identity diagonal at $(1,\,1)$. $E_{\lambda}(1)=1$ is the only fixed point. If $x<1$, $E_{\lambda}^{n}(x)\rightarrow1$ and if $x>1$, $E_{\lambda}^{n}(x)\rightarrow\infty$.
		\item When $0<\lambda<1/e$, the graph of $E_{\lambda}$ crosses the identity diagonal two times, at $q$ with $E_{\lambda}'(q)<1$ and at $p$ with $E_{\lambda}'(p)>1$, with which we see they are an attractor and a repellor, respectively.
	\end{enumerate}
	Figure \ref{fig:Saddle-Node} presents some of these characteristics graphically, and the corresponding phase portraits are presented in figure \ref{fig:Saddle-Node-PhaseDiag}. Figure \ref{fig:Bifurcation-Diagram-Example} presents a common and very useful way of depicting bifurcations of real maps, which will be frequently used in this work: a \emph{bifurcation diagram}; in it, the horizontal axis represents the values of the parameter $\lambda$ in which the family is varied, and the vertical axis represents the values of the periodic points present for each $f_{\lambda}$.
\end{example}

\begin{figure}
	\centering
	\includegraphics[width=0.5\columnwidth]{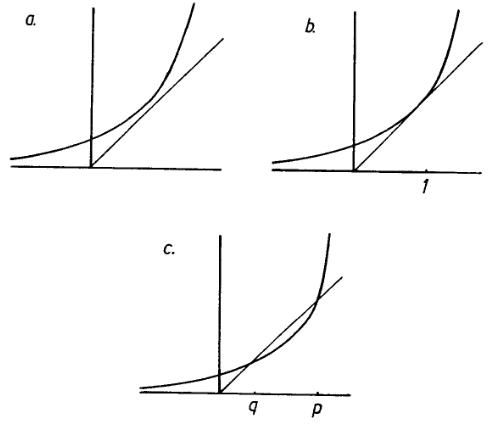}
	\caption{Graphs of $E_{\lambda}(x)=\lambda e^{x}$ where (a) $\lambda>1/e$, (b) $\lambda=1/e$, and (c) $0<\lambda<1/e$. Reproduced from \cite{DevaneyIntroCDD}.}
	\label{fig:Saddle-Node}
\end{figure}

\begin{figure}
	\centering
	\includegraphics[width=0.5\columnwidth]{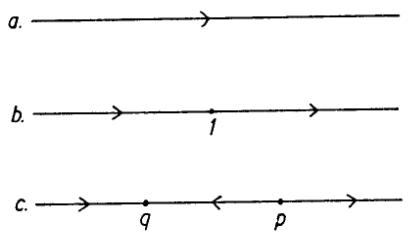}
	\caption{Phase portraits of $E_{\lambda}(x)=\lambda e^{x}$ for (a) $\lambda>1/e$, (b) $\lambda=1/e$, and (c) $0<\lambda<1/e$. Reproduced from \cite{DevaneyIntroCDD}.}
	\label{fig:Saddle-Node-PhaseDiag}
\end{figure}

\begin{figure}
	\centering
	\includegraphics[width=0.33\columnwidth]{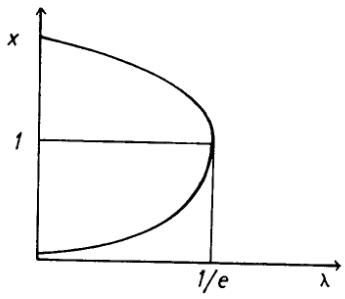}
	\caption{Bifurcation diagram of $E_{\lambda}(x)=\lambda e^{x}$. $x$ is plotted against $\lambda$. Reproduced from \cite{DevaneyIntroCDD}.}
	\label{fig:Bifurcation-Diagram-Example}
\end{figure}

The above example suggests that bifurcations may occur when non-hyperbolic fixed points are present. Indeed, this the only case when bifurcations of fixed points occur (in the sense of one fixed point giving birth to two others), as the next theorem demonstrates.

\begin{thm}
	\label{thm:fixedp}
	Let $f_{\lambda}$ be a one-parameter family of functions and suppose that $f_{\lambda_0}(x_0)=x_0$ and $f'_{\lambda_0}(x_0)\neq 1$. Then there are intervals $I$ about $x_0$ and $N$ about $\lambda_0$ and a smooth function $p:N\rightarrow I$ such that $p(\lambda_0)=x_0$ and $f_{\lambda}(p(\lambda))=p(\lambda)$. Moreover, $f_{\lambda}$ has no other fixed points in $I$.
\end{thm}
\begin{proof}
	Consider the function defined by $G(\lambda,x)=f_{\lambda}(x)-x$. By hypothesis, $G(\lambda_0,x_0)=0$ and
	\begin{equation}
		\frac{\partial G}{\partial x}(\lambda_0,x_0)=f'_{\lambda_0}(x_0)-1\neq 0,
	\end{equation}
	using the other hypothesis. Then, by the Implicit Function Theorem \cite{Marsden}, there are intervals $I$ about $x_0$ and $N$ about $\lambda_0$, and a smooth function $p:N\rightarrow I$ such that $p(\lambda_0)=x_0$ and $G(\lambda,p(\lambda))=0$ for all $\lambda\in N$; that is, $p(\lambda)$ is fixed by $f_{\lambda}$ for $\lambda\in N$. Moreover, $G(\lambda,x)\neq 0$ unless $x=p(\lambda)$.
\end{proof}

The above theorem obviously holds for periodic points by replacing $f_{\lambda}$ with $f^n_{\lambda}$.

\section{Chaos}
\label{sec:Chaos}

In this section we will precise what we understand by the mathematical concept of \emph{chaos}. It must be said though, that there is not a single universally accepted definition for this concept, as variations occur; however, they preserve the same general ``spirit''. We will here adopt one of the most popular definitions, due to \cite{DevaneyIntroCDD}.

Before we can properly define the concept of chaos, we need some other auxiliary concepts.

\begin{defn}[Topological transitivity]
	Let $(X,\,d)$ be a metric space and $f:\,X\rightarrow X$ a map. Then $f$ is said to be \emph{topologically transitive} if for any pair of nonempty open sets $U,\,V\subset X$ there exists $k>0$ such that $f^k(U)\cap V\neq \emptyset$.
\end{defn}

In plain words, what this definition states is that a map $f$ is topologically transitive if any arbitrarily small set of the domain of $f$ contains points that are mapped by $f$ to any other arbitrarily small set of the codomain in a finite number of iterations. An immediate consequence of topological transitivity is that the domain cannot be divided into two nonempty disjoint sets which are invariant under $f$. It is easy to see that if $f$ has a dense orbit in $X$ then it is topologically transitive (the converse is also true for compact subsets of $\mathbb{R}$ or $S^1$).

Another necessary concept for the definition of chaos is the following.

\begin{defn}[Sensitive Dependence to Initial Conditions]
	Let $(X,\,d)$ be a metric space and $f:\,X\rightarrow X$ a map. $f$ is said to posses \emph{sensitive dependence to initial conditions} if there exists $\varepsilon >0$ such that for any $x\in X$ and any neighborhood $N$ of $x$, there exists $y\in N$ and $n\geq 0$ such that $\vert f^n(x)-f^n(y)\vert > \varepsilon$.
\end{defn}

Intuitively, sensitive dependence to initial conditions means that no matter how close two points are to each other, they will eventually be separated by at least some distance $\varepsilon$ under a finite number of iterations of $f$.

With these preliminary definitions we can now state the main definition of this section.

\begin{defn}
	Let $(X,\,d)$ be a metric space and $f:\,X\rightarrow X$ a map. Then $f$ is said to be \emph{chaotic} on $X$ if
	\begin{enumerate}
		\item $f$ is topologically transitive,
		\item the set of periodic points of $f$ is dense in $X$, and
		\item $f$ has sensitive dependence on initial conditions.
	\end{enumerate}
\end{defn}

So a chaotic map ``mixes up'' its domain set by means of its topological transitivity but everywhere you look there are periodic points arbitrarily close and, finally, it is practically impossible to predict orbits numerically; these are the three elements that make up a chaotic map, under this definition. Actually, it was later shown \cite{Banks92} that topological transitivity and denseness of periodic points imply sensitive dependence to initial conditions, but no other pair of conditions imply the other. However, for continuous maps on intervals of $\mathbb{R}$, topological transitivity \emph{implies} that the set of periodic points is dense \cite{VellekoopBerglund}, so that in this case, by the above, topological transitivity alone actually implies chaos.

\section{The period doubling route to chaos}
\label{sub:PeriodDoubling}

Perhaps one of the most important types of bifurcations is the \emph{period doubling bifurcation}. This type of bifurcation will play a central role in the main results of this work. The most commonly know map which presents period doubling bifurcation is the famous logistic map. The issue is that the \emph{logistic quadratic map} $f_{\lambda}(x)=\lambda x(1-x)$ has simple dynamics for $0\leq\lambda\leq3$ but it becomes chaotic when $\lambda\geq4$ (we will precise the notion of chaos in the next section). The natural question here is: how exactly is it that $f_{\lambda}$ becomes chaotic as $\lambda$ increases? Where do the infinitely many periodic points come from when $\lambda$ is large? In this section we will see that ``period doubling'' bifurcation is central to this.

\emph{Sarkovskii's theorem} provides a partial answer to the question of how many periodic points emerge when the parameter is varied. Before $f_{\lambda}$ can have infinite periodic points of distinct periods, it must first have periodic points of all the periods of the form $2^{j}$ with $j\in\mathbb{N}$. Local bifurcation theory provides two typical ways in which these points may emerge: by means of saddle-node bifurcations or through period doublings. The question is then which type of bifurcations occur as $f_{\lambda}$ becomes chaotic.

The typical way in which $f_{\lambda}$ becomes chaotic is that $f_{\lambda}$ undergoes a series of  period doubling bifurcations. This is not always the case, but it is a typical ``route to chaos''.

\begin{thm}[Period Doubling Bifurcation]
	\label{thm:perioddoubling}
	Suppose that
	\begin{enumerate}
		\item $f_{\lambda}(x_{0})=x_{0}$ for all $\lambda$ in an interval around $\lambda_{0}$.
		\item $f_{\lambda}'(x_{0})=-1$.
		\item $\frac{\partial(f_{\lambda}^{2})'}{\partial\lambda}\vert_{\lambda=\lambda_{0}}(x_0)\neq0.$
	\end{enumerate}
	Then there exists an interval $I$ around $x_{0}$ and a function $p:I\rightarrow\mathbb{R}$ such that
	\[f_{p(x)}(x)\neq x\]
	but
	\[f_{p(x)}^{2}(x)=x.\]
\end{thm}
\begin{proof}
	Let $G(\lambda,\, x)=f^2_{\lambda}(x)-x$. Then, from hypotheses (1) and (2),
	\[\frac{\partial G}{\partial\lambda}(\lambda_0,\,x_0)=0\]
	so that we cannot apply the implicit function theorem directly. Thus, to rectify the situation, we set
	\begin{equation}
		B(\lambda,\,x)=
		\begin{cases}
			G(\lambda,\,x)/(x-x_0) & \mathrm{if}\,x\neq x_0,\\
			\frac{\partial G}{\partial x}(\lambda_0,\,x_0) & \mathrm{if}\,x=x_0.
		\end{cases}
	\end{equation}
	
	Straightforward calculation of the limits of the derivatives allow us to verify that $B$ is smooth and satisfies
	\begin{equation}
		\begin{aligned}
			\frac{\partial B}{\partial x}(\lambda_0,\,x_0)		& =\frac{1}{2}\frac{\partial^2 G}{\partial x^2}(\lambda_0,x_0)\\
			\frac{\partial^2 B}{\partial x^2}(\lambda_0,\,x_0)	& =\frac{1}{3}\frac{\partial^3 G}{\partial x^3}(\lambda_0,x_0).
		\end{aligned}
	\end{equation}
	
	We now observe that, once again using hypotheses (1) and (2),
	\begin{equation}
		\label{eq:Bl0x0}
		\begin{aligned}
			B(\lambda_0,\,x_0) 	& =\frac{\partial G}{\partial x}(\lambda_0,\,x_0)\\
			& =\left.\frac{\partial}{\partial x}f^2_{\lambda_0}(x)\right|_{x=x_0}-1\\
			& =\left.\frac{\partial}{\partial x}f_{\lambda_0}\left(f_{\lambda_0}(x)\right)\right|_{x=x_0}-1\\
			& =\left[\left.\frac{\partial}{\partial y}f_{\lambda_0}(y)\right|_{y=f_{\lambda_0}(x)}\cdot\frac{\partial}{\partial x}f_{\lambda_0}(x)\right]_{x=x_0}-1\\
			& =\frac{\partial f_{\lambda_0}}{\partial x}(f_{\lambda_0}(x_0))\frac{\partial f_{\lambda_0}}{\partial x}(x_0)-1\\
			& =\left[\frac{\partial f_{\lambda_0}}{\partial x}(x_0)\right]^2-1\\
			& =\left[f_{\lambda_0}'(x_0)\right]^2-1\\
			& =(-1)^2-1\\
			& =0.
		\end{aligned}
	\end{equation}
	
	Moreover, now using hypothesis (3),
	\begin{equation}
		\begin{aligned}
			\frac{\partial B}{\partial \lambda}(\lambda_0,\,x_0)	& =\frac{\partial}{\partial\lambda}\left(\frac{\partial G}{\partial x}(\lambda_0,\,x_0)\right)\\
			& =\frac{\partial}{\partial\lambda}\left(\frac{\partial}{\partial x}f^2_{\lambda_0}(x_0)-1\right)\\
			& =\frac{\partial^2}{\partial\lambda\partial x}f^2_{\lambda_0}(x_0)\\
			& =\frac{\partial}{\partial\lambda}(f^2_{\lambda_0})'(x_0)\\
			& \neq 0.
		\end{aligned}
	\end{equation}
	
	Therefore, by the implicit function theorem \cite{Marsden}, there exists a $C^1$ map $p(x)$ defined on an interval around $x_0$ such that $p(x_0)=\lambda_0$ and $B(p(x),x)=0$. So that, for $x\neq x_0$,
	\[\frac{1}{x-x_0}G(p(x),x)=0.\]
	
	Consequently, $G(p(x),x)=0$ and then $f^2_{p(x)}(x)=x$ and thus $x$ is of period 2 for $\lambda=p(x)$ and for $x$ in an interval around $x_0$. Notice that $x$ is not fixed by $f_{p(x)}$ because we can apply theorem \ref{thm:fixedp} in this case and, then, the curve of fixed points is unique.
	
	On the other hand, from the chain rule we have
	\[\frac{d}{dx}B(p(x),x)=\frac{\partial B}{\partial x}+\frac{\partial B}{\partial\lambda}p'(x)=0.\]
	
	From which,
	\begin{equation}
		\label{eq:dp}
		p'(x)=-\frac{\frac{\partial B}{\partial x}(p(x),x)}{\frac{\partial B}{\partial\lambda}(p(x),x}.
	\end{equation}
	
	We know that, for $x=x_0$, the denominator in \eqref{eq:dp} is different from zero and, on the other hand, the numerator is equal to $\frac{1}{2}\frac{\partial^2 G}{\partial x^2}$ and
	\begin{equation}
		\begin{aligned}
			\frac{\partial^2 G}{\partial x^2}(\lambda_0,x_0)	& =\frac{1}{2}\left.\left\{f''_{\lambda}(f_{\lambda}(x))[f'_{\lambda}(x)]^2+f'_{\lambda}(f_{\lambda}(x))f''_{\lambda}(x)\right\}\right|_{x=x_0,\,\lambda=\lambda_0}\\
			& =\frac{1}{2}\left\{f''_{\lambda_0}(x_0)[-1]^2-f''_{\lambda_0}(x_0)\right\}\\
			& =0.
		\end{aligned}
	\end{equation}
	Where we have once again used hypothesis (2). Ergo, $p'(x_0)=0$.
\end{proof}

Also, differentiating equation \eqref{eq:dp} again and evaluating in $x=x_0$, it is easy to see that
\begin{equation}
	\label{eq:d2p}
	p''(x_0)=-\frac{\frac{\partial^2 B}{\partial x^2}(\lambda_0,x_0)}{\frac{\partial B}{\partial\lambda}(\lambda_0,x_0)}.
\end{equation}

We already know that the denominator in \eqref{eq:d2p} above is not null and, also,
\begin{equation}
	\begin{aligned}
		\frac{\partial^2 B}{\partial x^2}(\lambda_0,x_0)	& =\frac{1}{3}\frac{\partial^3 G}{\partial x^3}\\
		& =\frac{1}{3}\left\{f'''_{\lambda}(f_{\lambda}(x))\left[f'_{\lambda}(x)\right]^3+2f''_{\lambda}(f_{\lambda}(x))f'_{\lambda}(x)+f''_{\lambda}(f_{\lambda}(x))f''_{\lambda}(x)f'_{\lambda}(x)\right.\\
		& \quad\left.\left.+f'_{\lambda}(f_{\lambda}(x))f'''_{\lambda}(x)\right\}\right|_{x=x_0,\,\lambda=\lambda_0}\\
		& =\frac{1}{3}\left\{-2f'''_{\lambda_0}(x_0)-3[f''_{\lambda_0}(x_0)]^2\right\}\\
		& =\frac{2}{3}\left\{\frac{f'''_{\lambda_0}(x_0)}{f'_{\lambda_0}(x_0)}-\frac{3}{2}\left[\frac{f''_{\lambda_0}(x_0)}{f'_{\lambda_0}(x_0)}\right]^2\right\}\\
		& =\frac{2}{3}\,Sf_{\lambda_0}(x_0),
	\end{aligned}
\end{equation}

where $Sf_{\lambda}$ is the Schwarzian derivative of $f_{\lambda}$ and, yet once again, we used hypothesis (2) from the theorem above. Therefore, substituting in equation \eqref{eq:d2p} we have
\begin{equation}
	p''(x_0)=-\frac{-\frac{2}{3}\,Sf_{\lambda_0}(x_0)}{\frac{\partial}{\partial\lambda}(f^2_{\lambda_0})'(x_0)}.
\end{equation}

So that, if $Sf_{\lambda}\neq0$ then $p''(x_0)\neq0$ and the curve $p(x)$ is either concave left or right. Along with $p(x_0)=0$ this gives us an idea of the geometry of the period 2 curve $p$ near the bifurcation point $x_0$. Notice that $x_0$ does not bifurcate in the sense of giving birth to another fixed point (as was the case in the pitchfork bifurcation) but, rather, $x_0$ continues to be the only fixed point in a interval around $\lambda_0$ (see theorem \ref{thm:fixedp}).

In the case of the logistic map, the period doubling bifurcation repeats over and over again for the fixed point $x^{*}=\frac{\lambda-1}{\lambda}$ as the value of the parameter $\lambda$ is increased, starting from $\lambda=3$:
\begin{itemize}
	\item Between $\lambda_{0}=1<\lambda<3=\lambda_{1}$, $f_{\lambda}$ has a single stable fixed point, $x^{*}$ (of period $2^{0}=1$).
	\item For $\lambda_{1}=3<\lambda<1+\sqrt{6}=\lambda_{2}$, $x^{*}$ loses its stability and 'gives rise' to an attracting cycle of period two ($2^{1}$), whose orbital points separate from $x^{*}$ as
	$\lambda$ increases.
	\item For $\lambda_{2}=1+\sqrt{6}<\lambda<3.54409=\lambda_{3}$ \cite{Elaydi}, the attracting cycle of period two loses in turn its stability and gives rise to an attracting cycle of period $2^{2}$.
	\item For $\lambda>3.54409=\lambda_{3}$, the period four cycle also loses its stability and bifurcates into an asymptotically stable cycle of period $2^{3}$.
\end{itemize}

The above mentioned process repeats itself indefinitely and produces a sequence $\{\lambda_{k}\}_{k=1}^{\infty}$. The \emph{Feigenbaum sequence} $\{\lambda_{k}\}_{k=0}^{\infty}$ describes the values of the parameter $\lambda$ for which \emph{period doubling bifurcations} arise in a unimodal map \cite{Feigenbaum}. The sequence converges exponentially. The important point to remark here is that the rate of convergence of these sequences (whose values depend on the specific unimodal map), is constant for every unimodal function and is given by $\delta\approx4.669201609\cdots$, which gives rise to the so-called property of \emph{universality}. In fact, its importance deserved naming it \emph{Feigenbaum's constant}, after \cite{Feigenbaum} discovered it in \cite{Feigenbaum} (though amazingly ignored for over a decade).

\section{Topological conjugacy}
\label{sec:Conjugacy}

Here we will define the important concept of \emph{topological conjugacy} which relates the dynamics of two different maps by means of a third one. The importance of conjugacy is that it preserves the stability and chaotic properties of a map, so that one map can be analyzed by means of analyzing another one, presumably simpler, so that the analysis becomes easier.

\begin{defn}[Topological Conjugacy]
	Let $f:A\rightarrow A$ and $g:B\rightarrow B$ be two maps. Then $f$ and $g$ are said to be \emph{topologically conjugate}, denoted $f\approx g$, if there exists a homeomorphism $h:A\rightarrow B$ such that $h\circ g=g\circ h$. The homeomorphism is called a \emph{topological conjugacy}. We also say that $f$ is $h$-conjugate to $g$ to emphasize the importance of the homeomorphism $h$.
\end{defn}

It is easy to prove that conjugacy is an equivalence relation. Topologically conjugate maps have equivalent dynamics. For example, if $f\approx g$ and $x$ is a fixed point of $f$ then $h(x)$ is a fixed point of $g$; and, also, $f^k \approx g^k$ for all $k\in\mathbb{N}$. Similarly, periodic orbits of $f$ are mapped to analogue periodic orbits of $g$.

\begin{thm}
	\label{thm:Conjugacy}
	Let $f:A\rightarrow A$ be $h$-conjugate to $g:B\rightarrow B$. If $f$ is chaotic on $A$, then $g$ is chaotic on $B$.
\end{thm}

Theorem \ref{thm:Conjugacy} will play a key role in the importance of the study of ``canonical polynomial maps'' in chapters \ref{cha:QuadRev}, \ref{cha:Cubic} and \ref{cha:Generalization}. The proof of this theorem can be found in \cite{Elaydi}, where it can be seen that, in fact, the topological conjugacy $h$ does not need to be a homeomorphism for theorem \ref{thm:Conjugacy} to hold, it only needs to be onto, continuous and open, which leads to the concept of

\begin{defn}[Semiconjugacy]
	The maps $f:A\rightarrow A$ and $g:B\rightarrow B$ are said to be \emph{semiconjugate} if there exists a map $h:A\rightarrow B$ onto, continuous and open, such that $h\circ f= g\circ h$.
\end{defn}

That been stated, it can be shown that theorem \ref{thm:Conjugacy} holds with ``conjugate'' replaced by ``semiconjugate''.

\chapter{Regular-Reversal quadratic maps}
\label{cha:Quadratic}

In this chapter the work of \cite{Solis2004} will be discussed, as the main precedent to the present work, where an application of bifurcation theory and chaos in discrete dynamical systems is analyzed to determine the conditions in which chaos arises for a family of unimodal maps, namely quadratic maps. If the asymptotic behavior and parametric dependency of quadratic maps is understood, the results can be generalized to more complex maps, starting to set a precedent for cubic maps and polynomial maps of higher order. The only original contribution in this chapter is given by some of the examples in the corresponding section below.

The main objective of this chapter is to analyze, for a given unimodal quadratic map, under what circumstances does the property of period doubling hold and when it disappears, which means that chaos is no longer a possibility, so as to be able to control the appearance (or not) of chaotic behavior by controlling the value of the parameter.

\section{Unimodal maps}

Consider a one-parameter family of real discrete dynamical systems of the form $x_{n+1}=g(x_{n},\lambda)$ with \emph{iteration function} $g(x_{n},\lambda)=g_{\lambda}(x_{n})$, $g:J\times I\rightarrow\mathbb{R}$, $I$, $J$ intervals, and $g$ smooth on $x$ and $\lambda$.

\begin{defn}[Unimodal Map]\index{unimodal map}
	An \emph{iteration function (map)} $f:I\subset\mathbb{R}\rightarrow I$, with $I$ an interval, is called \emph{unimodal} if it is smooth\footnote{This requirement follows from the computability of the Schwarzian derivative.} (at least $C^{3}(I)$) and with a unique maximal point. By extension, we call a family of maps unimodal if every member of the family is unimodal.
\end{defn}

In the following, let $g(x_{n},\lambda)$ be unimodal and with a \emph{fixed point} $x=x_{p}(\lambda)$.

\begin{defn}[Eigenvalue function]
	\label{def:eigfunction}
	The \emph{eigenvalue function} $\phi:\mathbb{R}\rightarrow\mathbb{R}$ corresponding to the fixed point $x=x_{p}(\lambda)$ of the one-parameter family of maps $g(x_{n},\lambda)$ is
	\begin{equation}
		\label{eq:EigenvalueFunction}
		\phi(\lambda)=\frac{\partial g}{\partial x}(x_{p}(\lambda)),
	\end{equation}
	defined for the values of $\lambda$ for which the fixed point exists.
\end{defn}

That is, the eigenvalue function gives us the multiplier of a fixed point as a function of the parameter of the family of maps. The eigenvalue function for fixed $\lambda$ is simply the derivative of the map at the fixed point, i.e. the ``traditional'' multiplier. Notice that $\phi$ is continuous since $g$ is smooth. The fixed point $x_{p}(\lambda)$ is \emph{asymptotically stable} for the values of $\lambda$ for which $-1<\phi(\lambda)<1$, agreeing with our notion of a hyperbolic attracting fixed point.

\begin{defn}[Region of Type 1]
	The \emph{region of type 1} of $x_{p}(\lambda)$ is the set $\{(\lambda,\,\phi(\lambda))\vert-1<\phi(\lambda)<1\}\subset\mathbb{R}^{2}$.
\end{defn}

The region of type 1 of $x_{p}$ is the set of values of the parameter $\lambda$ for which $x_{p}$ is a stable fixed point, along with its corresponding values of the eigenvalue function $\phi$ (multipliers). Therefore, the region of type 1 of a fixed point is its \emph{stability region} as a function of the parameter $\lambda$. Since $g(x,\lambda)$ is smooth on both $\lambda$ and $x$, regions of type 1 consist of the countable union of connected subsets of $\mathbb{R}^2$. In general, we can define regions of higher type respect to the fixed points of the $k$-th iteration of $g$, i.e. the periodic points of prime period $k$.

\begin{defn}[Region of Type $k$]
	If the family of maps $x_{n+1}=g(x_{n},\lambda)$, not necessarily unimodal, has an isolated periodic point of prime period $k\geq1,$, $x_{p}^{k}(\lambda)$, then we can define the \emph{$k$-th eigenvalue function} $\phi_{k}:\mathbb{R}\rightarrow\mathbb{R}$ as
	\begin{equation}
		\phi_{k}(\lambda)=\frac{\partial g^{(k)}}{\partial x}(x_{p}^{k}(\lambda)).\label{eq:EigenvalueFunctionKth}
	\end{equation}
	
	Notice that $\frac{\partial g^{(k)}}{\partial x}$ refers to the \emph{first} derivative of the $k$-th iteration of $g$ and not the $k$-th derivative of $g$, as in standard calculus notation. The set $\{(\lambda,\phi_{k}(\lambda))\in\mathbb{R}^{2}\vert x_{p}^{k}(\lambda)\text{ is stable}\}$ is denoted as the \emph{region of type $k$ for $x_{p}^{k}$}.
\end{defn}

The region of type $k$ of the periodic point $x_{p}^{k}$ is simply the region of type 1 of the function $g^{(k)}$, of which $x_{p}^{k}$ is a fixed point. Moreover, by the chain rule,
\[\phi_{k}(\lambda)=\frac{\partial g^{(k)}}{\partial x}=g'_{\lambda}(x_{p}^{k})g'_{\lambda}(g_{\lambda}(x_{p}^{k}))g'_{\lambda}(g_{\lambda}^{(2)}(x_{p}^{k}))\cdots g'_{\lambda}(g_{\lambda}^{(k-1)}(x_{p}^{k}))\]
so that it suffices that the absolute value of the product of the multipliers along the orbit of a periodic point $x_{p}^{k}$ is less than one in order to guarantee the stability of $x_{p}^{k}$. Nonetheless, the existence of periodic points of period $k>1$ is not assured even for a unimodal map in general. Thereafter, the periodic point, $x=x_{p}^{k}(\lambda)$, when it exists, is stable in an interval of the parameter $\lambda$ with ends given by $\phi_{k}^{-1}(1)$ and $\phi_{k}^{(-1)}(-1)$. We suppose here that the function $\phi_{k}^{-1}$ exists, at least locally, i.e. $\phi'_{k}(\lambda)\neq0$ for all $\lambda$ in an interval around a value of interest $\lambda_0$ associated with the periodic point $x_p^k(\lambda_0)=x_0^k$, according to the Inverse Function Theorem \cite{Protter}. A simpler way of saying the latter is that $x_p^k$ will be attracting as long as $-1<\phi_k(x_p^k)<1$, but notice that depending on the form of the graph of $\phi_k$, the eigenvalue function may never leave this region therefore undefining $\phi_k^{-1}$ for $\lambda=\pm1$. Another way of looking at the region of $k$ is that the type $k$ of a given region is given by the order $k$ of the attracting $k$-cycle existing in that region.

Now, since a point $(\lambda,\phi(\lambda))\in\mathbb{R}^{2}$ can only belong to a unique and specific region of a given periodic point, we can state the following definition.

\begin{defn}[Regular and Reversal Maps]
	Let $\mathcal{A}\subset\mathbb{R}$ be an interval. A map $x_{n+1}=g(x_{n},\lambda)$ is called a \emph{regular map} (respectively, \emph{reversal}) in $\mathcal{A}$ if it has an isolated periodic point such that its associated eigenvalue function $\phi$ has the property that if $\lambda_{1},\,\lambda_{2}\in\mathcal{A}$, with $\lambda_{1}<\lambda_{2}$, then the type of the region containing the point $(\lambda_{2},\phi(\lambda_{2}))$ is \emph{higher} (respectively, \emph{lower}) or equal than the type of the region containing the point $(\lambda_{1},\phi(\lambda_{1}))$.
	\label{def:Regular&ReversalMaps}
\end{defn}

In a regular map, as the value of $\lambda$ increases, new attracting periodic points of progressively higher periods are created and, therefore, the type of the regions associated with these points equally rises (or at least does not decrease) progressively. On the other hand, in a \emph{reversal map}, as the value of the parameter $\lambda$ increases, higher order attracting periodic points disappear and only progressively lower order attracting periodic points remain, so that the type of the associated regions decreases correspondingly.

\begin{example}[The Logistic Map]
	\label{ex:logistic}
	The logistic map $f_{\lambda}(x)=\lambda x(1-x)$ is a classical example of a regular map, since it satisfies the definition \ref{def:Regular&ReversalMaps} in the interval $[0,\lambda_{\infty}]$, with $\lambda_{\infty}\approx3.570$. Its eigenvalue function is
	\[\phi(\lambda)=\left.\lambda(1-2\lambda x)\right|_{x=x_{p}(\lambda)},\]
	for the nonzero fixed point $x_p(\lambda)$. In table \ref{tab:BifurcationsLogistic} some type $k$ regions are shown for the logistic map.
	
	\begin{table}
		\begin{tabular}{|c|c|}
			\hline
			\textbf{Type of the region ($k$)}  & \textbf{Interval {[}$(\lambda_{k-1},\lambda_{k})${]}} \tabularnewline
			\hline
			1  & (0, 3) \tabularnewline
			\hline
			2  & (3, 3.449489...) \tabularnewline
			\hline
			3  & (3.449489..., 3.544090...) \tabularnewline
			\hline
			4  & (3.544090..., 3.564407...) \tabularnewline
			\hline
			5  & (3.564407..., 3.568759...) \tabularnewline
			\hline
			6  & (3.568759..., 3.569692...) \tabularnewline
			\hline
			7  & (3.569692..., 3.569891...) \tabularnewline
			\hline
		\end{tabular}
		\caption{Table of regions of type $k$, $k=1,2,...,7$, for the logistic map. The logistic map is a regular map (see definition \ref{def:Regular&ReversalMaps}).
			Data reproduced from \cite[Table 1.1, p.41]{Elaydi}}
		\label{tab:BifurcationsLogistic}
	\end{table}
\end{example}

\begin{example}[The Exponential Map]
	Consider the one-parameter family of maps $E_{\lambda}(x)=\lambda e^x$, with $\lambda<0$. In figure \ref{fig:reversemap1} are shown three important cases of the graphs of the family$E_{\lambda}$. When $\lambda<-e$ the map has a repelling fixed point because of theorem \ref{thm:attractor-repellor} (see figure \ref{fig:reversemap1}c); when $\lambda=-e$, $E_{\lambda}(-1)=-1$ and $E'_{\lambda}(-1)=-1\neq 1$, so that -1 is a non-hyperbolic fixed point and, by theorem \ref{thm:fixedp}, there is an interval around $\lambda=-e$ in which this fixed point exists and is unique; and when $\lambda>-e$ the fixed point is attracting also because of theorem \ref{thm:attractor-repellor}. Therefore the fixed point of $E_{\lambda}$ undergoes a change in stability from unstable to stable as the value of $\lambda$ increases from $\lambda<-e$ to $\lambda>e$, with the \emph{bifurcation value} being exactly $\lambda=-e$. Also, it is easy to see that $E^2_{\lambda}$ is an increasing function that is concave upward if $E_{\lambda}(x)
	<-1$ and concave downward if $E_{\lambda}(x)>-1$ (see figure \ref{fig:reversemap2}). Thus $E^2_{\lambda}$ has two fixed points at the points $q_1$ and $q_2$ shown in figure \ref{fig:reversemap2} when $\lambda<-e$ and they disappear when $\lambda$ increases to $-e$. Since we know that $E_{\lambda}$ has a single fixed point, these must in fact be periodic points of period 2. Note that, as these period 2 points lose their ``attractiveness'', the fixed point acquires it, so that there is an ``exchange of stability''. Also note that, at $\lambda=-e$,  $\frac{\partial}{\partial\lambda}\left(E^2_{\lambda}\right)'(-1)\neq0$, so that we can apply theorem \ref{thm:perioddoubling} formally prove that a period doubling bifurcation effectively takes place. The bifurcation diagram of this example is shown in figure \ref{fig:reversemap3}, where it is easy to see that it satisfies the definition of a reversal map.
	
	\begin{figure}
		\centering
		\includegraphics[width=0.5\columnwidth]{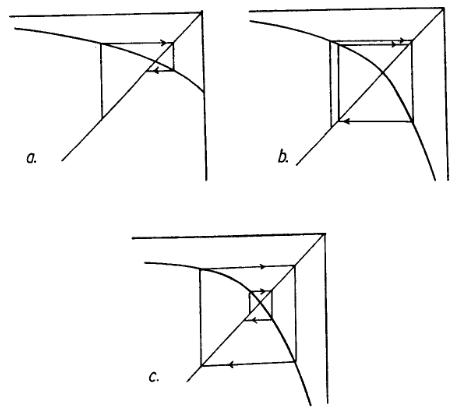}
		\caption{The graphs of the one-parameter family of maps $E_{\lambda}(x)=\lambda e^x$ for (a) $-e<\lambda<0$, (b) $\lambda=-e$, and (c) $\lambda<-e$. Reproduced from \cite{DevaneyIntroCDD}.}
		\label{fig:reversemap1}
	\end{figure}
	
	\begin{figure}
		\centering
		\includegraphics[width=0.5\columnwidth]{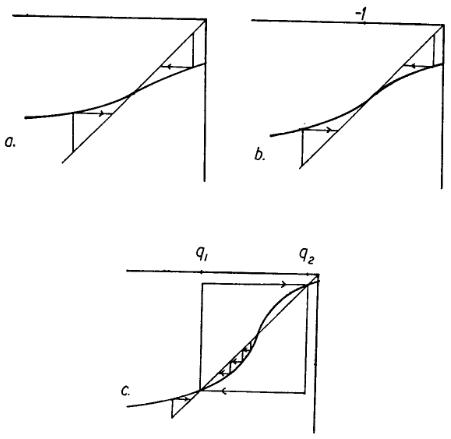}
		\caption{The graphs of $E^2_{\lambda}(x)$ for (a) $-e<\lambda<0$, (b) $\lambda=-e$, and (c) $\lambda<-e$. Reproduced from \cite{DevaneyIntroCDD}.}
		\label{fig:reversemap2}
	\end{figure}
	
	\begin{figure}
		\centering
		\includegraphics[width=0.5\columnwidth]{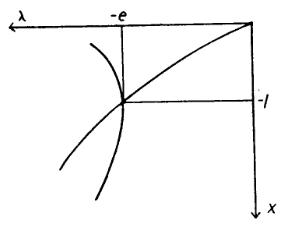}
		\caption{The bifurcation diagram of $E_{\lambda}(x)=\lambda e^x$. Notice that the ``vertical'' curve is the curve of period 2 attracting points. Reproduced from \cite{DevaneyIntroCDD}.}
		\label{fig:reversemap3}
	\end{figure}
\end{example}

\begin{defn}[Regular-Reversal Map]
	\label{def:regular-reversal}
	Let $\mathcal{A}$ be a nonempty interval in $\mathbb{R}$. A map $x_{n+1}=g(x_{n},\lambda)$ is called a \emph{regular-reversal map} in $\mathcal{A}$ if the interval $\mathcal{A}$ can be decomposed in two nonempty disjoint intervals $A_{1}$ and $A_{2}$ such that the map is regular in $A_{1}$ and reversal in $A_{2}$.
\end{defn}

In a regular-reversal map, there is a value of the parameter $\lambda$, say $\lambda_{0}$, that divides the interval of the family of maps $x_{n+1}=g(x_{n},\lambda)$ in two, in one of which it is regular and another one in which it is reversal. A regular-reversal map is interesting because it shows exactly the parametric range where periodicity takes place. So far, it has been difficult to determine for which maps this phenomenon occurs. It is for this reason that precedents have focused only in quadratic polynomial maps \cite{Solis2004}, but also because quadratic polynomials give a complete family of unimodal maps with negative Schwarzian derivative and they represent the simplest example of unimodal maps. In this work, we will focus in higher order polynomial maps, particularly cubic, and we will be able to determine a simple way to create regular-reversal polynomial maps of any degrees. We can also define \emph{reversal-regular maps} in a similar way.

\section{Quadratic maps}

Consider the family of \emph{quadratic discrete dynamical systems} given by

\begin{equation}
	y_{n+1}=\alpha(\lambda)y_{n}^{2}+(\beta(\lambda)+1)y_{n}+\gamma(\lambda),
	\label{eq:QuadraticFamilyGeneral}
\end{equation}

where $\alpha$, $\beta$ and $\gamma$ are smooth functions of the parameter $\lambda$ and $\alpha(\lambda)\neq0$ for all $\lambda$.

The fixed points of this system are obtained by solving the general quadratic equation
\[\alpha(\lambda)y^{2}+\beta(\lambda)y+\gamma(\lambda)=0\]
and they are given by the general quadratic formula
\[y_{0}=(2\alpha)^{-1}\left(-\beta+\sqrt{\beta^{2}-4\alpha\gamma}\right),\quad y_{1}=(2\alpha)^{-1}\left(-\beta-\sqrt{\beta^{2}-4\alpha\gamma}\right).\]

We will assume that $\beta^{2}-4\alpha\gamma>0$ in order to have two distinct real fixed points. We can simplify the system introducing the change of variable
\[x_{n}=Ay_{n}+B\]
where $A=\alpha^{-1}$ and $B=(2\alpha)^{-1}\left(-\beta-\sqrt{\beta^{2}-4\alpha\gamma}\right)=y_{1}$.
Note that $y=B$ is a fixed point of the system \eqref{eq:QuadraticFamilyGeneral}.
We obtain the following modified discrete dynamical system \cite{Solis2004}

\begin{equation}
	x_{n+1}=g(x_{n})\equiv x_{n}^{2}-b(\lambda)x_{n}=x_{n}(x_{n}-b(\lambda))
	\label{eq:QuadraticFamilyModified}
\end{equation}

where $b(\lambda)=-\beta(\lambda)-2\alpha(\lambda)B$. The map $x_{n+1}=g(x_{n})$ is unimodal. Without loss of generality, we will restrict to the case of positive and continuous $b(\lambda)$ for $\lambda>0$.

Under the above transformation, the fixed points of the new system are
\begin{equation}
	x_{0}=0,\quad x_{1}=1+b(\lambda)
	\label{eq:FixedPointsModifiedQuadratic}
\end{equation}

so that the eigenvalue function for each of them is given by
\[\phi(x_{0})=\frac{\partial g}{\partial x}(x_{0})=-b(\lambda),\quad\phi(x_{1})=\frac{\partial g}{\partial x}(x_{1})=2+b(\lambda).\]

The corresponding regions of type 1 are
\begin{eqnarray*}
	x_{0} & : & \{(\lambda,b(\lambda))\vert-1<b(\lambda)<1\}\\
	x_{1} & : & \{(\lambda,b(\lambda))\vert-3<b(\lambda)<-1\}.
\end{eqnarray*}

There are also two periodic points of period two, which we can calculate by solving
\[x=g^{(2)}(x)=g(g(x))=x^{4}-2bx^{3}+b(b-1)x^{2}+b^{2}x.\]
We can factor $x(x-(b+1))$ from the expression since it corresponds to the fixed points (order 1) and we then solve for the roots of
\[x^{2}+(1-b)x+(1-b)=0;\]
and, again by the general quadratic formula, we have
\begin{eqnarray*}
	x_{1}^{2} & = & (2)^{-1}\left(b-1+\sqrt{b^{2}+2b-3}\right)\\
	x_{2}^{2} & = & (2)^{-1}\left(b-1-\sqrt{b^{2}+2b-3}\right)
\end{eqnarray*}
with eigenvalue function
\[\phi_{2}(\lambda)=\frac{\partial g^{(2)}}{\partial x}(x)=\left.4x^{3}-6bx^{2}+2b(b-1)x+b^{2}\right|_{x=x_{j}^{2}},\quad j=1,2\]
and, therefore, corresponding type 2 region given by
\[\{(\lambda,b(\lambda))\vert1<b(\lambda)<\sqrt{6}-1\approx1.45\}.\]

Moreover, there are $2^{N}$ periodic points of period $2^{N}$ with region of type $N$ given by \[\{(\lambda,b(\lambda))\vert b_{N-1}<b(\lambda)<b_{N}\}\]
where the sequence $\{b_{k}\}_{k=0}^{\infty}$ is convergent, and converges to the limiting value $b_{\infty}=1.569945\cdots$ \cite{Solis2004}. The sequence $\{b_{k}\}_{k=0}^{\infty}$ is related to the one described by \cite{Feigenbaum} in \cite{Feigenbaum} and converges exponentially. Table \ref{tab:RegionsTypeK} summarizes this results.

\begin{table}
	\begin{tabular}{|c|c|}
		\hline
		Type of the region ($k$)  & Interval {[}$(b_{k},b_{k+1})${]} \tabularnewline
		\hline
		1 ($x_{0}$)  & (-3, -1) \tabularnewline
		\hline
		1 ($x_{1}$)  & (-1, 1) \tabularnewline
		\hline
		2 ($x_{1}^{2},x_{2}^{2}$)  & (1, $\sqrt{6}-1\approx1.45)$ \tabularnewline
		\hline
		$\vdots$  & $\vdots$ \tabularnewline
		\hline
		$\infty$  & ($b_{\infty}=1.56994567...,\infty$) \tabularnewline
		\hline
	\end{tabular}
	\caption{Regions of type $k$ for the quadratic map $x_{n+1}=x_{n}^{2}-b(\lambda)x_{n}$.}
	\label{tab:RegionsTypeK}
\end{table}

The values of the original Feigenbaum sequence, $\{\lambda_{k}\}_{k=0}^{\infty}$, correspond to $\lambda$ values of the points of intersection of the function $b(\lambda)$ with the the values of the sequence $\{b_{k}\}_{k=0}^{\infty}$. Stated otherwise, the  sequence $\{b_{k}\}_{k=0}^{\infty}$ results from evaluating the Feigenbaum sequence $\{\lambda_{k}\}_{k=0}^{\infty}$ in the function $b(\lambda)$. When the function $b(\lambda)$ is the identity function, we can recover the original Feigenbaum's sequence.

\emph{Feigenbaum's sequence}, $\{\lambda_{k}\}_{k=0}^{\infty}$, describes the values of the parameter $\lambda$ for which the \emph{period doubling bifurcation} takes place in a unimodal map. The sequence converges exponentially. The values $\{\lambda_{k}\}_{k=0}^{\infty}$ for a given iteration function $g$ --of the following form-- are given implicitly by the definition of  \cite{Feigenbaum} in its original paper \cite{Feigenbaum}, for a map $f$ as
\[x_{p}^{2^{k}}=g^{2^{k}}(x_{p}^{2^{k}})=(\lambda_{k}f)^{2^{k}}(x_{p}^{2^{k}}),\quad\lambda\in\mathbb{N},\]
for the periodic point of period $2^{k}$, $x_{p}^{2^{k}}$, where we must solve for $\lambda_{k}$. In table \ref{tab:feigenbaum} the values of a Feigenbaum sequence are shown for a specific iteration function.

A simple way to estimate the values of Feigenbaum sequences is starting from Feigenbaum's constant, $\delta=4.669201609...$, and the first calculated values of the sequence $\{\lambda_{k}\}_{k\in\mathbb{N}}$, through the approximation
\begin{equation}
	\lambda_{k+1}=\lambda_{k}+\frac{\lambda_{k}-\lambda_{k-1}}{\delta}.
	\label{eq:Bapprox}
\end{equation}
This estimate is used as a predictor for the next value of the sequence and tends to be more precise as $k\rightarrow\infty$.

\begin{figure}
	\centering
	\includegraphics[width=0.66\columnwidth]{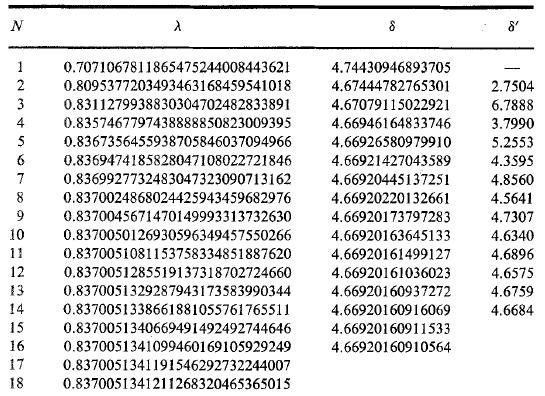}
	\caption{Feigenbaum's sequence for $f(x)=1-2x^{2}$. $N$ denotes the cycle of order $2^{N}$. The parameter is denoted as $\lambda$. $\delta_{N}=(\lambda_{N+1}-\lambda_{N})/(\lambda_{N+2}-\lambda_{N+1})$
		and $\delta_{N}'=(\delta_{N+1}-\delta_{N})/(\delta_{N+2}-\delta_{N+1})$. Reproduced from \cite{Feigenbaum}}
	\label{tab:feigenbaum}
\end{figure}

\section{Types of periodic points and chaos}
\label{sec:types-perpt-chaos}

By the known properties of the sequence $\{b_{k}\}_{k=0}^{\infty}$ and Sarkovskii's theorem, we can state the following propositions \cite{Solis2004}:

\begin{prop}
	\label{prop:bk-perpt}
	If the function $b(\lambda)$ is bounded from above by $b_{N}$, the map $x_{n+1}=g(x_{n})\equiv x_{n}^{2}-b(\lambda)x_{n}$ can only have periodic points of period $2^{m}$ with $m<N$. Moreover, if the bound is given by $b_{\infty}$, the system is not chaotic and it can only have periodic points of periods of powers of two.
\end{prop}

It is worth remarking that there are unimodal maps that have no periodic points with periods greater than one, in which case the bifurcation diagram consists of a single branch given by the fixed point.

\begin{prop}
	\label{prop:not-chaotic}
	Suppose that the function $b(\lambda)$ is bounded. If the supreme of the function $b(\lambda)$ lies within the interval $(b_{n},\, b_{n+1})$, then the map $x_{n+1}=x_{n}^{2}-b(\lambda)x_{n}$ only has periodic points of period $2^{k}$ with $k\in\{1,\,2,\,...,\, n\}$ and therefore, it is \emph{not} chaotic.
\end{prop}

In order for the system $x_{n+1}=x_{n}^{2}-b(\lambda)x_{n}$ to be a regular reversal map for $\lambda$ in some interval $\mathcal{A}=(\lambda_I,\,\lambda_F)$, it is necessary that the function $b(\lambda)$ has the following property: there exist $\lambda_{1}<\lambda_{2}<\lambda_{3}$ in $\mathcal{A}$ such that $b(\lambda_{1})<b(\lambda_{3})=b_{i}<b(\lambda_{2})$ for some $i\in\mathbb{N}$ where $b_{i}$ is the $i$-th element of the sequence $\{b_{k}\}_{k=0}^{\infty}$. It is relatively easy to construct functions $b(\lambda)$ with the aforementioned property, which means there is a large class of unimodal maps that are \emph{not} chaotic.

Beyond the value $b_{\infty}$, there are other regions that divide the positive quadrant of $\mathbb{R}^{2}$. The upper region of $\mathbb{R}^{2}$, that is, the region $\{(\lambda,b(\lambda))\vert b(\lambda)>b_{\infty}\}$, can be defined as the region of type $\infty$ and corresponds to divergent orbits. It is not clear how this region is subdivided and, in general, this is an open question.

\section{Quadratic examples}

The goal of the following examples is to show that regular reversal maps are a very broad family and that it is relatively simple to construct maps with specific --\emph{ad
	hoc}-- bifurcation diagrams  with or without chaos.

\begin{example}
	\label{eg:quadratic1} Let $b(\lambda)$ be a quadratic unimodal function given by
	\[b(\lambda)=\frac{4\alpha_{0}}{\beta_{0}^{2}}\lambda(\beta_{0}-\lambda),\]
	with $\alpha_{0},\,\beta_{0}>0$. Note that $\alpha_{0}$ is the maximum value and $\beta_{0}$ is a root of $b(\lambda)$ (the other root is $\lambda=0$). We are interested in knowing the periodic point structure of the map $x_{n+1}=x_{n}^{2}-b(\lambda)x_{n}$ as we vary the parameter $\alpha_{0}$. Fixing $\beta_{0}=2$, we then have the following cases:
	
	\begin{itemize}
		\item If $0<\alpha_{0}<b_{0}=1$, we have a bifurcation diagram consisting of a single straight line coincident with the $\lambda$ axis.
		\item If $1<\alpha_{0}<b_{2}=\sqrt{6}-1$, we obtain a regular reversal map with a bifurcation diagram consisting of a straight line, as the previous case, but with a closed asymmetrical loop homeomorphic to a circle (see figure \ref{fig:diagbif1-a1.2} on page \pageref{fig:diagbif1-a1.2}). The lack of symmetry stems from the fact that the upper branch of the loop corresponds the the periodic point $x_{1}^{2}$ and the lower branch to $x_{2}^{2}$. The diameter of the loop is calculated by finding the two intersections of $b(\lambda)$ with the line $b=b_{1}=1$ and it is equal to $2\beta_{0}\sqrt{\alpha_{0}^{2}-1}\alpha_{0}$. The graph of $b(\lambda)$ is shown in figure \ref{fig:blambda1-a1.2b2} at page \pageref{fig:blambda1-a1.2b2}.
		\item If $b_{1}=1<\alpha_{0}<\sqrt{6}-1=b_{2}$, two new loops appear: one that comes out of the lower branch of the previous case and another one from the upper branch, corresponding to the periodic points of period $2^{2}$ (see figure \ref{fig:diagbif1-a1.5}). The matching graph of $b(\lambda)$ is shown in figure \ref{fig:blambda1-a1.5b2}.
		\item If $b_{n}<\alpha_{0}<b_{n+1}$ for some $n\in\mathbb{N}$, the graph of $b(\lambda)$ looks like in figure \ref{fig:blambda1-a1.562b2} and we have a bifurcation diagram consisting of a straight line with a set of nested loops as shown in figure \ref{fig:diagbif1-a1.562}.
		\item If $\alpha_{0}=1.56994567...=b_{\infty}$, then we have branches that bifurcate to form infinite nested loops, as is shown in figure \ref{fig:diagbif1-a1.57} and whose graph of $b(\lambda)$ is shown in figure \ref{fig:blambda1-a1.57b2}.
	\end{itemize}
	
	\begin{figure}
		\centering
		\subfloat[$\alpha_{0}=1.2$ and $\beta_{0}=2$.]
		{
			\centering
			\includegraphics[width=0.5\columnwidth]{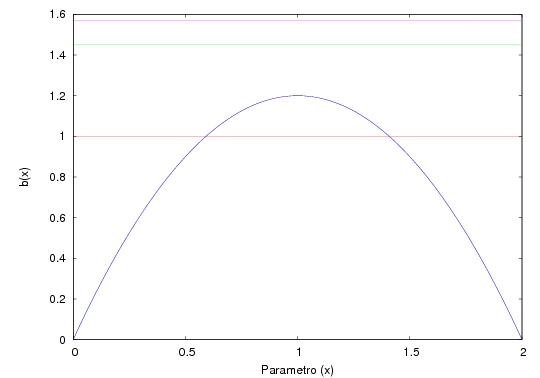}
			\label{fig:blambda1-a1.2b2}
		}
		\subfloat[$\alpha_{0}=1.5$ and $\beta_{0}=2$.]
		{
			\centering
			\includegraphics[width=0.5\columnwidth]{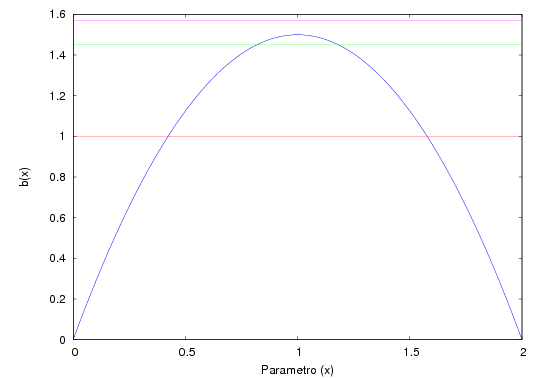}
			\label{fig:blambda1-a1.5b2}
		}\\
		\subfloat[$\alpha_{0}=1.562$ and $\beta_{0}=2$.]
		{
			\centering
			\includegraphics[width=0.5\columnwidth]{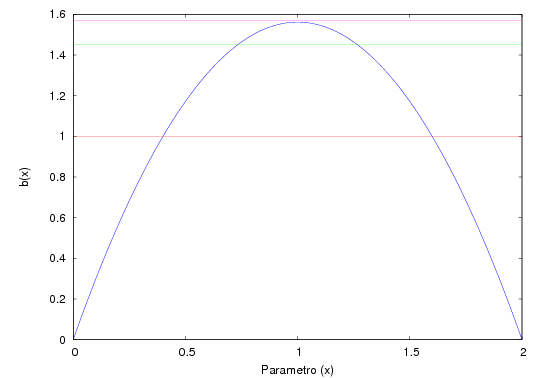}
			\label{fig:blambda1-a1.562b2}
		}
		\subfloat[$\alpha_{0}=b_{\infty}\approx1.57$ and $\beta_{0}=2$.]
		{
			\centering
			\includegraphics[width=0.5\columnwidth]{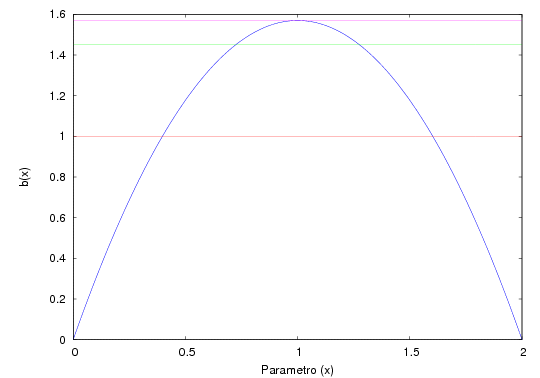}
			\label{fig:blambda1-a1.57b2}
		}
		\caption{$b(\lambda)$ functions of the quadratic regular reversal maps of example \ref{eg:quadratic1}. The horizontal lines correspond to the values of $b_{k}$, for $k=1,2$ and $\infty$.}
		\label{fig:blambda1}
	\end{figure}
	
	\begin{figure}
		\centering
		\subfloat[$\alpha_{0}=1.2$ y $\beta_{0}=2$.]
		{
			\centering
			\includegraphics[width=0.5\columnwidth]{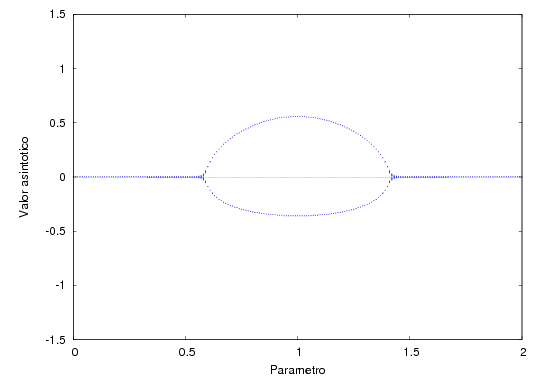}
			\label{fig:diagbif1-a1.2}
		}
		\subfloat[$b_{2}<\alpha_{0}=1.5<b_{3}$ and $\beta_{0}=2$.]
		{
			\centering
			\includegraphics[width=0.5\columnwidth]{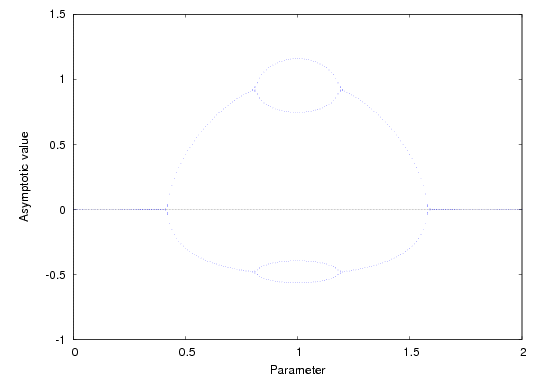}
			\label{fig:diagbif1-a1.5}
		}\\
		\subfloat[$b_{3}<\alpha_{0}=1.562<b_{4}$ y $\beta_{0}=2$.]
		{
			\centering
			\includegraphics[width=0.5\columnwidth]{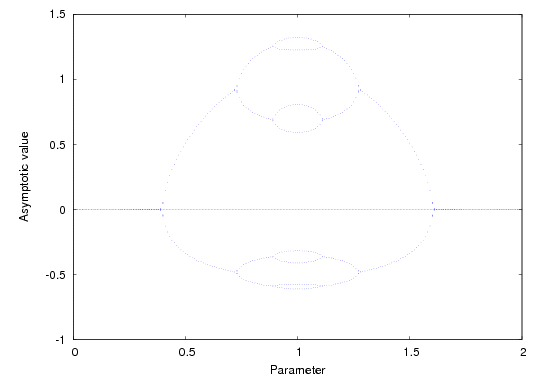}
			\label{fig:diagbif1-a1.562}
		}
		\subfloat[$\alpha_{0}=1.57\approx b_{\infty}$ and $\beta_{0}=2$.]
		{
			\centering
			\includegraphics[width=0.5\columnwidth]{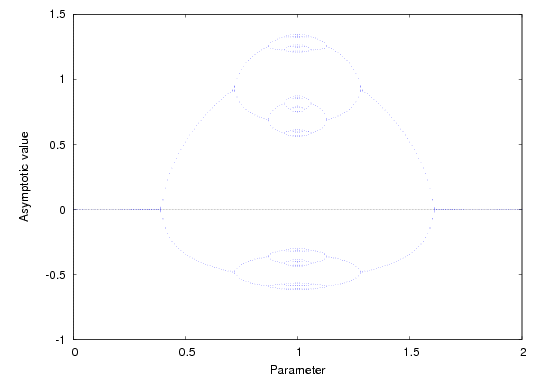}
			\label{fig:diagbif1-a1.57}
		}
		\caption{Bifurcation diagrams of the quadratic regular reversal maps of example \ref{eg:quadratic1}.}
		\label{fig:diagbif1}
	\end{figure}
	
\end{example}

\begin{example}
	\label{ex:quadratic1b}
	If we modify $b(\lambda)$ choosing it as periodic ``unimodal''\footnote{Formally, this cannot be a unimodal function, since it is a requirement for a unimodal function to have a single maximum} of the form
	\begin{equation}
		b(\lambda)=\alpha\sin{2\pi(\lambda+3/2)}+1,
		\label{eq:quadratic1b}
	\end{equation}
	with $1+\alpha<b_{\infty}$), the new bifurcation diagrams will be very similar to those of the last case, except that they will be periodic. In figure \ref{fig:blambda1b} at page \pageref{fig:blambda1b}, the graphs of $b(\lambda)$ are shown, and in figure \ref{fig:diagbif1b}, are the corresponding bifurcation diagrams, both for the corresponding parameter cases of figures \ref{fig:blambda1} and \ref{fig:diagbif1}, respectively.
	
	\begin{figure}
		\centering
		\subfloat[$\alpha=0.2$.]
		{
			\centering
			\includegraphics[width=0.5\columnwidth]{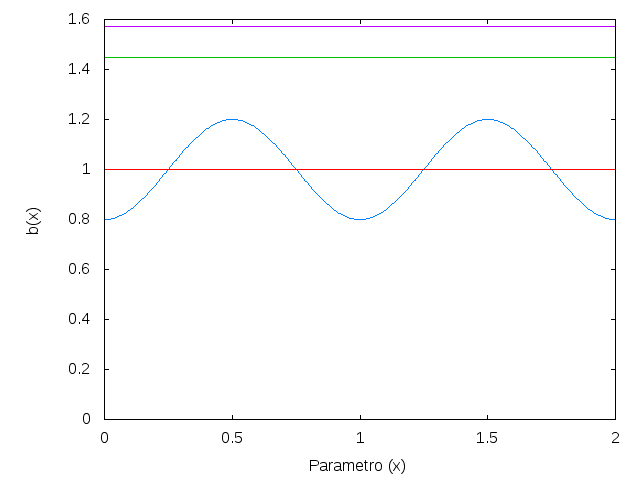}
			\label{fig:blambda1b-a1.2}
		}
		\subfloat[$\alpha=0.5$.]
		{
			\centering
			\includegraphics[width=0.5\columnwidth]{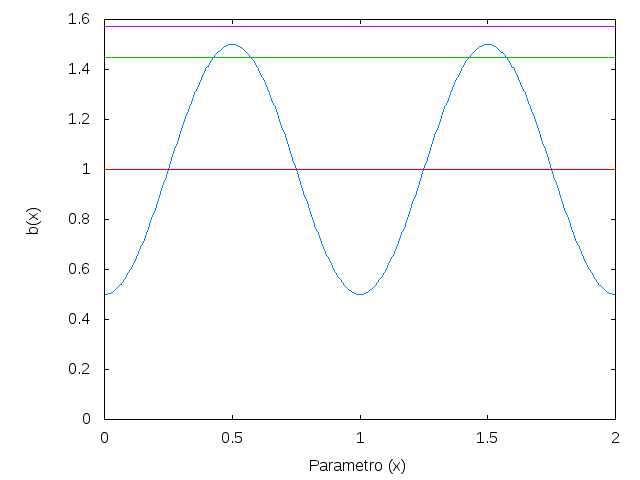}
			\label{fig:blambda1b-a0.5}
		}\\
		\subfloat[$\alpha=0.562$.]
		{
			\centering
			\includegraphics[width=0.5\columnwidth]{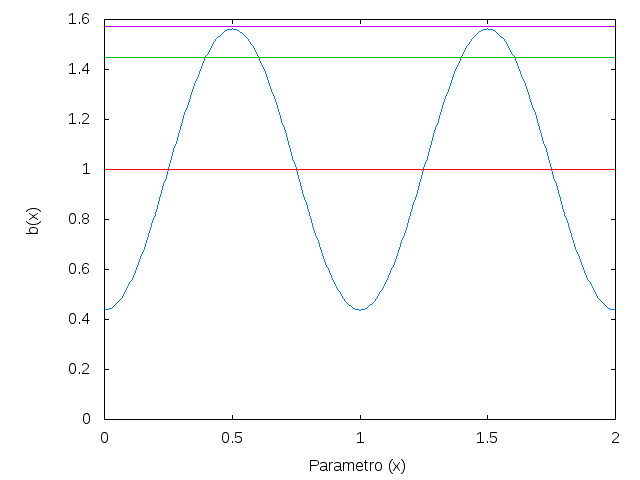}
			\label{fig:blambda1b-a1.562}
		}
		\subfloat[$\alpha=0.57$. ($1+\alpha\approx b_{\infty}$.)]
		{
			\centering
			\includegraphics[width=0.5\columnwidth]{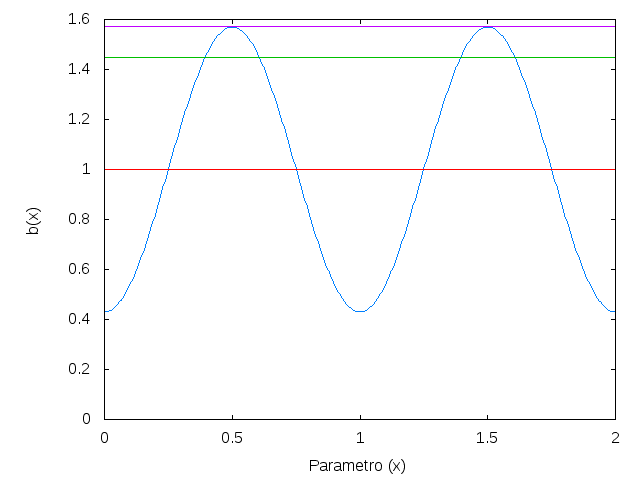}
			\label{fig:blambda1b-a0.57}
		}
		\caption{Functions $b(\lambda)$ of the regular reversal maps of example \ref{ex:quadratic1b} given by the equation \ref{ex:quadratic1b}. The horizontal lines correspond to the values of $b_{k}$, for $k=1,2$ and $\infty$.}
		\label{fig:blambda1b}
	\end{figure}
	
	\begin{figure}
		\centering
		\subfloat[$\alpha=0.2$.]
		{
			\centering
			\includegraphics[width=0.5\columnwidth]{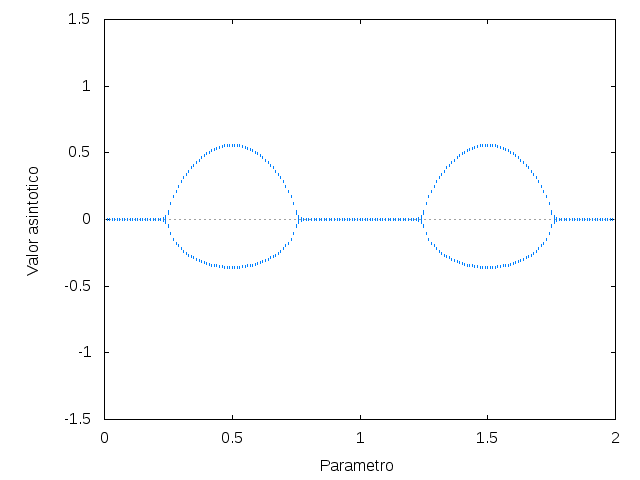}
			\label{fig:diagbif1b-a1.2}
		}
		\subfloat[$b_{2}<1+\alpha=1.5<b_{3}$.]
		{
			\centering
			\includegraphics[width=0.5\columnwidth]{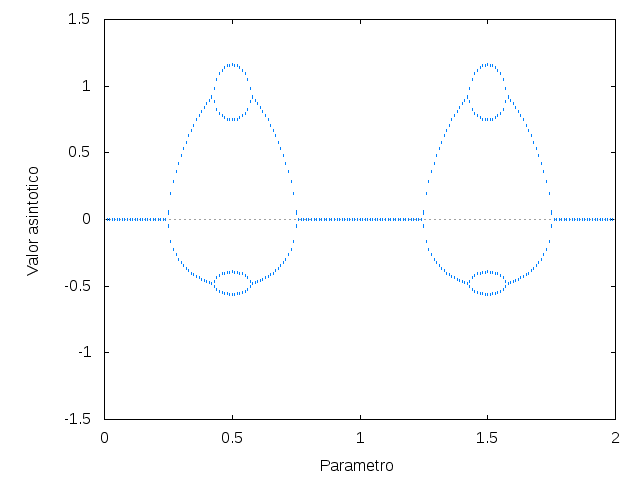}
			\label{fig:diagbif1b-a0.5}
		}\\
		\subfloat[$b_{3}<1+\alpha=1.562<b_{4}$.]
		{
			\centering
			\includegraphics[width=0.5\columnwidth]{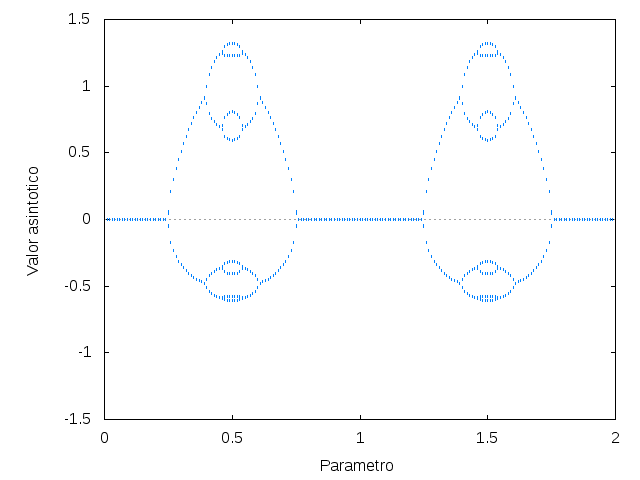}
			\label{fig:diagbif1b-a0.562}
		}
		\subfloat[$1+\alpha=1.57\approx b_{\infty}$.]
		{
			\centering
			\includegraphics[width=0.5\columnwidth]{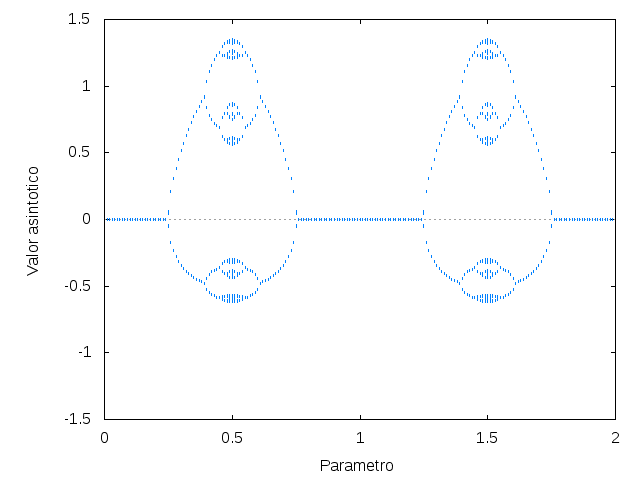}
			\label{fig:diagbif1b-a0.57}
		}
		\caption{Bifurcation the regular reversal maps of example \ref{ex:quadratic1b}.}
		\label{fig:diagbif1b}
	\end{figure}
	
\end{example}

\begin{example}
	\label{eg:mimod} If we introduce an asymmetry factor to the function \ref{eq:quadratic1b} we can achieve an asymmetrical bifurcation diagram with the different lobes having different amplitudes. One way to accomplish the above is defining $b(\lambda)$ as
	\begin{equation}
		b(\lambda)=\left(1-e^{-\lambda}\right)\sin^{2}\left(\pi\lambda/4\right)\{\alpha\sin\left[\pi\left(\lambda+3/2\right)\right]+4/5\}+4/5.
		\label{eq:mimod}
	\end{equation}
	
	Varying the parameter $\alpha$ we obtain graphs of $b(\lambda)$ like the ones shown in figure \ref{fig:blambda-mimod} at page \pageref{fig:blambda-mimod}, with its respective bifurcation diagrams, shown in figure \ref{fig:diagbif-mimod} at page \pageref{fig:diagbif-mimod}. For small values of $\alpha$ we have two lobes (or ``humps'') in the function $b(\lambda)$, with its local maxima under the value $b_{2}$, which translates into a bifurcation diagram with four asymmetrical lobes and of different amplitude (see figure \ref{fig:diagbif-mimod-a0.2} at page \pageref{fig:diagbif-mimod-a0.2}). However, due to asymmetry, rising the value of $\alpha$ to $0.4$ will only produce the appearance of nested loops in two of the right lobes of the bifurcation diagram (see figure \ref{fig:diagbif-mimod-a0.4}), which repeats itself as we continue to raise the value of $\alpha$ (see figure \ref{fig:diagbif-mimod-a0.5}).
	
	\begin{figure}
		\centering
		\subfloat[$\alpha=0.2$.]
		{
			\centering
			\includegraphics[width=0.5\columnwidth]{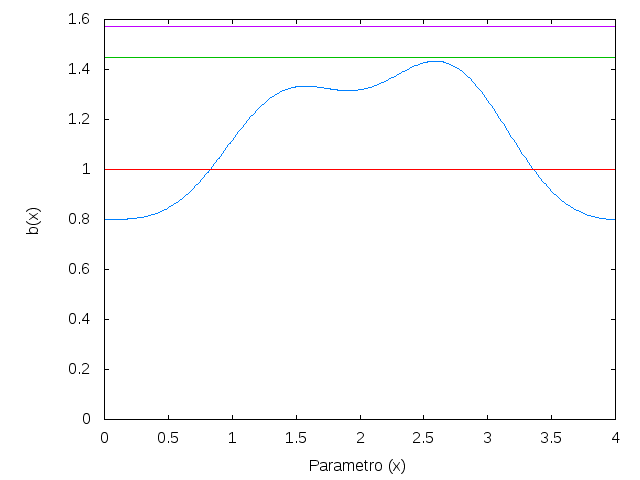}
			\label{fig:blambda-mimod-a0.2}
		}
		\subfloat[$\alpha=0.4$.]
		{
			\centering
			\includegraphics[width=0.5\columnwidth]{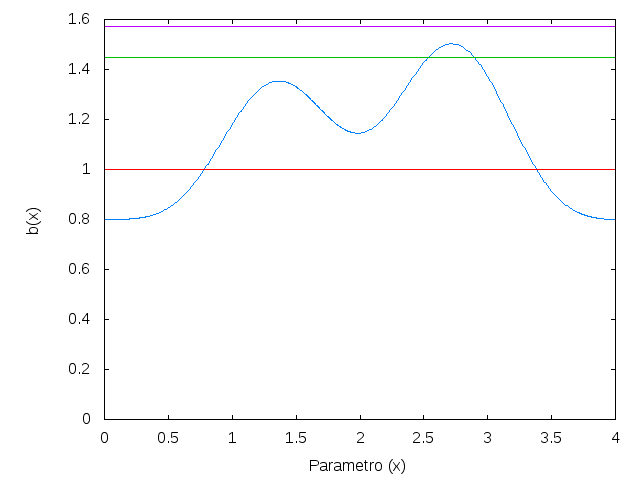}
			\label{fig:blambda-mimod-a0.4}
		}\\
		\subfloat[$\alpha=0.5$.]
		{
			\centering
			\includegraphics[width=0.5\columnwidth]{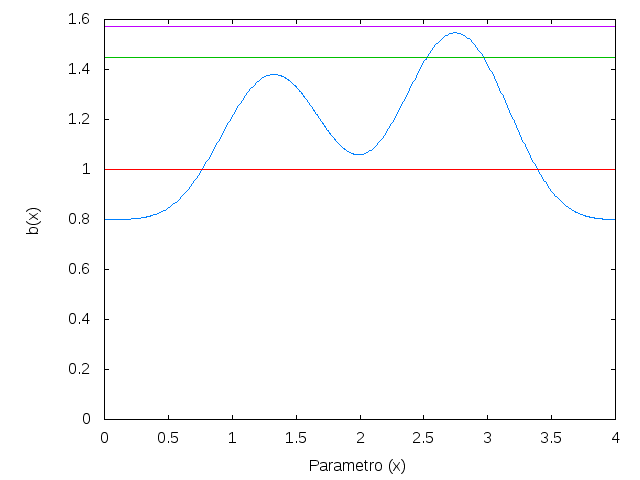}
			\label{fig:blambda-mimod-a0.5}
		}
		\caption{Graphs of $b(\lambda)$, given by equation \ref{eq:mimod}, for the regular reverse maps of example \ref{eq:mimod}.}
		\label{fig:blambda-mimod}
	\end{figure}
	
	\begin{figure}
		\centering
		\subfloat[$\alpha=0.2$.]
		{
			\centering
			\includegraphics[width=0.5\columnwidth]{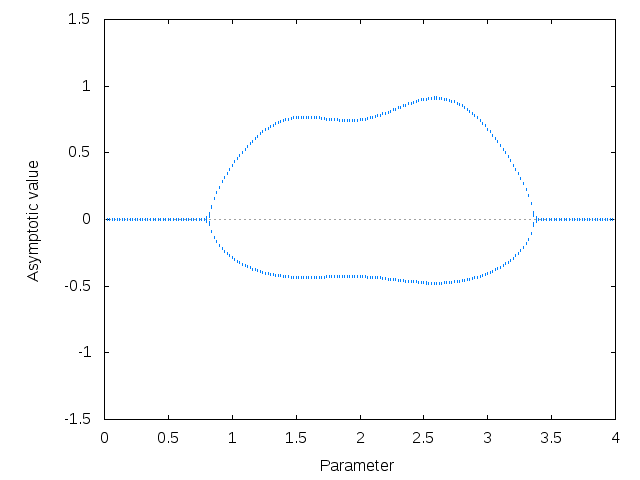}
			\label{fig:diagbif-mimod-a0.2}
		}
		\subfloat[$\alpha=0.4$.]
		{
			\centering
			\includegraphics[width=0.5\columnwidth]{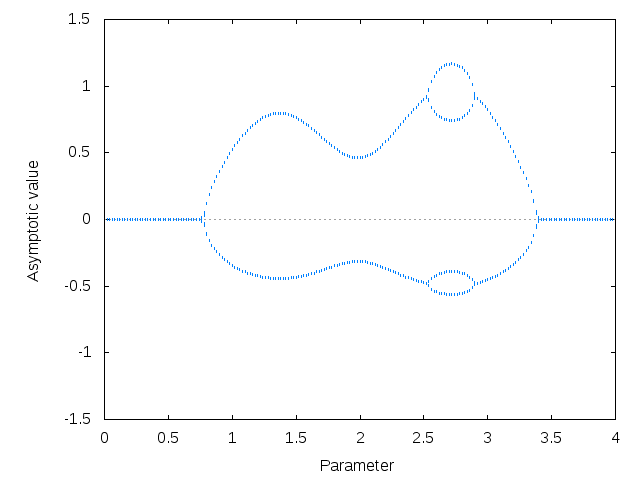}
			\label{fig:diagbif-mimod-a0.4}
		}\\
		\subfloat[$\alpha=0.5$.]
		{
			\centering
			\includegraphics[width=0.5\columnwidth]{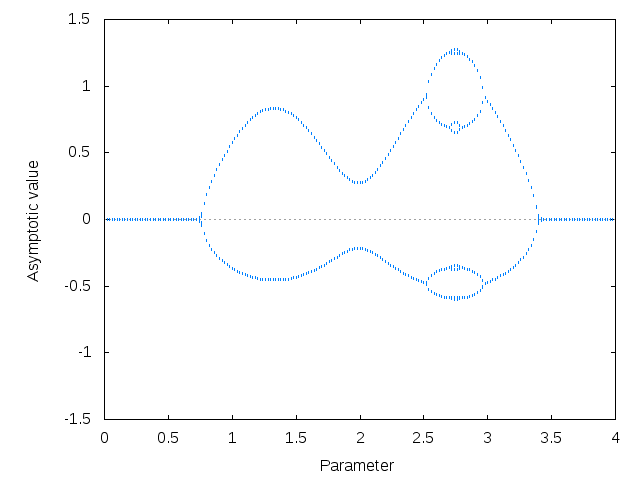}
			\label{fig:diagbif-mimod-a0.5}
		}
		\caption{Bifurcation diagrams of the regular reversal maps of example \ref{eg:mimod}.}
		\label{fig:diagbif-mimod}
	\end{figure}
	
\end{example}

\begin{example}
	\label{eg:example2}
	A similar effect to the one of the last example can be achieved by taking  $b(\lambda)$ as a polynomial of fourth grade in $\lambda$, with its coefficients chosen in such way that it is not \emph{unimodal}. For example,
	\begin{equation}
		b(\lambda)=-\left(\alpha\lambda^{4}-1.533\lambda^{3}+4.1083\lambda^{2}-4.166\lambda\right)
		\label{eq:example2}
	\end{equation}
	with $\alpha\in(0.18,0.192)$. The function $b(\lambda)$ has three local extremes (see figure \ref{fig:blambda2} at page \pageref{fig:blambda2}). In this case, the bifurcation diagram is shown for different values of the parameter in figure \ref{fig:diagbif2} at page \pageref{fig:diagbif2} where, for progressively higher values of $\alpha$ we progressively obtain lobes that stem from the branches corresponding to the periodic points. In this case the diagram is no symmetrical nor periodic because the function $b(\lambda)$ is neither.
	
	\begin{figure}
		\centering
		\subfloat[$\alpha=0.19166$.]
		{
			\centering
			\includegraphics[width=0.5\columnwidth]{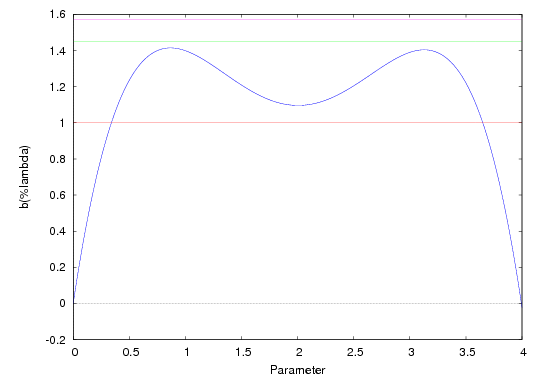}
			\label{fig:blambda2-a0.19166}
		}
		\subfloat[$\alpha=0.191$.]
		{
			\centering
			\includegraphics[width=0.5\columnwidth]{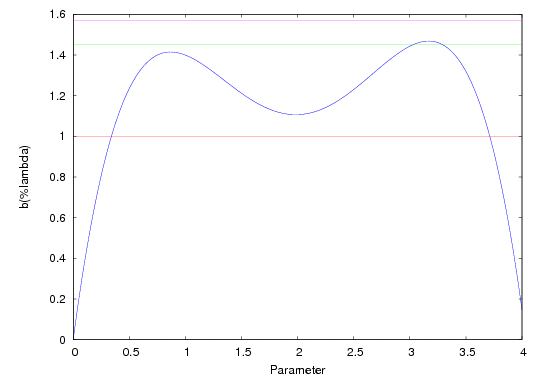}
			\label{fig:blambda2-a0.191}
		}\\
		\subfloat[$\alpha=0.19$.]
		{
			\centering
			\includegraphics[width=0.5\columnwidth]{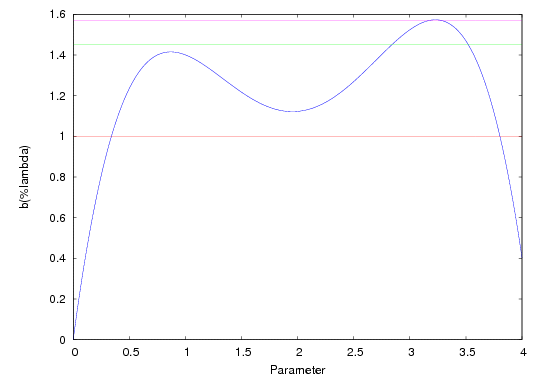}
			\label{fig:blambda2-a0.19}
		}
		\caption{Graphs of $b(\lambda)$, given by equation \ref{eq:example2}, for the quadratic regular reversal maps of example \ref{eg:example2}.}
		\label{fig:blambda2}
	\end{figure}
	
	\begin{figure}
		\centering
		\subfloat[$\alpha=0.19166$.]
		{
			\centering
			\includegraphics[width=0.5\columnwidth]{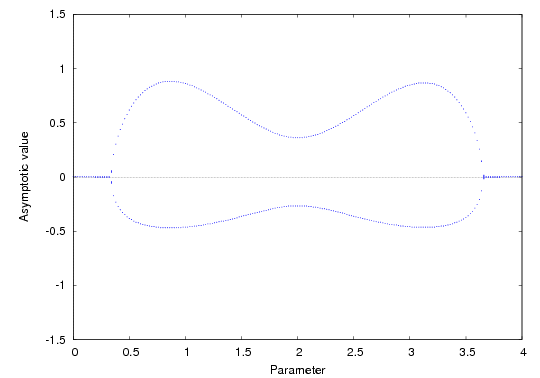}
			\label{fig:diagbif2-a0.19166}
		}
		\subfloat[$\alpha=0.191$.]
		{
			\centering
			\includegraphics[width=0.5\columnwidth]{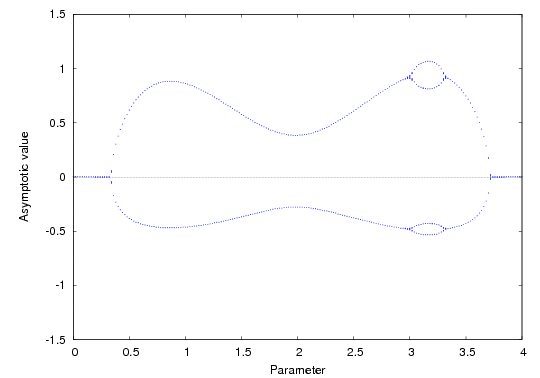}
			\label{fig:diagbif2-a0.191}
		}\\
		\subfloat[$\alpha=0.19$.]
		{
			\centering
			\includegraphics[width=0.5\columnwidth]{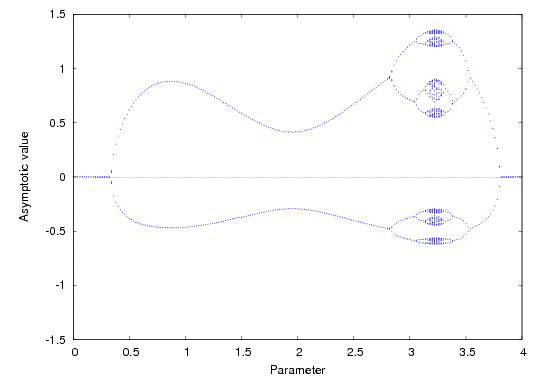}
			\label{fig:diagbif2-a0.19}
		}
		\caption{Bifurcation diagrams of the regular reversal maps of example \ref{eg:example2}.}
		\label{fig:diagbif2}
	\end{figure}
	
\end{example}

With the above examples we have shown we can design discrete dynamical systems in order to have specific bifurcation diagrams ``on demand'' with desired properties. The fundamental fact that allows us to do the latter is to note where the function $\phi(\lambda)$ intersects the values of the sequence $\{b_k\}_{k=1}^{\infty}$. The detailed form of $\phi(\lambda)$ is irrelevant as long as it does not cross a value of the sequence, for it only affects the specific form of the ``branches'' of the periodic points in the bifurcation diagram, but it does not alter the fundamental structure of the set of periodic points.

\chapter{Quadratic maps revisited}
\label{cha:QuadRev}

Now, we will restate and expand the work of \cite{Solis2004} in a way that will allow us to gain greater insight into the way we can control the appearance of chaos in a parametric family of quadratic maps and we will be able to generalize these results in a natural way for cubic maps and, later on, for general polynomial maps of $n$-th degree.

\section{General quadratic form}
\label{sub:GeneralQuadraticForm}
Consider the quadratic iteration function $f(y)$ as stated by the

\begin{defn}[General Quadratic Map]
	\label{def:GQM}
	The \emph{General Quadratic Map (GQM)} is defined by
	\begin{equation}
		f_2(y) := \alpha+(\beta+1)y+\gamma y^2 = y - P_{f_2}(y),
	\end{equation}
	where
	\begin{equation}
		P_{f_2}(y):=-(\alpha+\beta y+\gamma y^2)
	\end{equation}
	and $\alpha$, $\beta$ and $\gamma$ are real coefficients that depend upon the parameter $\lambda$ and are $C^1$. $P_{f_2}$ is called the Fixed Points Polynomial (of the general quadratic map).
\end{defn}

It is clear that any quadratic map can be put in this form. Although the minus sign in the definition of $P_{f_2}$ may now seem ``unnatural'', its function will be clear shortly\footnote{Check page \pageref{minus_sign} below.}.

Evidently, the zeros of $P_{f_2}$ are the fixed points of $f_2$ --thence the name of $P_{f_2}$-- since, if $y_0$ is a root of $P_{f_2}$, then $P_{f_2}(y_0)=0$ implies $f_2(y_0)=y_0$.

In the following, for simplicity, we will drop the subscript ``2'' from $f$ to denote the CQM. Since the derivatives of $f$ are
\begin{equation}
	\begin{aligned}
		f'(y) &=  2\gamma y + \beta + 1, \\
		f''(y) &= 2\gamma,
	\end{aligned}
\end{equation}
straightforward calculations lead us to
\begin{lem}
	$f$ has the following two roots and critical point, respectively,
	
	\begin{equation}
		\begin{aligned}
			y_0^r &= -\frac{\beta+1+\sqrt{(\beta+1)^2-4\alpha\gamma}}{2\gamma},\\
			y_1^r &= -\frac{\beta+1-\sqrt{(\beta+1)^2-4\alpha\gamma}}{2\gamma},\\
			y_c &= -\frac{\beta+1}{2\gamma},
		\end{aligned}
	\end{equation}
	with corresponding extreme value
	\[f_{ext}=\frac{4\alpha\gamma-(\beta+1)^2}{4\gamma}=\alpha-\gamma y_c.\]
	And, finally, it has the fixed points
	\begin{equation}
		\begin{aligned}
			y_0^* &= -\frac{\beta+\sqrt{{\beta}^{2}-4\,\alpha\,\gamma}}{2\,\gamma},\\
			y_1^* &= -\frac{\beta-\sqrt{{\beta}^{2}-4\,\alpha\,\gamma}-\beta}{2\,\gamma}.
		\end{aligned}
	\end{equation}
\end{lem}

\section{Linear factors form}
\label{sub:LinearFactorsForm}

Suppose we know both roots of $P_{f_2}$ and they are $y_0,\,y_1\in\mathbb{C}$, which exist by the fundamental theorem of algebra. Then, by the factor theorem, we know that $(y-y_i),\,i\in\{0,1\}$ is a factor of $P_{f_2}$ and we can rewrite it as
\begin{equation}
	P_{f_2}(y)=M(y-y_0)(y-y_1),\quad M\in\mathbb{R}\setminus\{0\}.
\end{equation}
Therefore, we can also rewrite $f_2$ as
\begin{equation}
	f_2(y)=y-M(y-y_0)(y-y_1).
\end{equation}

Since $M\in\mathbb{R}\setminus\{0\}$, then we can define $M:=s\tilde{M}$, where
\begin{equation}
	s:=\mathrm{sgn}(M)=
	\begin{cases}
		1, &\mathrm{if}\,M>0,\\
		-1,&\mathrm{if}\,M<0,
	\end{cases}
\end{equation}
is the sign function and $\tilde{M}:=|M|$; this definition will be useful for our purposes further below.

With the above, we have

\begin{defn}[Linear Factors Form of the Quadratic Map]
	\label{def:LFFQM}
	The GQM can be rewritten as
	\begin{equation}
		f_2(y)=y-s\tilde{M}(y-y_0)(y-y_1):=h_2(y,\,\lambda)
		\label{eq:LinearFactorsForm}
	\end{equation}
	where $M,\,y_1$ and $y_2$ are smooth functions of the parameter $\lambda$. We call $h_2$ the \emph{Linear Factors Form of the Quadratic Map (LFFQM)}.
\end{defn}

Again, straightforward calculations lead us to

\begin{lem}
	\eqref{eq:LinearFactorsForm} has first and second derivatives with respect to $y$ given by
	\begin{equation}
		\begin{aligned}
			h_2'(y,\,\lambda) &= 1-s\tilde{M}\left(2y-y_0-y_1\right),\\
			h_2''(y,\,\lambda) &= -2\,s\tilde{M}.
		\end{aligned}
	\end{equation}
	
	The roots of $h_2$ are
	
	\begin{equation}
		\begin{aligned}
			y_0^r &= -\frac{\sqrt{s^2\tilde{M}^2(y_1-y_0)^2+2s\tilde{M}(y_1+y_0)+1}-s\tilde{M}(y_1+y_0)-1}{2s\tilde{M}}\\
			y_1^r &= \frac{\sqrt{s^2\tilde{M}^2(y_1-y_0)^2+2s\tilde{M}(y_1+y_0)+1}+s\tilde{M}(y_1+y_0)+1}{2s\tilde{M}}
		\end{aligned}
	\end{equation}
	
	and the only critical point is at
	
	\[y_c=\frac{s\tilde{M}(y_1+y_0)+1}{2s\tilde{M}},\]
	
	with corresponding extreme value
	
	\[h_{ext}=y_0+\frac{[s\tilde{M}(y_1-y_0)+1]^2}{4s\tilde{M}}.\]
\end{lem}

\label{minus_sign}Here the apparently ``unnatural'' minus sign of the definition \ref{def:GQM} of the fixed points polynomial $P_{f_2}$ will show its usefulness: for \emph{purely aesthetic reasons}, in the definition of the LFFQM, we want that $h_2'(y_0)>0$ whenever $0<y_0<y_1$ and $M>0$ which, without loss of generality, is accomplished by arbitrarily introducing the mentioned sign.

\section{Canonical form}

We will introduce a \emph{canonical form} of the quadratic map by requiring both fixed points in the linear factors form \eqref{def:LFFQM} to be real or, equivalently, $\beta^2-4\alpha\gamma>0$ in the general form \eqref{def:GQM}; also, we will require an ``amplitude'', $M(\lambda)$, equal to unity\footnote{Recall $M(\lambda)$ is the function  preceding the product of the linear factors of the fixed points in the LFFQM defined in Eq. \eqref{def:LFFQM}.}, a fixed point mapped to zero\footnote{This is the reason we require both fixed points to be real, otherwise we would need a complex transformation to map the complex fixed point to zero, which is beyond the scope of this work.} and the other fixed point being mapped to

\begin{equation} \label{eq:x1dependence}
	x_1(\lambda):=s\tilde{M}(y_1-y_0).
\end{equation}

The latter is accomplished by taking one of the following linear transformations

\begin{equation}
	\begin{aligned}
		y=T_0(x) &:= y_0+s\tilde{M}^{-1}x, \\
		y=T_1(x) &:= y_1+s\tilde{M}^{-1}x.
	\end{aligned}
\end{equation}

Without loss of generality, we will take transformation $T_0$ and call it simply $T$. Since $T$ is evidently linear, its inverse can be calculated directly as

\begin{equation}
	\label{eq:invT}
	x=T^{-1}(y)=s\tilde{M}(y-y_0).
\end{equation}

Performing the corresponding calculations on the LFFQM we obtain

\begin{defn}[Canonical Quadratic Map]
	\label{def:CQM}
	The map defined by
	\begin{equation}
		g_2(x,\lambda):=x-x(x-x_1(\lambda)),
	\end{equation}
	where $\lambda\in\mathbb{R}$ is a parameter, is called the \emph{Canonical Quadratic Map (CQM)}.
\end{defn}

We have thus mapped the real fixed points $y_0\mapsto 0$ and $y_1\mapsto x_1$. This form is much simpler than the previous ones and it has the advantage of putting the whole parametric dependence onto the single non-constant fixed point $x_1=x_1(\lambda)$ through \eqref{eq:x1dependence}. Implicit in this dependence is the parametric dependence of the corresponding original roots of the LFFQM \eqref{eq:LinearFactorsForm}, so we can explicitly give this dependence if desired. Moreover, the fact that $T$ is actually a \emph{topological conjugacy} between $h_2$ and $g_2$ will be proved in chapter \ref{cha:Generalization}, by which we can state that \emph{the study of stability and chaos in quadratic maps with real fixed points can be summarized by the study of the Canonical Quadratic Map}, which yields its importance.

With some straightforward calculations, we easily arrive at the following

\begin{lem}
	The Canonical Quadratic Map has the roots
	
	\begin{equation}
		\begin{aligned}
			x_0^r &= 0, \\
			x_1^r &= x_1+1,
		\end{aligned}
	\end{equation}
	
	and has derivatives
	
	\begin{equation}
		\begin{aligned}
			\label{eq:CQM-derivatives}
			g_2'(x,\lambda) &=-2x+x_1+1,\\
			g_2''(x,\lambda) &= -2,
		\end{aligned}
	\end{equation}
	
	so that its critical point is at
	
	\[x_c=\frac{x_1+1}{2}\]
	
	with value
	
	\begin{equation}
		x_m=\frac{(x_1+1)^2}{4}=x_c^2.
		\label{eq:CanonicalMaximum}
	\end{equation}
\end{lem}

We see then that the second root of the canonical map, $x_1^r$, is always one unit to the right of the non-constant fixed point $x_1$. Also, the critical point is always a \emph{maximum} and is located in the middle point between 1 and $x_1$, with its value being simply the square of the point.

\section{Stability and chaos in the canonical quadratic map}
\subsection{Fixed points}

The stability of the fixed points is given by the absolute value of its multiplier, which must be less than one. Making use of the defined eigenvalue function, we have from \eqref{eq:CQM-derivatives} that for the fixed point $x_0=0$,

\begin{equation}
	\label{eq:eigfun-x0-CQM}
	\phi_0(\lambda)=g_2'(0)=x_1(\lambda)+1
\end{equation}

so that it is straightforward to see that if

\begin{equation}
	\label{eq:stb-cond-x0-CQM}
	\tag{SC0}
	-2<x_1(\lambda)<0
\end{equation}

then $x_0=0$ is a stable fixed point. We will call inequality \ref{eq:stb-cond-x0-CQM} the \emph{stability condition} for the fixed point $x_0=0$. We see then that the stability of $x_0$ actually depends on the value of $x_1$. In terms of the original fixed points of the quadratic map, $y_0$ and $y_1$, what counts is the separation between them scaled by the factor $s\tilde{M}$, as can be seen from equation \eqref{eq:x1dependence}.

On the other hand, for $x_1$,

\begin{equation}
	\label{eq:eigfun-x1-CQM}
	\phi_1(\lambda)=g_2'(x_1)=1-x_1
\end{equation}

so that $x_1$ is stable as long as

\begin{equation}
	\label{eq:stb-cond-x1-CQM}
	\tag{SC1}
	0<x_1(\lambda)<2.
\end{equation}

We call the value $b_1=2$ our first \emph{bifurcation value}. Notice that $-b_1=-2$ is also a bifurcation value since it is the limit of the stability condition for $x_0=0$, given by \eqref{eq:stb-cond-x0-CQM}. We summarize this results in the following

\begin{prop}
	Let $x_0=0$ and $x_1$ be the two fixed points of the canonical quadratic map as defined above. Then
	\begin{itemize}
		\item if $-2<x_1<0$, zero will be an asymptotically stable fixed point;
		\item if $0<x_1<2$, $x_1$ will be an asymptotically stable fixed point.
	\end{itemize}
\end{prop}

\subsection{Periodic points of period two}

To determine the period two periodic points we must of course solve

\begin{equation*}
	g_2^2(x)=x.
\end{equation*}

Some basic algebra shows that

\begin{equation*}
	g_2^2(x)=(x_1+1)^2 x-(x_1+1)(x_1+2)x^2+2(x_1+1)x^3-x^4.
\end{equation*}

Therefore, we must solve

\begin{equation*}
	x_1(x_1+2) x-(x_1+1)(x_1+2)x^2+2(x_1+1)x^3-x^4=0.
\end{equation*}

Clearly, we must factor out $x=0$ and $x=x_1$ since we know those are the fixed points of the original map, and we are looking (only) for period two points. Performing the factorization, we get

\begin{equation}
	\label{eq:CQM-per2eq}
	x^2-(x_1+2)x+(x_1+2)=0.
\end{equation}

Solving the latter equation we get to the following

\begin{lem}
	The canonical quadratic map has a unique 2-cycle given by
	\begin{equation*}
		\label{eq:per2-CQM}
		x_0^2=\frac{1}{2}\left[x_1+2+\sqrt{(x_1)^2-4}\right], \quad x_1^2=\frac{1}{2}\left[x_1+2-\sqrt{(x_1)^2-4}\right],
	\end{equation*}
	where the superscript will indicate the period of the periodic point.
\end{lem}
\begin{proof}
	It is straightforward to see that $g_2(x_0^2)=x_1^2$ and $g_2(x_1^2)=x_0^2$, therefore proving that both $x_0^2$ and $x_1^2$ are period 2 periodic points of $g_2$ and the orbit $\{x_0^2,\,x_1^2\}$ is indeed a 2-cycle. The uniqueness comes from the fact that both are algebraic roots of the second degree equation \eqref{eq:CQM-per2eq} and every such period 2 periodic point must be so.
\end{proof}

To analyze the stability of this 2-cycle, we must calculate its own eigenvalue function. The derivative of $g_2^2$ is

\begin{equation*}
	\frac{\partial g_2^2}{\partial x}(x)=(x_1+1)^2-2(x_1+1)(x_1+2)x+6(x_1+1)x^2-4x^3,
\end{equation*}

so that evaluating in $x=x_0^2(\lambda)$ and $x=x_1^2(\lambda)$ we have

\begin{equation}
	\label{eq:eigfun-x02-CQM}
	\left(g^2_2\right)'(x_0^2(\lambda))=\left(g^2_2\right)'(x_1^2(\lambda))=\phi_2(\lambda)=5-x_1(\lambda)^2,
\end{equation}

so that, not surprisingly, both points of the 2-cycle have the same stability criteria. The stability region is then determined by

\[\vert 5-(x_1)^2\vert<1\]

whose solutions lead us to

\begin{lem}
	\label{lem:CQM-stb-2cycle}
	The unique 2-cycle of the canonical quadratic map is asymptotically stable whenever one of the following inequalities is met
	\begin{equation*}
		\begin{aligned}
			-\sqrt{6}	&< x_1 &< -2,\\
			2 &< x_1 &< \sqrt{6},
		\end{aligned}
	\end{equation*}
	where $x_1$ is the nonzero fixed point of the CQM.
\end{lem}

So, we see that the stability region is divided in two in terms of the value of the original fixed point $x_1$. As a reference, $\sqrt{6}\approx2.44949$. We call this value $b_2=\sqrt{6}$ the second bifurcation value of the parameter $\lambda$ for the canonical quadratic map. Notice again that $-b_2=-\sqrt{6}$ is also a bifurcation value.

It is worth noting that in the case of the first inequality of lemma \ref{lem:CQM-stb-2cycle}, it is the zero fixed point that produces a bifurcation, while for the second one, it is $x_1$ that bifurcates to produce the 2-cycle. The latter can be seen from lemma \ref{eq:per2-CQM} since, for $|x_1|<2$, the 2-cycle is not real and it only begins to be so at the bifurcation values $\pm b_1$; when $x_1=-2$, both $x_0^2$ and $x_1^2$ are zero (and equal to the zero fixed point) and thereon separate as the two branches of the square root function (positive and negative). On the other hand, when $x_1=2$, both $x_0^2$ and $x_1^2$ are also 2 in value and thereon separate in the same way as before.

We can now prove that, indeed, what happens at the bifurcation values of $\pm b_1=\pm2$ are period doubling bifurcations.

\begin{prop}
	\label{prop:CQM-per2bif-b1}
	Let\footnote{The dot represents derivative with respect to the parameter $\lambda$ to avoid confusion with the apostrophe representing derivative with respect to $x$.} $x_1(\lambda_0)=-2$ (respectively, $x_1(\lambda_0)=2$) and $\dot{x_1}(\lambda_0)\neq 0$. Then the zero fixed point (respectively, the $x_1$ fixed point) of the CQM undergoes a period doubling bifurcation precisely when $\lambda=\lambda_0$.
\end{prop}
\begin{proof}
	According to theorem \ref{thm:perioddoubling} we must prove that:
	\begin{enumerate}
		\item $g_2(0)=0$ for all $\lambda$ in an interval around $\lambda=-2$. This is trivially true since the zero fixed point does not vary with the parameter.
		\item $g_2'(0)=-1$. Which follows directly from eq. \ref{eq:eigfun-x0-CQM}.
		\item $\left.\frac{\partial\left(g^2_2\right)'}{\partial\lambda}\right|_{\lambda=-2}(0)\neq 0$. From \eqref{eq:eigfun-x02-CQM} we see that this is true since $\dot{x_1}(\lambda)\neq 0$ by hypothesis.
	\end{enumerate}
	The case for the period doubling bifurcation of the other fixed point when $x_1(\lambda)=2$ is analogous.
\end{proof}

The requirement $\dot{x_1}(\lambda)\neq 0$ will be understood more clearly in the examples below. Likewise, we can prove that, in turn, the fixed point $x_0^2$ (or $x_1^2$) of $g_2^2$ ---periodic point of period 2 for $g_2$--- also undergoes a period doubling bifurcation associated with the second bifurcation value when $x_1(\lambda)=\pm b_2=\pm\sqrt{6}$.

\begin{prop}
	\label{prop:CQM-per2bif-b2}
	Let $x_1(\lambda_0)=-\sqrt{6}$ and $\dot{x_1}(\lambda_0)\neq 0$. Then the fixed points $x_0^2$ and $x_1^2$ of the $g_2^2$ undergo a period doubling bifurcation precisely when $\lambda=\lambda_0$.
\end{prop}
\begin{proof}
	Consider first the case for $x_0^2$ given by lemma \ref{eq:per2-CQM}. According to theorem \ref{thm:perioddoubling} we must prove that:
	\begin{enumerate}
		\item $\left(g_2^2\right)'(x_0^2)=-1$. Which follows directly from eq. \ref{eq:eigfun-x02-CQM} since by hypothesis $x_1(\lambda_0)=-\sqrt{6}$.
		\item $g_2^2(x_0^2)=x_0^2$ for all $\lambda$ in an interval around $\lambda=-\sqrt{6}$. Theorem \ref{thm:fixedp} allows us to affirm this since by the above point $\left(g_2^2\right)'(x_0^2)\neq1$.
		\item $\left.\frac{\partial\left(g^4_2\right)'}{\partial\lambda}\right|_{\lambda=\lambda_0}(x_0^2)\neq 0$. Since, by the chain rule,
		\begin{equation*}
			\begin{aligned}
				\left(g_2^4(x_0^2)\right)' 	&= \left(g_2^2\left(g_2^2(x_0^2)\right)\right)'\\
				&= \left(g_2^2\right)'\left(g_2^2(x_0^2)\right)\,\left(g_2^2(x_0^2)\right)'\\
				&= \left(g_2^2(x_0^2)\right)'\left(1-(x_1)^2\right)\\
				&=  \left(1-(x_1)^2\right)^2
			\end{aligned}
		\end{equation*}
		using eq. \eqref{eq:eigfun-x02-CQM} and the fact that $x_0^2$ is a fixed point of $g_2^2$. The partial derivative with respect to $\lambda$ is then
		\begin{equation*}
			\begin{aligned}
				\frac{\partial\left(g^4_2\right)'}{\partial\lambda}(x_0^2(\lambda_0)) &= -4\left(1-(x_1(\lambda_0))^2\right)(x_1(\lambda_0))^2\dot{x_1}(\lambda_0)\\
				&= 120\,\dot{x_1}(\lambda_0)\neq 0
			\end{aligned}
		\end{equation*}
		since $x_1(\lambda_0)=-\sqrt{6}$ and $\dot{x_1}(\lambda_0)\neq0$ by hypothesis.
	\end{enumerate}
	The cases for $x_1^2$ and $\lambda_0=-\sqrt{6}$ are analogous.
\end{proof}

\subsection{Periodic points of higher period}

Some more algebra shows us that

\begin{multline}
	\label{eq:g23}
	g_2^3(x)=x\,\left( {x_1}^{3}+3\,{x_1}^{2}+3\,x_1+1\right) +{x}^{2}\,\left( -{x_1}^{4}-5\,{x_1}^{3}-10\,{x_1}^{2}-9\,x_1-3\right)\\
	+{x}^{3}\,\left( 2\,{x_1}^{4}+10\,{x_1}^{3}+20\,{x_1}^{2}+18\,x_1+6\right) +{x}^{4}\,\left( -{x_1}^{4}-10\,{x_1}^{3}-25\,{x_1}^{2}-25\,x_1-9\right)\\
	+{x}^{5}\,\left( 4\,{x_1}^{3}+18\,{x_1}^{2}+24\,x_1+10\right) +{x}^{6}\,\left( -6\,{x_1}^{2}-14\,x_1-8\right) +{x}^{7}\,\left( 4\,x_1+4\right) -{x}^{8}
\end{multline}

with derivative

\begin{multline}
	\label{eq:d1g23}
	\left(g_2^3(x)\right)'=\left({x_1}^{3}+3\,{x_1}^{2}+3\,x_1+1\right)+x\,\left( -2\,{x_1}^{4}-10\,{x_1}^{3}-20\,{x_1}^{2}-18\,x_1-6\right)\\
	+{x}^{2}\,\left( 6\,{x_1}^{4}+30\,{x_1}^{3}+60\,{x_1}^{2}+54\,x_1+18\right)+{x}^{3}\, \left( -4\,{x_1}^{4}-40\,{x_1}^{3}-100\,{x_1}^{2}-100\,x_1-36\right)\\
	+{x}^{4}\left( 20\,{x_1}^{3}+90\,{x_1}^{2}+120\,x_1+50\right)+{x}^{5}\,\left(-36\,{x_1}^{2}-84\,x_1-48\right)+{x}^{6}\,\left( 28\,x_1+28\right)-8\,{x}^{7}
\end{multline}

So that, in order to find the periodic points of period three we must, in principle, explicitly solve

\begin{multline}
	-x^7+4(x_1 +1)\,x^6-(6\,x_1^2+14\,x_1+8) \,x^5+( 4\,x_1^3+18\,x_1^2+24\,x_1 +10) \,x^4 \\
	-(x_1^4+10\,x_1^3+25\,x_1^2+25\,x_1 +9) \,x^3+( 2\,x_1^4+10\,x_1^3+20\,x_1^2+18\,x_1 +6) \,x^2\\
	-(x_1^4+5\,x_1^3+10\,x_1^2+9\,x_1 +3) \,x+( x_1^3+3\,x_1^2+3\,x_1 +1)=0
\end{multline}

which is an seventh degree polynomial in $x$. Even factoring out the other known fixed point of $g_2$, $x_1$, we are still left with a sixth degree polynomial on $x$ which is, in general, impossible to solve explicitly by algebraic means. Nevertheless, for specific polynomials we can determine the roots numerically.

Much less can be said about determining higher period periodic points. However, we may still approximate the stability regions of these $k$-cycles by numerically determining the bifurcation value of the parameter $\lambda$, which in this case we have done by plotting very high-precision bifurcation diagrams (see Appendix \ref{app:numerical}). One way to determine an approximation to the value of the infinite period bifurcation value, after which the onset of chaos takes place, is through equation \eqref{eq:Bapprox}, which yields a very rough estimate of $b_{\infty}\approx2.57$. Some bifurcation values for the CQM found in this work are shown in table \ref{tab:CQM-BifurcationValues}; moreover, table \ref{tab:CQM-StabilityConditions} shows the the stability conditions for each periodic point,up to period $2^7$; the boundary of the stability bands correspond to the $b_k$ bifurcation values.

\begin{table}
	\begin{tabular}[]{|c|c|c|}
		\hline
		$k$ & $b_k$ \tabularnewline
		\hline
		0 & 0 \tabularnewline
		\hline
		1 & 2 \tabularnewline
		\hline
		2 & $\sqrt{6}$\tabularnewline
		\hline
		3 & $2.5440\pm0.0005$\tabularnewline
		\hline
		4 & $2.5642\pm0.0002$\tabularnewline
		\hline
		5 & $2.56871\pm4\times10^{-5}$\tabularnewline
		\hline
		6 & $2.56966\pm1\times10^{-5}$\tabularnewline
		\hline
		7 & $2.569881\pm5\times10^{-6}$\tabularnewline
		\hline
		$\vdots$ & $\vdots$\tabularnewline
		\hline
		$\infty$ & $\sim2.569941\pm5\times10^{-7}$\tabularnewline
		\hline
	\end{tabular}
	\caption{Bifurcation values for the Canonical Quadratic Map. $b_0=0$ is included only as a reference, although it is not a bifurcation value.}
	\label{tab:CQM-BifurcationValues}
\end{table}

\begin{table}
	\begin{tabular}[]{|c|c|c|}
		\hline
		Period & Periodic Points & Stability Condition \tabularnewline
		\hline
		1 & $x_0^1=0$ & $-b_1< x_1 <b_0$ \tabularnewline
		\hline
		1 & $x_1^1=x_1$ & $b_0<x_1<b_1$ \tabularnewline
		\hline
		2 & $x_0^2=\frac{1}{2}\left[x_1+2+\sqrt{(x_1)^2-4}\right]$ & $b_1<\vert x_1\vert<b_2$ \tabularnewline
		& $x_1^2=\frac{1}{2}\left[x_1+2-\sqrt{(x_1)^2-4}\right]$ & \tabularnewline
		\hline
		4 & $x_0^4,\, x_1^4,\, x_2^4,\, x_3^4$ & $b_2<|x_1|<b_3$ \tabularnewline
		\hline
		8 & $x_0^8,\, x_1^8,\, ..., x_7^8$ & $b_3<|x_1|<b_4$ \tabularnewline
		\hline
		16 & $x_0^{16},\, x_1^{16},\, ..., x_{15}^{16}$ & $b_4<|x_1|<b_5$ \tabularnewline
		\hline
		32 & $x_0^{32},\, x_1^{32},\, ..., x_{31}^{32}$ & $b_5<|x_1|<b_6$ \tabularnewline
		\hline
		64 & $x_0^{64},\, x_1^{64},\, ..., x_{63}^{64}$ & $b_6<|x_1|<b_7$ \tabularnewline
		\hline
		\vdots & \vdots & \vdots \tabularnewline
		\hline
		$\infty$ & - & $\vert x_1\vert> b_{\infty}$ \tabularnewline
		\hline
	\end{tabular}
	\caption{Periodic points and corresponding stability conditions for the canonical quadratic map.}
	\label{tab:CQM-StabilityConditions}
\end{table}

\subsection{Fixed points with multiplicity}

In the CQM, the only way we can have multiplicity in the fixed points is when $x_1(\lambda)=x_0=0$. In this case, $g_2$ takes the form

\begin{equation}
	\label{CQM-multiplicity}
	g_2(x)=x\,(1-x).
\end{equation}

Its derivative is then

\begin{equation}
	g_2'(x)=1-2x,
\end{equation}

and, therefore, $g_2'(0)=1$, so that it is a nonhyperbolic fixed point, according to definition \ref{def:hyperbolic-fp}.

\begin{prop}\label{prop:quadsemi}
	Let $g_2$ be the CQM with a single fixed point with multiplicity of two, as in eq. \eqref{CQM-multiplicity}. Then the single fixed point $x_0=x_1=0$ is semistable from the right.
\end{prop}
\begin{proof}
	It is straightforward to see that $g_2''(x)=-2$ and then $g_2'''(x)=0$ for all $x$. Therefore, in particular, $g_2''(0)=-2\neq0$ and then $g_2'''(0)=0$, so that, by the latter, we cannot use theorem \ref{thm:nonhyperbolic-pos}. However, theorem \ref{thm:semistability} tells us that the fixed point is semistable from the right. Indeed, as $g_2$ is concave downward at $x=0$, then $g_2'$ is decreasing in a small neighborhood about zero and, since $g_2'(0)=1$ and $g_2'$ is continuous, we have that there exists $\delta>0$, so that $g_2'(x)>1$ for $x\in(-\delta,\,0)$ and $g_2'(x)<1$ for $x\in(0,\,\delta)$. Consequently, $x_0=x_1=0$ is unstable from the left and semistable from the right, according to definition \ref{def:semistability} and the proof of theorem \ref{thm:attractor-repellor} \cite{Elaydi}.
\end{proof}

It is worth noting that, although the case of multiplicity in the fixed point of the CQM is possible, the case of \emph{complex} fixed points is not, since the CQM requires one fixed point to be zero, then its complex conjugate ---which should also be a fixed point--- is itself.

\subsection{Regular-reversality and chaos in the canonical quadratic map}

Let us recall, that the definition of $x_1$ transforms, through $T$, the parametric dependence of the original fixed points of the general quadratic map into a single parameter-dependent function $x_1=x_1(\lambda)$. Recall then that,

\[x_1(\lambda)=M(\lambda)\,\left(y_1(\lambda)-y_0(\lambda)\right).\]

Therefore, assuming a constant value of $M(\lambda)=M_0=1$, we see that the value of $x_1$ is exactly the difference between the two original fixed points, $y_1$ and $y_0$, of the General Quadratic Map (which we forced to be real). We can then interpret the found values of the sequence $\{b_k\}_{k\in\mathbb{N}}$ as determining the

\begin{defn}[Stability Bands of a Quadratic Map]
	Let $y_i:\mathcal{A}\subseteq\mathbb{R}\rightarrow\mathbb{R}$, $i\in\{1,\,2\}$ be the two fixed points of the family of quadratic maps $f_{\lambda}$, as given by \ref{def:GQM}, and $\{b_k\}_{k\in\mathbb{N}}$ the sequence of bifurcation values of table \ref{tab:CQM-BifurcationValues}. The union of open intervals
	\begin{equation}
		\left(y_i(\lambda)-b_{k+1},\,y_i(\lambda)-b_k\right)\bigcup\left(y_i(\lambda)+b_k,\,y_i(\lambda)+b_{k+1}\right),\quad\lambda\in\mathcal{A}
	\end{equation}
	is called the $k$-th \emph{stability band} of $y_i$.
\end{defn}

Graphically, we see then that stability bands run along the values of $y_1(\lambda)$ and $y_0(\lambda)$ as functions of the parameter and, as long as the value of one fixed point does not cross a the limit of a band of the other fixed point, there will be no changes in the periodic points structure of the map, i.e. there will only be bifurcations when the value of one fixed point crosses the limit of a stability band of the other fixed point. If $M$ were not held constant, it would simply act as a ``modulator'' for the width of the bands along the parameter $\lambda$.

Stability bands for $x_0=0$ and $x_1(\lambda)$ in the CQM are defined analogously.

\begin{defn}[Stability Bands of the CQM]
	Let $x_1:\mathcal{A}\subseteq\mathbb{R}\rightarrow\mathbb{R}$ be the nonzero fixed point of the family of quadratic maps $g_2$, as given by definition \ref{def:CQM}, and $\{b_k\}_{k\in\mathbb{N}}$ the sequence of bifurcation values of table \ref{tab:CQM-BifurcationValues}. The union of open intervals
	\begin{equation}
		\left(-b_{k+1},\,-b_k\right)\bigcup\left(b_k,\,b_{k+1}\right),\quad\lambda\in\mathcal{A}
	\end{equation}
	is called the $k$-th \emph{stability band} of the CQM.
\end{defn}

Notice that that the part of the $k$-th stability band in the upper (respectively, lower) semi-plane corresponds to the stability condition of period $2^k$ attracting periodic points associated with period doubling bifurcation of the $x_1$ (respectively, zero) fixed point.

Finally, we must remark that the width of these stability bands ---given by the values of the sequence $\{b_k\}_{k\in\mathbb{N}}$--- is not equal for polynomial maps of different degrees, as we shall see in the next chapters.

By the above, the big picture is that depending on the absolute value of $x_1(\lambda)$ the one-parameter family of canonical quadratic maps can present or not attracting periodic points of different periods or even become chaotic. In particular, as long as $\vert x_1(\lambda)\vert<b_N$ we will only have attracting periodic points of period less or equal to $2^N$. We can thus restate the propositions \ref{prop:bk-perpt} and \ref{prop:not-chaotic} given in section \ref{sec:types-perpt-chaos} of chapter \ref{cha:Quadratic}, originally by \cite{Solis2004} (\cite{Solis2004}).

\begin{prop}
	\label{prop:CQM-boundedmax1}
	If the absolute value of the fixed point $x_1^*(\lambda)$ is bounded from above by $b_{N}$, the GQM map $x_{n+1}=g_2(x_n)\equiv x_n - x_n (x_n-x_1^*(\lambda))$ is not chaotic and can only have periodic points of period $2^{m}$ with $m<N$. Moreover, if the bound is given by $b_{\infty}$, the system is still not chaotic and it can only have periodic points of periods of powers of two.
\end{prop}

And more specifically,

\begin{prop}
	\label{prop:CQM-boundedmax2}
	Suppose that the absolute value of the fixed point $x_1^*(\lambda)$ is bounded. If the supreme of the function $|x_1^*(\lambda)|$ lies within the interval $(b_{n},\, b_{n+1})$, then the map $x_{n+1}=g_2(x_{n})\equiv x_n - x_n (x_n-x_1^*(\lambda))$ only has periodic points of period $2^{k}$ with $k\in\{1,\,2,\,...,\, n\}$ and therefore, it is \emph{not} chaotic.
\end{prop}

In order for the system $x_{n+1}=x_n-x_n\,\left(x_n-x_1(\lambda)\right)$ to be a regular reversal map for $\lambda$ in some interval $\mathcal{A}=(\lambda_I,\,\lambda_F)$, it is necessary that the continuous function $x_1(\lambda)$ has the following property: there exist $\lambda_{1}<\lambda_{2}<\lambda_{3}$ in $\mathcal{A}$ such that $x_1(\lambda_{1})<x_1(\lambda_{3})=b_{i}<x_1(\lambda_{2})$ for some $i\in\mathbb{N}$ where $b_{i}$ is the $i$-th element of the sequence of bifurcation values $\{b_{k}\}_{k=0}^{\infty}$. It is easy to construct functions $x_1(\lambda)$ with the aforementioned property, which means that constructing regular-reversal canonical quadratic maps is also easy. Furthermore, knowing beforehand the properties of a desired bifurcation diagram, i.e. the fixed points, the $\lambda$ values for desired bifurcations, etc. we can construct $x_1(\lambda)$ functions that accomplish the desired properties in the bifurcation diagram.

\section{Quadratic examples with the canonical map}

Recall that the CQM is

\[g_2(x,\lambda)=x-x\left(x-x_1(\lambda)\right),\]

where
\begin{equation}
	\label{eq:CQM-x1def}
	x_1(\lambda)=M(\lambda)\left(y_1(\lambda)-y_0(\lambda)\right).
\end{equation}

Therefore, we can ``parametrize'' the map $g_2(x,\lambda)$ by defining either $x_1(\lambda)$ directly or going back to the parametric dependence of the original LFFQM $h_2(x,\lambda)$ from which $x_1$ was defined using the transformation $T$, i.e. defining $M(\lambda)$, $y_0(\lambda)$ and $y_1(\lambda)$. We will look at both interpretations at the same time in the following examples.

\subsection{Linear fixed points}
\label{ex:linearfps-CQM}

We will begin with the simplest case: linearly varying $y_0(\lambda)$ and $y_1(\lambda)$ and $M(\lambda)=1$. Consider,

\[y_0(\lambda):=\lambda, \quad y_1(\lambda):=2\,\lambda.\]

The graphs of these parametrizations of the fixed points are shown in figure \ref{fig:liny0liny1}. Clearly, $x_1(\lambda)=\lambda$ and its graph is also shown in figure \ref{fig:liny0liny1}.

\begin{figure}
	\centering
	\subfloat[Graphs of linear $y_0(\lambda)$ and $y_1(\lambda)$ along the selected values of $\lambda$. Also, the stability conditions (constrained by the bifurcation values $b_1$, $b_2$ and $b_{\infty}$) of $y_0$ are shown around it.]
	{
		\centering
		\includegraphics[width=0.45\columnwidth]{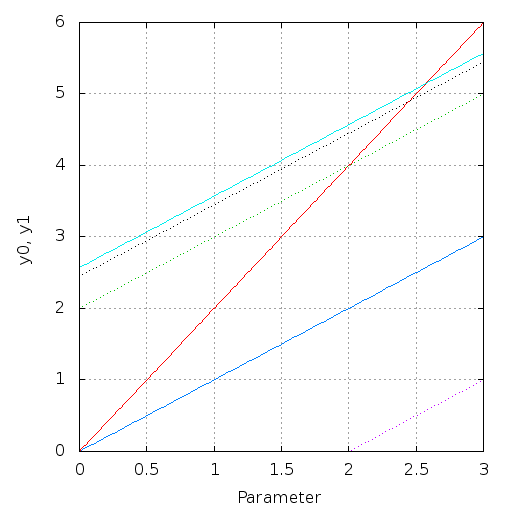}
	}
	\subfloat[Graph of corresponding $x_1(\lambda)$ along with the bifurcation values $b_1$, $b_2$ and $b_{\infty}$.]
	{
		\centering
		\includegraphics[width=0.45\columnwidth]{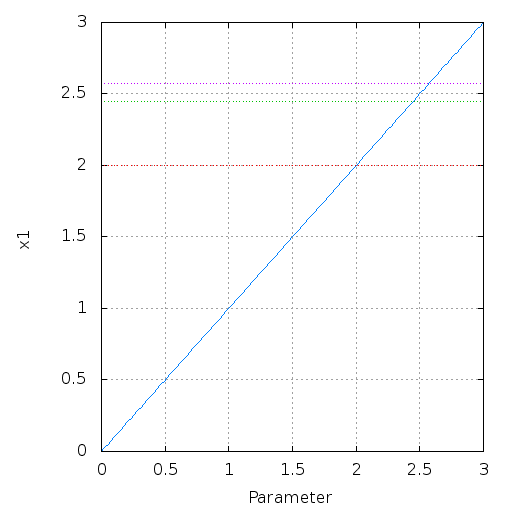}
	}
	\caption{Linear parametrization of fixed points (example \ref{ex:linearfps-CQM}).}
	\label{fig:liny0liny1}
\end{figure}

The bifurcation diagram for this specific parametrization of the fixed points of the CQM is shown in figure \ref{fig:dbif_liny0liny1}. In this bifurcation diagram we can see that it is precisely at the values of $\lambda$ for which $y_1$ crosses limits of the stability conditions of $y_0$ (or, equivalently, where $x_1$ crosses the bifurcation values) that bifurcations take place. Notice that, since $x_1>0$, it is $x_1$ the fixed point that undergoes period doubling bifurcation.

\begin{figure}
	\centering
	\includegraphics[width=0.67\columnwidth]{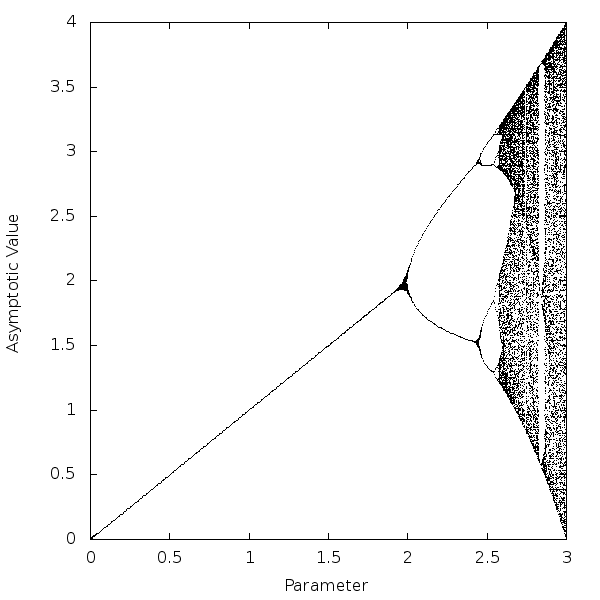}
	\caption{Bifurcation diagram for the linear parametrization of the fixed points example.}
	\label{fig:dbif_liny0liny1}
\end{figure}

In fact, it does not matter what the actual values of the slopes of $y_0(\lambda)$ and $y_1(\lambda)$ are: as long as their difference is the same (has the same slope) they will produce exactly the same bifurcation diagram, since the parametric dependence of $x_1(\lambda)$ will be unchanged (see eq. \ref{eq:CQM-x1def}). Making precise calculations of the bifurcations diagram near $\lambda$-values of interest we can determine numerically the rest (or at least some more) of the bifurcation values forming the sequence $\{b_k\}_{k\in\mathbb{N}}$, but this is best accomplished using the parametrization of example \ref{ex:expfp-CQM} below.

\FloatBarrier

\subsection{Quadratic fixed point}
\label{ex:quadfp-CQM}

Redefining the parametric dependencies of $y_0$ and $y_1$ we can achieve \emph{regular-reversal maps} as described in definition \ref{def:regular-reversal} in page \pageref{def:regular-reversal}. For example, leaving $y_0$ in the above example \ref{ex:linearfps-CQM}, but redefining $y_1$ with a quadratic parametric dependence on $\lambda$ as

\begin{equation}
	\begin{aligned}
		y_0(\lambda) &:=\lambda\\
		y_1(\lambda) &:=\lambda-p_1\lambda(\lambda-p_2),\quad p_1,\,p_2\in\mathbb{R}.
	\end{aligned}
	\label{eq:QuadraticDependece}
\end{equation}

We immediately see that zero is a root of the parametric form of $y_1$, and the other root is $p_2$. Since it is exactly like the defined \emph{Linear Factors Form} defined for the quadratic map in subsection \ref{sub:LinearFactorsForm} on page \pageref{sub:LinearFactorsForm}, we know that its maximum lies at

\[\lambda_c=\frac{p_1\,p_2+1}{2\,p_1},\]

with corresponding value

\[y_1^{max}=\lambda_c^2=\frac{(p_1\,p_2+1)^2}{4\,p_1^2}.\]

Using, for example, the values, $p_1=1$ and $p_2=3$ we know then that $\lambda_c=2$ and, therefore, $y_1^{max}=4$ (this functions are shown in figure \ref{fig:liny0quady1}a. As for $x_1$ it is then

\[x_1(\lambda)=\lambda (3 - \lambda),\]

where the maximum is clearly located at $\lambda=3/2$ with a value of $x_1^{max}=9/4$. The graph of $x_1$ is shown in figure \ref{fig:liny0liny1}b. Also, $x_1$ crosses the first bifurcation value ($b_1=2$) two times at $\lambda=1$ and $\lambda=2$, and it is easy to see that $\dot{x_1}(\lambda)\neq0$ for those values of $\lambda$, so that by proposition \ref{prop:CQM-per2bif-b1}, we expect period doubling bifurcations at this values. Also, since $x_1^{max}=2.25<2.45\approx b_2$, according to proposition \ref{prop:CQM-boundedmax2} we know that we will only get one such bifurcation leading to stable 2-cycles in the interval $(1,\,2)$, and then going back to a unique stable fixed point (i.e. $x_1$ stays in the 1-stability band in the interval $(1,\,2)$). Thus, this example complies with our definition of a \emph{regular-reversal map}. The latter predictions are confirmed in the bifurcation diagram shown in figure \ref{fig:dbif_liny0quady1}.

\begin{figure}
	\centering
	\subfloat[Graphs of linear $y_0(\lambda)$ and quadratic $y_1(\lambda)$ along the selected values of $\lambda$. Also, the stability bands ($b_1$, $b_2$ and $b_{\infty}$) of $y_0$ are shown around it.]
	{
		\centering
		\includegraphics[width=0.45\columnwidth]{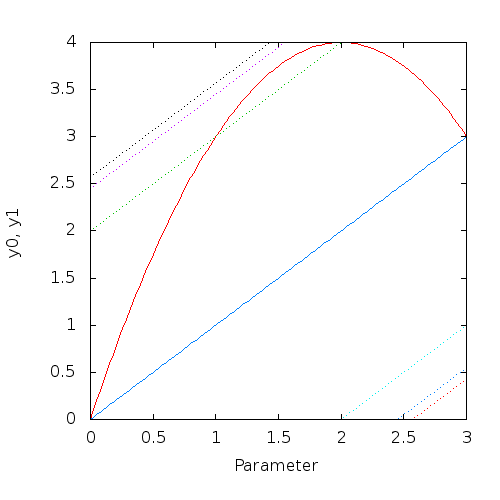}
	}
	\subfloat[Graph of corresponding $x_1(\lambda)$ along with its stability bands.]
	{
		\centering
		\includegraphics[width=0.45\columnwidth]{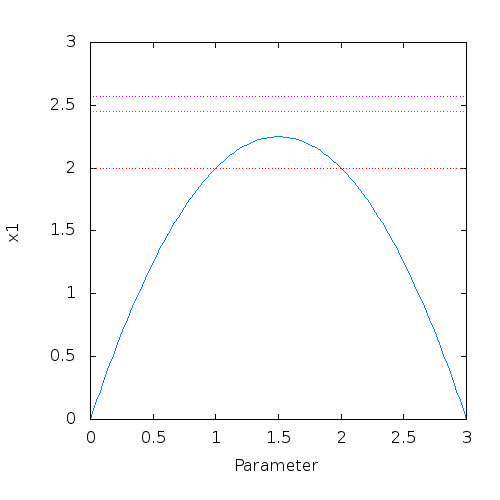}
	}
	\caption{Quadratic parametrization of fixed points with $p_1=1$ and $p_2=3$.}
	\label{fig:liny0quady1}
\end{figure}

\begin{figure}
	\centering
	\includegraphics[width=0.45\columnwidth]{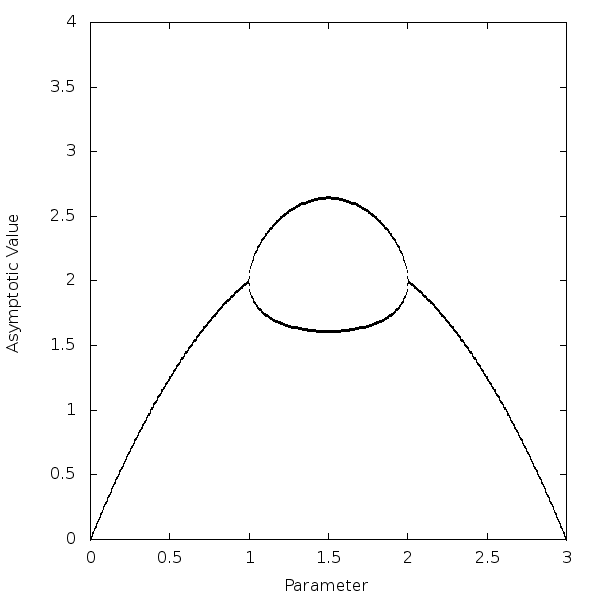}
	\caption{Bifurcation diagram for the quadratic parametrization of the fixed points with $p_1=1$ and $p_2=3$..}
	\label{fig:dbif_liny0quady1}
\end{figure}

Now, increasing the maximum value of $x_1$ we can accomplish more bifurcations in the diagram. This can be achieved by increasing the slope of $y_1$ to $p_1=1.1422$, since for this value $x_0^{max}\approx b_{\infty}$. The corresponding graph of $x_1$ is shown in figure \ref{fig:liny0quady1b}. And, also, the respective bifurcation diagram is shown in figure \ref{fig:dbif_liny0quady1b}, where we can see that the system undergoes a cascade of period doubling bifurcations in the central part. We thus produce another regular-reversal map with more nested regions of different types.

\begin{figure}
	\centering
	\subfloat[Graph of corresponding $x_1$. In this case $x_1^{max}\approx b_{\infty}$. Also, the stability bands ($b_1$, $b_2$ and $b_{\infty}$) of $x_0$ are shown above it.]
	{
		\centering
		\includegraphics[width=0.45\columnwidth]{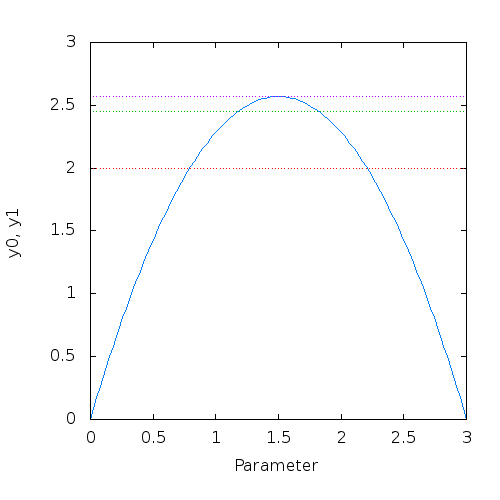}
		\label{fig:liny0quady1b}
	}
	\subfloat[Bifurcation diagram. We have set here $x_1^{max}\approx b_{\infty}$.]
	{
		\centering
		\includegraphics[width=0.45\columnwidth]{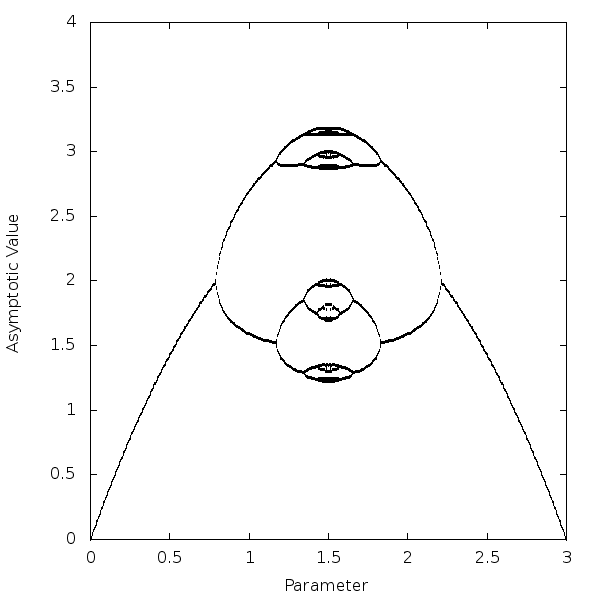}
		\label{fig:dbif_liny0quady1b}
	}
	\caption{Quadratic parametrization of the fixed points with $p_1=1.1422$ and $p_2=3$.}
\end{figure}

If we continue increasing the maximum value of $x_1$ we can achieve a chaotic region within the regular-reversal map. Any value above the used last above will suffice. Say we take $p_1=1.2$. The graph of $x_1$ for this case is shown in figure \ref{fig:liny0quady1c} and we then get the bifurcation diagram from figure \ref{fig:dbif_liny0quady1c}, which displays a chaotic region in the central part, as stated. This is still a regular-reversal map.

\begin{figure}
	\centering
	\includegraphics[width=0.45\columnwidth]{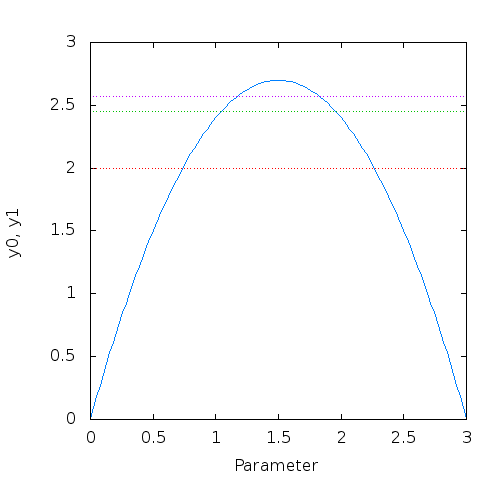}
	\caption{Graph of corresponding $x_1$ for linear $y_0(\lambda)$ and quadratic $y_1(\lambda)$ with $p_1=1.2$ and $p_2=3$. The maximum value of $x_1(\lambda)$ is set to be slightly above $b_{\infty}$. Also, the bifurcation values $b_1$, $b_2$ and $b_{\infty}$ are shown.}
	\label{fig:liny0quady1c}
\end{figure}

\begin{figure}
	\centering
	\includegraphics[width=0.67\columnwidth]{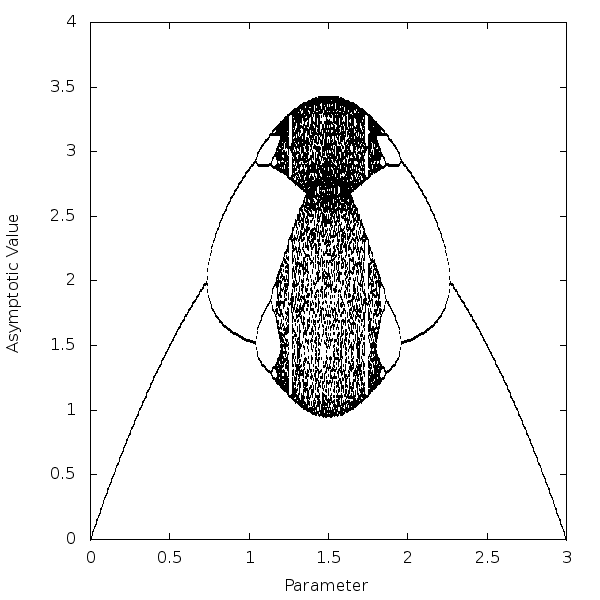}
	\caption{Bifurcation diagram for the quadratic parametrization of the fixed points with $p_1=1.2$ and $p_2=3$. We have set here $x_1^{max}>b_{\infty}$.}
	\label{fig:dbif_liny0quady1c}
\end{figure}

\FloatBarrier

\subsection{Sinusoidal fixed point}
\label{ex:sinfp-CQM}

We can construct a periodic regular-reversal map by making one of the fixed points have a sinusoidal parametric dependence on $\lambda$. That is, we leave $y_0$ as before and define

\begin{equation}
	y_1(\lambda):= \lambda + a\,(1+\sin{\pi\,\lambda}),
	\label{eq:y1sinusoidal}
\end{equation}

where we can tweak the sinusoidal form of $y_1$ by means of the constant $a$ which is just an amplitude (and translation, at the same time). Choosing $a=1.3\gtrsim \frac{b_{\infty}}{2}$ we can get a periodic regular-reversal map with chaotic regions in the periodic central regions. The graphs of the roots are shown in figure \ref{fig:liny0siny1} and the bifurcation diagram is shown in figure \ref{fig:dbif_liny0siny1}.

\begin{figure}
	\centering
	\subfloat[Graphs of linear $y_0(\lambda)$ and sinusoidal $y_1(\lambda)$ along the selected values of $\lambda$. Also, the stability bands ($b_1$, $b_2$ and $b_{\infty}$) of $y_0$ are shown around it.]
	{
		\centering
		\includegraphics[width=0.45\columnwidth]{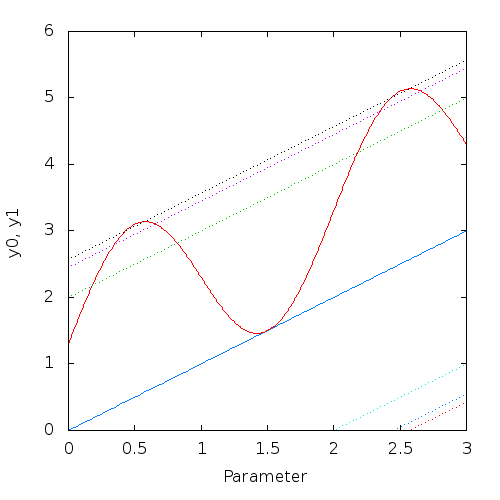}
	}
	\subfloat[Graph of corresponding $x_1(\lambda)$ along with its stability bands.]
	{
		\centering
		\includegraphics[width=0.45\columnwidth]{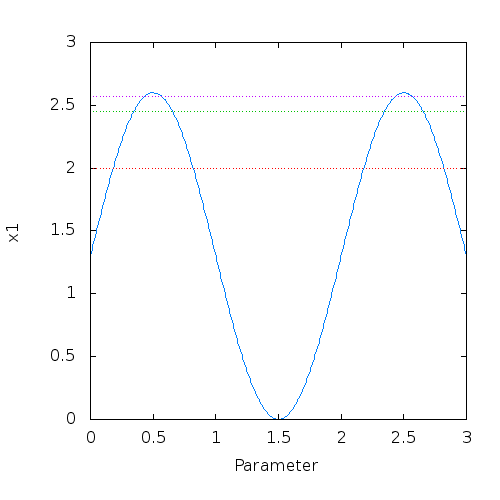}
	}
	\caption{Sinusoidal parametrization of fixed points.}
	\label{fig:liny0siny1}
\end{figure}

\begin{figure}
	\centering
	\includegraphics[width=0.45\columnwidth]{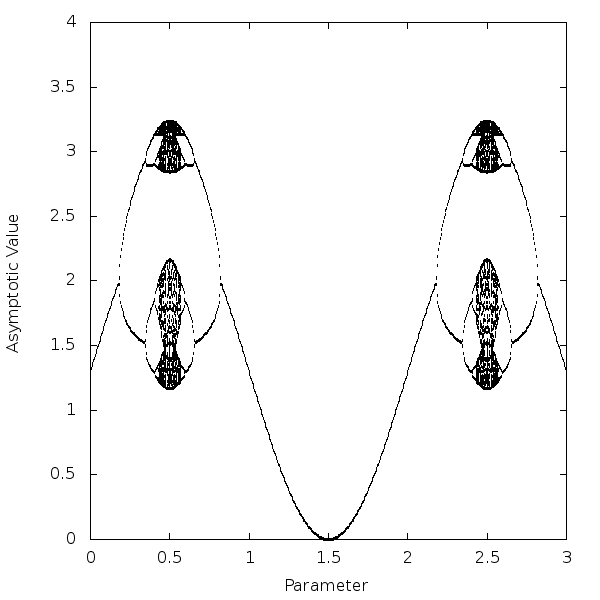}
	\caption{Bifurcation diagram for the sinusoidal parametrization of the fixed points example.}
	\label{fig:dbif_liny0siny1}
\end{figure}

\FloatBarrier

\subsection{When zero loses its stability}

It is important to point out that if $x_1$ becomes negative, i.e. the difference between $y_1$ and $y_0$ becomes negative (assuming again $M=1$), then we enter the stability region of the fixed point at zero, $x_0$,  in the canonical quadratic map (and $y_0$ in the general quadratic map). Thus decreasing values of $x_1$ will cause bifurcations to appear this time from the zero fixed point, as the stability bands are crossed by $x_1$ (correspondingly, $y_1$ in the general quadratic map). We can achieve the latter by defining again a linear $y_1$ as

\[y_1(\lambda):= p_1 + p_2\,\lambda\]

with, say $p_1=3$ and $p_2=-1$. Then we get $y_1$ to cross through \emph{all} the stability bands of $y_0$ in the interval $(0,\,3)$ of the parameter $\lambda$. This is shown graphically in figure \ref{fig:liny0liny1d} along with its corresponding graph of $x_1$.

The bifurcation diagram in this case starts, from left to right, being chaotic, then reversal, then switching to stability of $x_0$ (correspondingly, $y_0$) and then being regular until chaos is reached again, as $x_1$ crosses the stability bands. This is shown in figure \ref{fig:dbif_liny0liny1d}.
\begin{figure}
	\centering
	\subfloat[Graphs of linear $y_0(\lambda)$ and linear $y_1(\lambda)$ crossing all stability bands.]
	{
		\centering
		\includegraphics[width=0.45\columnwidth]{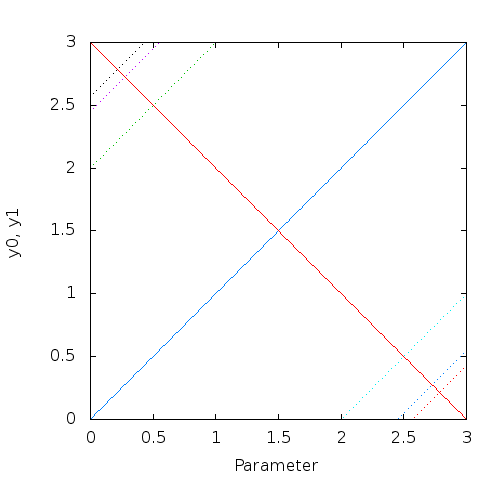}
	}
	\subfloat[Graph of corresponding $x_1(\lambda)$ crossing all stability bands.]
	{
		\centering
		\includegraphics[width=0.45\columnwidth]{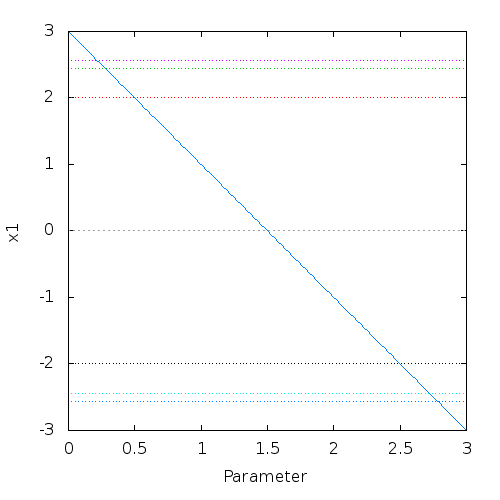}
	}
	\caption{Linear parametrization of fixed points crossing through all stability bands.}
	\label{fig:liny0liny1d}
\end{figure}

\begin{figure}
	\centering
	\includegraphics[width=0.67\columnwidth]{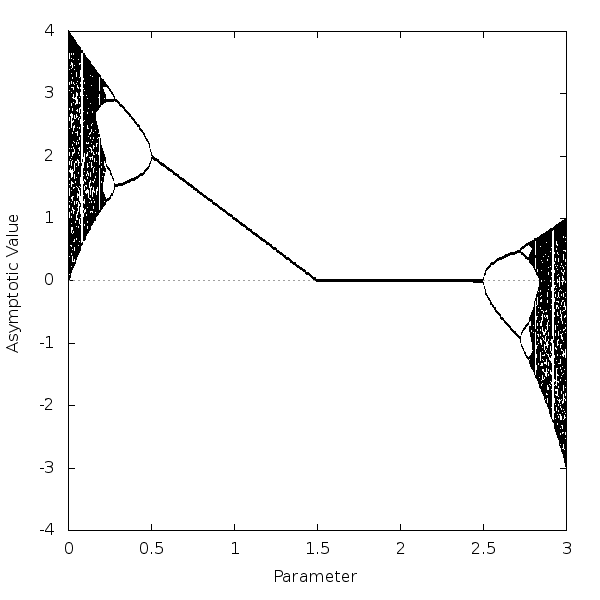}
	\caption{Bifurcation diagram for the linear parametrization of the fixed points example that crosses through all stability bands.}
	\label{fig:dbif_liny0liny1d}
\end{figure}

\FloatBarrier

\subsection{Asymptotic fixed point}
\label{ex:expfp-CQM}

Since we know that the onset of chaos occurs at $b_{\infty}\approx 2.57$ (as calculated through approximation formula \ref{eq:Bapprox}), we can force the moving $x_1$ fixed point of the CQM to asymptotically approach this value in order to visualize more clearly the bifurcation process that takes place very near this value. Now, for symmetry reasons, it will be best if we make the $x_0=0$ fixed point to be the one that bifurcates, instead of $x_1$; we do this by making $x_1\rightarrow -b_{\infty}$ (remember $(-b_{\infty},\,0)$ is the interval where the period doubling bifurcations cascade occurs for $x_0=0$). So, we define

\begin{equation}
	x_1(\lambda):= (-1)^{p_1} b_{\infty} + (-1)^{p_1+1} \exp(-p_2\,\lambda),
	\label{eq:x1exp}
\end{equation}

where $p_1\in\{0,\,1\}$ and $p_2>0$. We choose for this example, for the above reasons, $p_1=p_2=1$.  The corresponding bifurcation diagram for this case is shown in figure \ref{fig:dbif_expx1_pdf-fps}, upper panel, where we can clearly see the bifurcations take place up to the attracting periodic point of period $2^6$ (64-cycle); the lower panel in the same figure shows the value of the fixed points, which shows the asymptotically varying $x_1$. The scale of the parameter value is the same in both panels to allow for comparison; notice that, when the fixed point curve crosses the marked $b_k$ values, period doubling bifurcation takes place in the bifurcation diagram, as expected. By making precise calculations based on this diagrams we could determine the bifurcation values up to an arbitrary $2^N$-cycle, as we did for up to $N=7$ in this work.

\begin{figure}
	\centering
	\includegraphics[width=0.9\columnwidth]{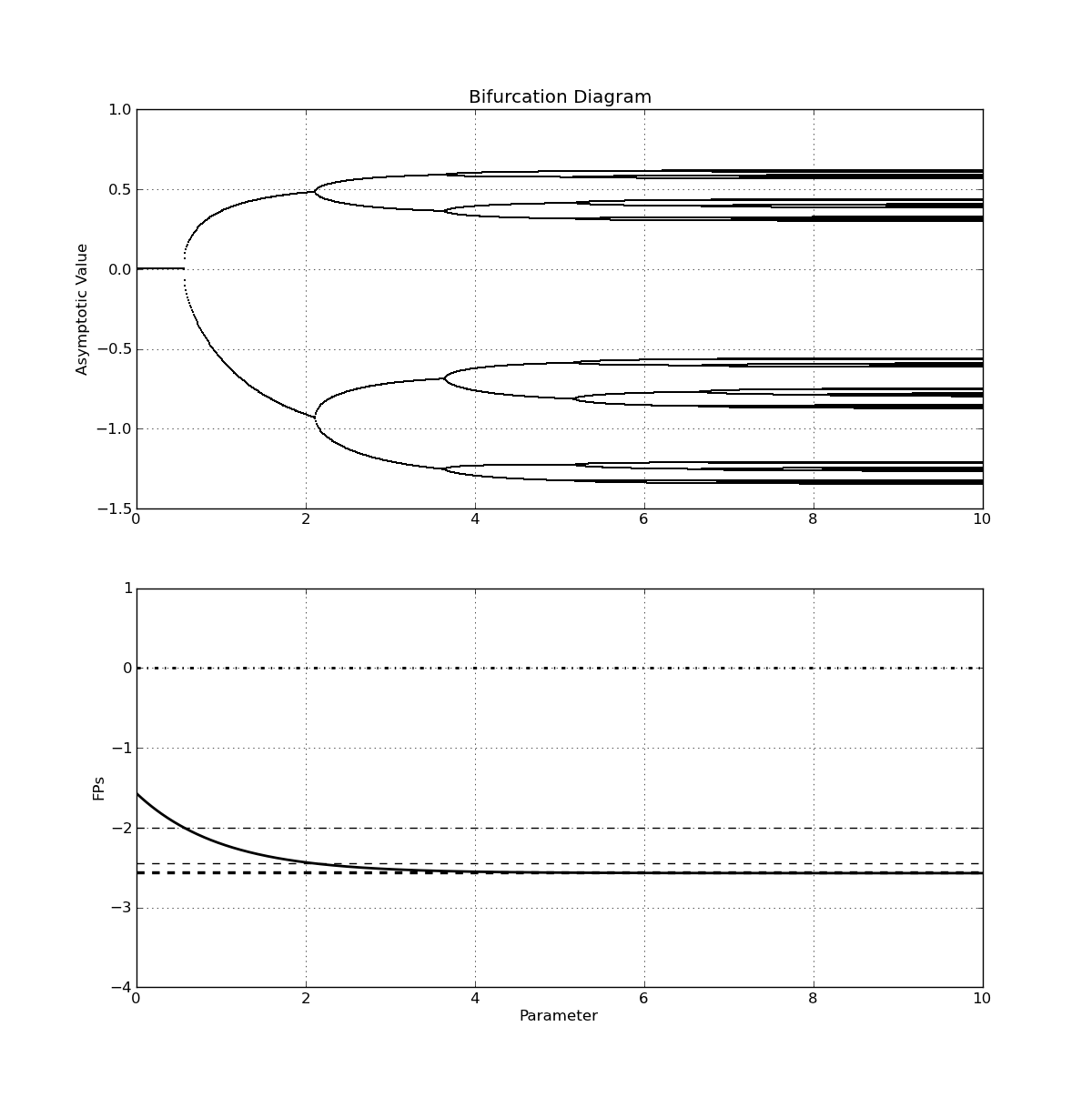}
	\caption{Bifurcation diagram (upper) and fixed points plot with stability bands for the exponential parametrization of the fixed point example (lower). In the lower panel, we have set $x_1\rightarrow -b_{\infty}$ (continuous thick line), the zero fixed point is marked by a dash-dotted line and the bifurcation values $-b_1,\,-b_2,...,\,-b_5$ are also shown as dashed lines.}
	\label{fig:dbif_expx1_pdf-fps}
\end{figure}

\chapter{Cubic maps}
\label{cha:Cubic}

As stated before, in chapter \ref{cha:QuadRev} we furthered the results of chapter \ref{cha:Quadratic}, mainly based on the work of \cite{Solis2004}, and restated them in the frame of a new formulation that would allow for generalization. Well, the first such generalization comes precisely for third degree polynomial maps, i.e. cubic polynomials.

\section{General Cubic Map}
\label{sec:GeneralCubicMap}

Consider a cubic map (iteration function) in its general form, as stated by the

\begin{defn}[General Cubic Map]
	The \emph{General Cubic Map} is defined by
	\begin{equation}
		f_3(y):=y+P_{f_3}(y),
		\label{eq:GCM}
	\end{equation}
	where
	\begin{equation}
		\label{eq:GCM-fppoly}
		P_{f_3}(y)=\alpha + \beta\,y + \gamma\,y^2 + \delta\,y^3
	\end{equation}
	is called the \emph{Fixed Points Polynomial} of $f_3$. All the coefficients $\alpha,\,\beta,\,\gamma$ and $\delta$ are functions of the parameter $\lambda$.
\end{defn}

It is evident that any cubic map can be put in this form by adjusting the corresponding values of the coefficients in the fixed points polynomial. By the fundamental theorem of algebra, we know that \eqref{eq:GCM-fppoly} has three roots, by which the GCM has three fixed points. The roots of $P_{f_3}$ are then \cite{MathHandbook}

\begin{equation}
	\begin{aligned}
		y_0 &= S+T-\frac{1}{3}\tilde{\gamma}\\
		y_1 &= -\frac{1}{2}(S+T)-\frac{1}{3}\tilde{\gamma}+\frac{1}{2}i\,\sqrt{3}(S-T)\\
		y_2 &= -\frac{1}{2}(S+T)-\frac{1}{3}\tilde{\gamma}-\frac{1}{2}i\,\sqrt{3}(S-T)
	\end{aligned}
\end{equation}

where $\tilde{\alpha}=\alpha/\delta$, $\tilde{\beta}=\beta/\delta$, $\tilde{\gamma}=\gamma/\delta$, and

\begin{equation}
	\begin{aligned}
		Q&=\frac{3\tilde{\beta}-\tilde{\gamma}^2}{9}, & R&=\frac{9\tilde{\beta}\tilde{\gamma}-27\tilde{\alpha}-2\tilde{\gamma}^3}{54},\\
		D&=Q^3+R^2, & &\\
		S&=\sqrt[3]{R+\sqrt{D}}, &  T&=\sqrt[3]{R-\sqrt{D}}.
	\end{aligned}
\end{equation}

The coefficients of eq. \eqref{eq:GCM-fppoly} and its fixed points are related by

\begin{equation}
	\begin{aligned}
		y_0+y_1+y_2&=-\tilde{\gamma}, & y_0y_1+y_1y_2+y_2y_0&=\tilde{\beta}, & y_0y_1y_2&=-\tilde{\alpha}.
	\end{aligned}
\end{equation}

$D$ is called the discriminant and we have three cases:

\begin{itemize}
	\item If $D>0$ then one fixed point is real and the other two are complex conjugates.
	\item If $D=0$ then the three fixed points are real with at least two of them equal.
	\item If $D<0$ then all fixed points are real and distinct.
\end{itemize}

It is the last two cases (real fixed points) that will interest us most for the time being. Suppose in particular that $D<0$. Then, we can write \cite{MathHandbook}

\begin{equation}
	\begin{aligned}
		y_0 &= 2\sqrt{-Q}\,\cos\left(\frac{\theta}{3}\right)-\frac{1}{3}\tilde{\beta},\\
		y_1 &= 2\sqrt{-Q}\,\cos\left(\frac{\theta+\pi}{3}\right)-\frac{1}{3}\tilde{\beta},\\
		y_2 &= 2\sqrt{-Q}\,\cos\left(\frac{\theta+2\pi}{3}\right)-\frac{1}{3}\tilde{\beta},\\
	\end{aligned}
\end{equation}

where $\cos\theta=R/\sqrt{-Q^3}$.

\section{Linear Factors Product Form}
\label{sec:LinearFactorsCM}

Since we know that \eqref{eq:GCM-fppoly} has three roots, by the factor theorem we can rewrite $P_{f_3}$ as

\begin{equation}
	\begin{aligned}
		P_{f_3}(y) 	&= M(y-y_0)(y-y_1)(y-y_2),\\
		&= s\tilde{M}(y-y_0)(y-y_1)(y-y_2)
	\end{aligned}
\end{equation}

where $s=\mathrm{sign}(M)$, $\tilde{M}=|M|$; and then rewrite \eqref{eq:GCM} as

\begin{defn}[Linear Factors Form of the Cubic Map]
	\label{def:LFFCM}
	Let $f_3$ be a general cubic map with three fixed points, $y_0$, $y_1$ and $y_2\in\mathbb{C}$. We can write $f_3$ as
	\begin{equation}
		\label{eq:LinearFactorsCM}
		\begin{aligned}
			h_3(y)	&= y + M(y-y_0)(y-y_1)(y-y_2)\\
			&= y + s\tilde{M}(y-y_0)(y-y_1)(y-y_2).
		\end{aligned}
	\end{equation}
	where all $M,\,y_0,\,y_1$ and $y_2$ are functions of the parameter $\lambda$. We call $h_3$ the \emph{Linear Factors Form of the cubic map (LFFCM)}.
\end{defn}

The LFFCM has the derivatives

\begin{equation}
	\begin{aligned}
		h'(y) &= 1 + M \left[ (y-y_1)(y-y_2) + (y-y_0)(y-y_2) + (y-y_0)(y-y_1) \right],\\
		h''(y) &= -2\,M\,\left( y_0 + y_1 + y_2 -3\,y\right).
	\end{aligned}
\end{equation}

By equating the first derivative to zero we can obtain the two critical points

\begin{equation}
	\begin{aligned}
		y_1^c &=\frac{1}{3}\left[(y_0+y_1+y_2)+\sqrt{\frac{1}{2}\left[(y_1-y_0)^2+(y_2-y_1)^2+(y_2-y_0)^2\right]-3/M}\right],\\
		y_1^c &=\frac{1}{3}\left[(y_0+y_1+y_2)-\sqrt{\frac{1}{2}\left[(y_1-y_0)^2+(y_2-y_1)^2+(y_2-y_0)^2\right]-3/M}\right]
	\end{aligned}
\end{equation}

and, evaluating these critical points in $h''$ we have

\begin{equation}
	\begin{aligned}
		h''(y_1^c) &= 2M\sqrt{\frac{1}{2}\left[(y_1-y_0)^2+(y_2-y_1)^2+(y_2-y_0)^2\right]-3/M},\\
		h''(y_2^c) &= -2M\sqrt{\frac{1}{2}\left[(y_1-y_0)^2+(y_2-y_1)^2+(y_2-y_0)^2\right]-3/M}= -h''(y_1^c),
	\end{aligned}
\end{equation}

so that necessarily one critical point is a maximum and the other one is a minimum, unless they are both equal, in which case we have a saddle point.

\section{Canonical Cubic Map}
\label{sec:CanonicalCubicMap}

As in the case for the Canonical Quadratic Map, we will apply a linear transformation to \eqref{eq:LinearFactorsCM} so that one fixed point is mapped to zero and the ``amplitude'' of the linear factors term is unity. We know that, since $f_3$ is cubic, at least one fixed point is real, so we can map this fixed point to zero. This can be accomplished by one of the following transformations

\begin{equation}
	\begin{aligned}
		T_0(x)&=y_0\pm s\tilde{M}^{-1/2}x, & T_1(x)&=y_1\pm s\tilde{M}^{-1/2}x, & T_2(x)&=y_2\pm s\tilde{M}^{-1/2}x,
	\end{aligned}
\end{equation}

by taking $y=T_k(x),\,k\in\{0,\,1,\,2\}$ if $y_k$ is real. Without loss of generality, we will assume $y_0$ is real and apply $T_0$ with the plus sign calling it simply $T$, so that, once the calculations are done, we get

\begin{defn}[Canonical Cubic Map]
	\label{def:CCM}
	The \emph{Canonical Cubic Map (CCM)} is defined by
	\begin{equation}
		g_3(x;\lambda)=x+sx(x-x_1(\lambda))(x-x_2(\lambda)),
		\label{eq:CanonicalCubicMap}
	\end{equation}
	where it has been stressed out that both fixed points $x_1$ and $x_2$ depend upon the parameter $\lambda$.
\end{defn}

So that, if $M>0$ then

\begin{equation}
	g_3(x)=x+x(x-x_1)(x-x_2),
\end{equation}

and if $M<0$

\begin{equation}
	g_3(x)=x-x(x-x_1)(x-x_2),
\end{equation}

From the easily verifiable calculations, it can be shown that the relation between the roots of the linear factors form of the cubic map and the canonical cubic map is given by

\begin{lem}
	The fixed points of the linear factors form of the cubic map and the canonical cubic map are related by
	\begin{equation}
		\begin{aligned}
			x_1(\lambda)&=s\tilde{M}^{1/2}\left[y_1(\lambda)-y_0(\lambda)\right],
			& x_2(\lambda)&=s\tilde{M}^{1/2}\left[y_2(\lambda)-y_0(\lambda)\right].
		\end{aligned}
	\end{equation}
\end{lem}

We have then again reduced the parametric dependence to only two functions of the parameter $\lambda$: $x_1$ and $x_2$. Notice $T$ is a homeomorphism between the \emph{domains} of both maps; this will help us in chapter \ref{cha:Generalization} to prove that the Linear Factors Form and the Canonical Form of polynomial maps are actually topologically conjugate, which in turn means that the stability and chaos properties are preserved between the maps, which allows us to determine stability properties for any cubic map by analyzing only the CCM.

Now, one first immediate result is

\begin{prop}
	The Canonical Cubic Map has three roots:
	\begin{equation}
		\begin{aligned}
			x_0^r&=0, & x_1^r&=\frac{x_1+x_2+s\sqrt{(x_2-x_1)^2-4s}}{2}, & x_2^r=\frac{x_1+x_2-s\sqrt{(x_2-x_1)^2-4s}}{2}.
		\end{aligned}
	\end{equation}
	Moreover, if $M>0$
	\begin{itemize}
		\item If $x_2\neq-x_1$ and $x_2-x_1=\pm2$, we will have a single root different from zero with multiplicity of two, and equal to $(x_1+x_2)/2$, i.e. the midpoint between the two non-zero fixed points.
		\item If $x_2=-x_1$ and $x_2-x_1=\pm2$, zero will be the only root, with multiplicity of three.
		\item If $\vert x_2-x_1\vert>2$, we will have two real and distinct non-zero roots.
		\item If $\vert x_2-x_1\vert<2$, zero will be the only real root of the CCM and the other two roots will be complex conjugates.
	\end{itemize}
	And if $M<0$ the three roots are always real and distinct, except when $x_1=x_2=\pm1$ when $x_0^r=0$ has multiplicity of two.
\end{prop}

As for the derivatives of the CCM, they are

\begin{equation}
	\begin{aligned}
		g_3'(x)&=1+s\left[(x-x_1)(x-x_2)+x(x-x_2)+x(x-x_1)\right],\\
		g_3''(x)&=s\left[6x-2(x_1+x_2)\right]
	\end{aligned}
\end{equation}

from which we can determine that the critical points are

\begin{equation}
	\begin{aligned}
		x_1^c&= \frac{1}{3}\left[x_1+x_2+s\sqrt{x_1^2-x_1x_2+x_2^2-3s}\right],\\
		x_2^c&= \frac{1}{3}\left[x_1+x_2-s\sqrt{x_1^2-x_1x_2+x_2^2-3s}\right].
	\end{aligned}
\end{equation}

So that, evaluating $g''$ at the critical values of $g,$ $x_1^c$ and  $x_2^c$ we have thatwe have that

\[g_3''(x_1^c)=-g_3''(x_2^c)=2\sqrt{x_1^2-x_1x_2+x_2^2-3s},\]

and, then, we have as expected one maximum and one minimum unless both are equal, in which case we have a saddle point. From this we have proved

\begin{prop}
	Consider the CCM as defined above. Then, if $M>0$
	\begin{itemize}
		\item If $x_1^2+x_2^2>3+x_1x_2$, both critical points are real and distinct, one corresponding to a maximum and the other to a minimum.
		\item If $x_1^2+x_2^2=3+x_1x_2$, both critical points are equal and correspond to a saddle point.
		\item If $x_1^2+x_2^2<3+x_1x_2$, the two critical points are complex conjugates, so we do not have any real critical points.
	\end{itemize}
	And if $M<0$,
	\begin{itemize}
		\item If $x_1^2+x_2^2+3>x_1x_2$, both critical points are real and distinct, one corresponding to a maximum and the other to a minimum.
		\item If $x_1^2+x_2^2+3=x_1x_2$, both critical points are equal and correspond to a saddle point.
		\item If $x_1^2+x_2^2+3<x_1x_2$, the two critical points are complex conjugates, so we do not have any (real) critical points.
	\end{itemize}
\end{prop}

\section{Stability and Chaos in the Canonical Cubic Map}
\label{sec:Stability&ChaosCCM}

Next, let us determine the stability of the periodic points of the CCM. As we will see further below in chapter \ref{cha:Generalization}, this analysis will suffice for any cubic map, by means of topological conjugacy. However, we can only explicitly give this for the fixed points.

\subsection{Fixed Points}
We already know, by construction, that the fixed points of the CCM are $x_0=0,\,x_1$ and $x_2$. While the first is constant always, the other two fixed points are set to be functions of the parameter $\lambda$. By evaluating in $g_3'$ we get the eigenvalue functions. For $x_0=0$ we have

\[\phi_0(\lambda)=g_3'(0)=sx_1(\lambda)x_2(\lambda)+1,\]

So that, the stability condition for this fixed point is

\begin{equation}
	\label{eq:CCM-stbcond-x0}
	-2<sx_1x_2<0.
\end{equation}

We can draw some conclusions from this. In order for zero to be a stable (attracting) fixed point one must have:

\begin{lem}
	The following are sufficient conditions for the asymptotic stability of the zero fixed point of the Canonical Cubic Map:
	\begin{itemize}
		\item in magnitude, $\vert x_1\vert\vert x_2\vert<2$; and
		\item if $M>0$, $x_1$ and $x_2$ must have \emph{different} signs; or
		\item if $M<0$, $x_1$ and $x_2$ must have \emph{the same} sign.
	\end{itemize}
\end{lem}

Notice that the stability condition \eqref{eq:CCM-stbcond-x0} states that the product of the distances from the other two fixed points to the zero fixed point must within the range $(-2,\,0)$, for positive $M$ [or $(0,\,2)$ for negative $M$], for the zero fixed point to be asymptotically stable.

The case of $x_k=0$, $k\in\{1,\,2\}$, is not included in the discussion here since this would represent repeated fixed points (multiplicity), which will be discussed in section \ref{sec:CCM-multiplicity} below; likewise, in the remainder of this section we will avoid dealing with multiplicity of the fixed points.

Now, for $x_1$, its eigenvalue function is

\[\phi_1(\lambda)=g_3'(x_1(\lambda))=1+sx_1(\lambda)(x_1(\lambda)-x_2(\lambda)),\]

so that the stability condition for this fixed point is

\begin{equation}
	-2<sx_1(x_1-x_2)<0.
	\label{eq:CCM-stbcond-x1}
\end{equation}

This gives us the following

\begin{lem}
	The following are sufficient conditions for the asymptotic stability of the $x_1$ fixed point:
	
	If $M>0$ then
	\begin{itemize}
		\item $x_1$ and $x_2$ must have \emph{the same} sign; and
		\item $|x_1|<|x_2|<|x_1+2/x_1|$.
	\end{itemize}
	
	On the other hand, if $M<0$,
	\begin{itemize}
		\item $|x_2|<|x_1|$; and
		\item if $|x_1|\geq\sqrt{2}$, then $|x_1-2/x_1|<|x_2|<|x_1|$; or
		\item if $0<x_1<\sqrt{2}$, then $x_1-2/x_1<x_2<x_1$; or
		\item if $-\sqrt{2}<x_1<0$, then $x_1<x_2<x_1-2/x_1$.
	\end{itemize}
\end{lem}

Again, notice that the stability condition \eqref{eq:CCM-stbcond-x1} for $x_1$ can be translated as that the product of the distances between the other two fixed points and $x_1$ must be within the range $(-2,\,0)$ for positive $M$ [or $(0,\,2)$ for negative $M$]. Also notice that when $0<|x_1|<\sqrt{2}$ the bound $x_1-2/x_1$ may be negative even if $x_1>0$ or positive even if $x_1<0$, therefore the usefulness of the distinction.

Finally, for $x_2$ we have the eigenvalue function

\[\phi_2(\lambda)=g_3'(x_2(\lambda))=1+sx_2(\lambda)(x_2(\lambda)-x_1(\lambda)),\]

which gives the stability condition

\begin{equation}
	\label{eq:CCM-stbcond-x2}
	-2<sx_2(x_2-x_1)<0,
\end{equation}

which in turn produces the following

\begin{lem}
	The following are sufficient conditions for the asymptotic stability of the $x_2$ fixed point of the Canonical Cubic Map:
	
	If $M>0$ then
	\begin{itemize}
		\item $x_1$ and $x_2$ must have \emph{the same} sign; and
		\item $|x_2|<|x_1|<|x_2+2/x_2|$.
	\end{itemize}
	
	On the other hand, if $M<0$,
	\begin{itemize}
		\item $|x_1|<|x_2|$; and
		\item if $|x_2|\geq\sqrt{2}$, then $|x_2-2/x_2|<|x_1|<|x_2|$; or
		\item if $0<x_2<\sqrt{2}$, then $x_2-2/x_2<x_1<x_2$; or
		\item if $-\sqrt{2}<x_2<0$, then $x_2<x_1<x_2-2/x_2$.
	\end{itemize}
\end{lem}

Likewise, we see that the stability condition \eqref{eq:CCM-stbcond-x2} involves the product of the distances from the other two fixed points to the one whose stability is of interest. We will later generalize these ``stability conditions'' to functions of the parameter which are different for each fixed point, but of whose value depend the stability of not only the fixed points, but higher period periodic points also, through period doubling bifurcations. From the stability conditions for the three fixed points we have proved the following

\begin{cor}
	A cubic polynomial map with three different real roots can only have a single attracting fixed point.
\end{cor}
\begin{proof}
	Compare the stability conditions for the three fixed points.
\end{proof}

And,

\begin{thm}
	If $M>0$, then sufficient conditions for the stability of a fixed point of the canonical cubic map are
	\begin{itemize}
		\item The product of the distances between each unstable fixed point and the stable one must be negative, which means one distance is positive and the other negative, which leads us to
		\item the fixed point that lies between the other two will be stable, while the outer fixed points will be unstable, as long as
		\item the product of the distances between each unstable fixed point and the stable one must be greater than -2.
	\end{itemize}
\end{thm}

\section{Higher period periodic points}

Although, as previously stated, in general, we cannot calculate the values of the periodic points of period 2 or higher, we can calculate for which values of the stability conditions above the fixed points undergo period doubling bifurcations. We will see in chapter \ref{cha:Generalization} that these stability conditions can actually be generalized to something called the ``Product Distance Function'', which depends on the parameter and is different for each fixed point. Example \ref{ex:CCM-exponential} below allowed us to determine the bifurcation values, $c_k$, of the fixed points of the CCM up to some precision. The values obtained are shown in table \ref{tab:CCM-BifurcationValues}. When the stability conditions of each fixed point crosses these values, bifurcations take place.

\begin{table}
	\begin{tabular}[]{|c|c|c|}
		\hline
		$k$ & $c_k$ \tabularnewline
		\hline
		0 & 0 \tabularnewline
		\hline
		1 & 2 \tabularnewline
		\hline
		2 & $3.0\pm0.005$\tabularnewline
		\hline
		3 & $3.236\pm0.002$\tabularnewline
		\hline
		4 & $3.288\pm0.0005$\tabularnewline
		\hline
		5 & $3.29925\pm0.00025$\tabularnewline
		\hline
		$\vdots$ & $\vdots$\tabularnewline
		\hline
		$\infty$ & $\sim3.30228\pm5\times10^{-6}$\tabularnewline
		\hline
	\end{tabular}
	\caption{Bifurcation values for the Canonical Cubic Map. $c_0=0$ is included only as a reference, although it is not a bifurcation value.}
	\label{tab:CCM-BifurcationValues}
\end{table}

From these values, we can construct the analogue of the stability bands of the CQM for the CCM.

\begin{defn}[Stability Bands of the CCM]
	Let $x_1,\,x_2:\mathcal{A}\subseteq\mathbb{R}\rightarrow\mathbb{R}$ be the two nonzero fixed points of the family of cubic maps $g_3$, as given by definition \ref{def:CCM}, and $\{c_k\}_{k\in\mathbb{N}}$ the sequence of bifurcation values of table \ref{tab:CCM-BifurcationValues}. The open interval
	\begin{equation}
		\left(-c_{k+1},\,-c_k\right),\quad\lambda\in\mathcal{A}
	\end{equation}
	is called the $k$-th \emph{stability band} of the CCM.
\end{defn}

Notice, however, that in contrast with the stability bands of the CQM, the stability bands of the CCM cannot be plotted along the fixed points plots, at least not directly as just defined, but rather they must be represented in a separate plot for the stability conditions, as we shall see in the examples of section \ref{sec:CCM-examples} further below.

\section{Multiplicity of the fixed points}
\label{sec:CCM-multiplicity}

When multiplicity of the fixed points takes place in the CCM, without loss of generality, $g_3$ can take the following forms

\begin{equation}
	g_3(x)=
	\begin{cases}
		x + sx^2(x-x_1),	&\text{if } x_2=x_0=0,\,x_1\neq0\\
		x + sx(x-x_1)^2, 	&\text{if } x_1=x_2\neq0\\
		x + sx^3,			&\text{if } x_1=x_2=0.
	\end{cases}
\end{equation}

with corresponding derivatives

\begin{equation}
	g'_3(x)=
	\begin{cases}
		1 + 2sx(x-x_1)+sx^2,			&\text{if } x_2=x_0=0,\,x_1\neq0\\
		1 + sx(x-x_1)^2+2sx(x-x_1), 	&\text{if } x_1=x_2\neq0\\
		1 + 3sx^2,						&\text{if } x_1=x_2=0.
	\end{cases}
\end{equation}

and therefore, $g'_3(x_k)=1,\,k\in\{0,\,1,\,2\}$, for all three cases, so that we deal with nonhyperbolic fixed points, according to definition \ref{def:hyperbolic-fp}.

\begin{prop}
	The stability of the fixed points of the CCM when they present multiplicity is, without loss of generality, as follows
	\begin{enumerate}
		\item If $x_2=x_0=0,\,x_1\neq0$, the zero fixed point is an unstable fixed point with multiplicity of two; and
		\begin{itemize}
			\item if $M>0$ and
			\begin{itemize}
				\item if $x_1>0$ it is semiasymptotically stable from the right;
				\item if $x_1<0$ it is semiasymptotically stable from the left;
			\end{itemize}
			\item or if $M<0$ and
			\begin{itemize}
				\item if $x_1>0$ it is semiasymptotically stable from the left;
				\item if $x_1<0$ it is semiasymptotically stable from the right.
			\end{itemize}
		\end{itemize}
		\item If $x_1=x_2\neq0$, this fixed point has multiplicity of two and it is \emph{unstable}; moreover
		\begin{itemize}
			\item if $M>0$ and
			\begin{itemize}
				\item if $x_1>0$ it is semiasymptotically stable from the left;
				\item if $x_1<0$ it is semiasymptotically stable from the right;
			\end{itemize}
			\item or if $M<0$ and
			\begin{itemize}
				\item if $x_1>0$ it is semiasymptotically stable from the right;
				\item if $x_1<0$ it is semiasymptotically stable from the left.
			\end{itemize}
		\end{itemize}
		\item If $x_0=x_1=x_2=0$, the zero fixed point has multiplicity of three and
		\begin{itemize}
			\item if $M>0$, it is unstable;
			\item if $M<0$, it is asymptotically stable.
		\end{itemize}
	\end{enumerate}
\end{prop}
\begin{proof}
	Notice that
	\begin{equation}
		g''_3(x)=
		\begin{cases}
			2s(x-x_1)+4sx,	&\text{if } x_2=x_0=0,\,x_1\neq0\\
			4s(x-x_1)+2sx, 	&\text{if } x_1=x_2\neq0\\
			6sx,			&\text{if } x_1=x_2=0.
		\end{cases}
	\end{equation}
	and
	\begin{equation}
		g'''_3(x)=6s\neq0,
	\end{equation}
	for all cases.
	Therefore
	\begin{enumerate}
		\item If $x_2=x_0=0,\,x_1\neq0$, the zero fixed point has multiplicity of two and we have that $g'_3(0)=1$, $g''_3(0)=-2sx_1\neq0$ and $g'''_3(0)=6s\neq0$; therefore, by theorem \ref{thm:nonhyperbolic-pos}, the zero fixed point is an \emph{unstable} fixed point. Applying theorem \ref{thm:semistability} we get the particular cases of semistability.
		\item If $x_1=x_2\neq0$, this is fixed point has multiplicity of two and we have that $g'_3(x_1)=1$, $g''_3(x_1)=2sx_1\neq0$ and $g'''_3(0)=6s\neq0$; therefore, by theorem \ref{thm:nonhyperbolic-pos}, this fixed point is \emph{unstable}; moreover, the semistability cases are inferred from theorem \ref{thm:semistability} again.
		\item If $x_0=x_1=x_2=0$, the zero fixed point has multiplicity of three and we have that $g'_3(0)=1$, $g''_3(0)=0$ and $g'''_3(0)=6s\neq0$; therefore, by theorem \ref{thm:nonhyperbolic-pos}, if $M>0$ the zero fixed point is \emph{unstable} and if $M<0$ it is \emph{asymptotically stable}.
	\end{enumerate}
\end{proof}

\section{Cubic examples}
\label{sec:CCM-examples}

Here we will deal with specific parametrizations for the fixed points $x_1$ and $x_2$ in order to clarify the above findings and to demonstrate how we can construct bifurcation diagrams with specific predetermined properties with cubic maps. We will consider $M>0$ unless otherwise stated explicitly.

\begin{example}
	\label{ex:CCM-linx1linx2}
	First, consider linear parametrizations for both $x_1$ and $x_2$ as
	
	\begin{equation}
		x_1(\lambda)=-\lambda,\quad x_2(\lambda)=\lambda.
		\label{eq:linx1linx2}
	\end{equation}
	
	The result is plotted in the lower panel of figure \ref{fig:CCM-linx1linx2}, where we see that the middle fixed point is $x_0=0$ always, so we expect this to be the only stable fixed point, until the separation between this and the other points breaks the stability condition and the period doubling bifurcation cascade sets on. The corresponding stability conditions are shown in the middle panel of figure \ref{fig:CCM-linx1linx2}, where we confirm that the curve for $x_0$ is the only one within the stability band $(-2,\,0)$ until $\lambda\approx1.45$ where the curve crosses the barrier of -2 and drops from then, causing $x_0$ to bifurcate. The corresponding bifurcation diagram is shown in the upper panel of figure \ref{fig:CCM-linx1linx2}, where we confirm the stated above.
	\begin{figure}
		\centering
		\includegraphics[width=0.9\columnwidth]{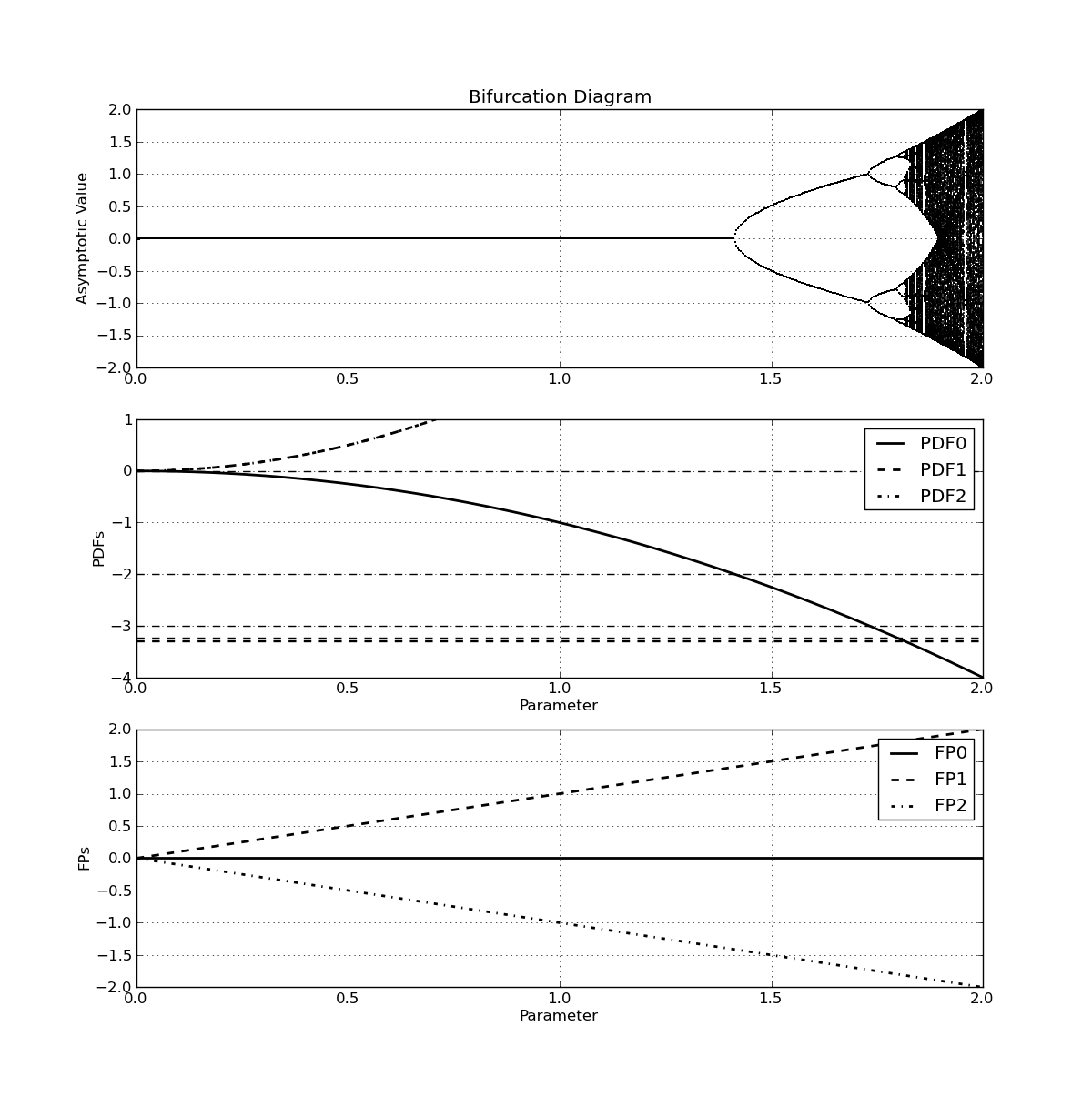}
		\caption{Bifurcation diagram (upper), stability conditions plot (middle) and fixed points plot (lower) for the linear parametrization of the fixed points example of the CCM.}
		\label{fig:CCM-linx1linx2}
	\end{figure}
\end{example}

\FloatBarrier

\begin{example}
	\label{ex:CCM-sqrtx1sqrtx2}
	
	Next, based on the fact that the stability conditions depend upon the product of the distances between the stable fixed point and the other two fixed points, we take square root parametrizations of the fixed points and so that the product is linear. We take the parametrizations as
	
	\begin{equation}
		x_1(\lambda)=\sqrt{\lambda},\quad x_2(\lambda)=-\sqrt{\lambda},
		\label{eq:CCM-sqrtx1sqrtx2}
	\end{equation}
	
	The plot of the fixed points is shown in lower panel of figure \ref{fig:CCM-sqrtx1sqrtx2}, where we see again that the middle fixed point is $x_0$, so this will be the only attracting fixed point as long as its corresponding stability curve is within the stability region. In the middle panel of figure \ref{fig:CCM-sqrtx1sqrtx2} we see that such stability curve goes out of said region at $\lambda=2$ and $x_0$ continues to bifurcate (period doubling) thereof. The corresponding bifurcation diagram is shown in the upper panel of figure \ref{fig:CCM-sqrtx1sqrtx2}, which is very much like the one of example \ref{ex:CCM-linx1linx2} above, but looks ``stretched'' since the progression towards bifurcation is now linear, instead of quadratic, and given by the stability conditions curves.
	
	\begin{figure}
		\centering
		\includegraphics[width=0.9\columnwidth]{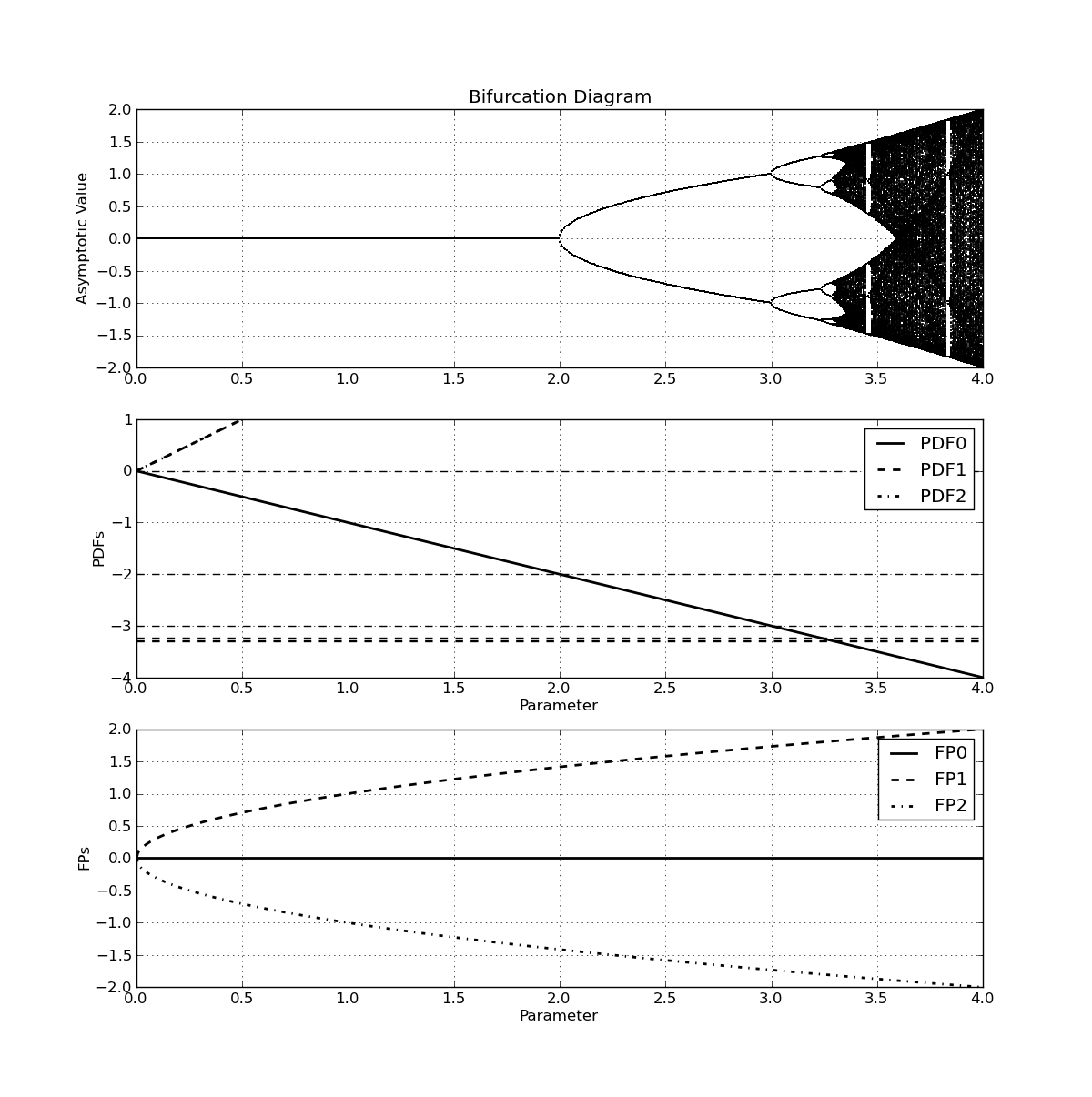}
		\caption{Bifurcation diagram (upper), stability conditions plot (middle) and fixed points plot (lower) for the square root parametrizations of the fixed points example of the CCM.}
		\label{fig:CCM-sqrtx1sqrtx2}
	\end{figure}
	
\end{example}

\FloatBarrier

\begin{example}
	\label{ex:CCM-cstx1linx2}
	
	Now we will explore the full range of stability regions by making a linearly varying fixed point pass through the regions defined by the constant $x_0$ and a constant $x_1$. We define then
	
	\begin{equation}
		x_1(\lambda)=2,\quad x_2(\lambda)=6\lambda+1,
		\label{eq:CCM-cstx1linx2}
	\end{equation}
	
	We then obtain the plot of the lower panel of figure \ref{fig:CCM-cstx1linx2}, for the selected range of interest of the parameter $\lambda$. In the middle panel of the same figure we can see the stability curves for the fixed points, where we see that initially, from left to right, all fixed points are unstable then, progressively, $x_0$, $x_2$ and $x_1$ become stable, the latter one losing stability for still greater values of $\lambda$. The corresponding bifurcation diagram is shown in the upper panel, where we can see how first, the stable fixed point is $x_0=0$, since it is the middle one, but begins in the chaotic region and goes ``reversal'' towards being stable; then, as $x_2$ crosses through zero, it becomes the middle stable fixed point and, when it in turn crosses the constant $x_2$, this latter one becomes the stable fixed point, again loosing stability when $x_2$ crosses the stability band for $x_1$.
	
	\begin{figure}
		\centering
		\includegraphics[width=0.9\columnwidth]{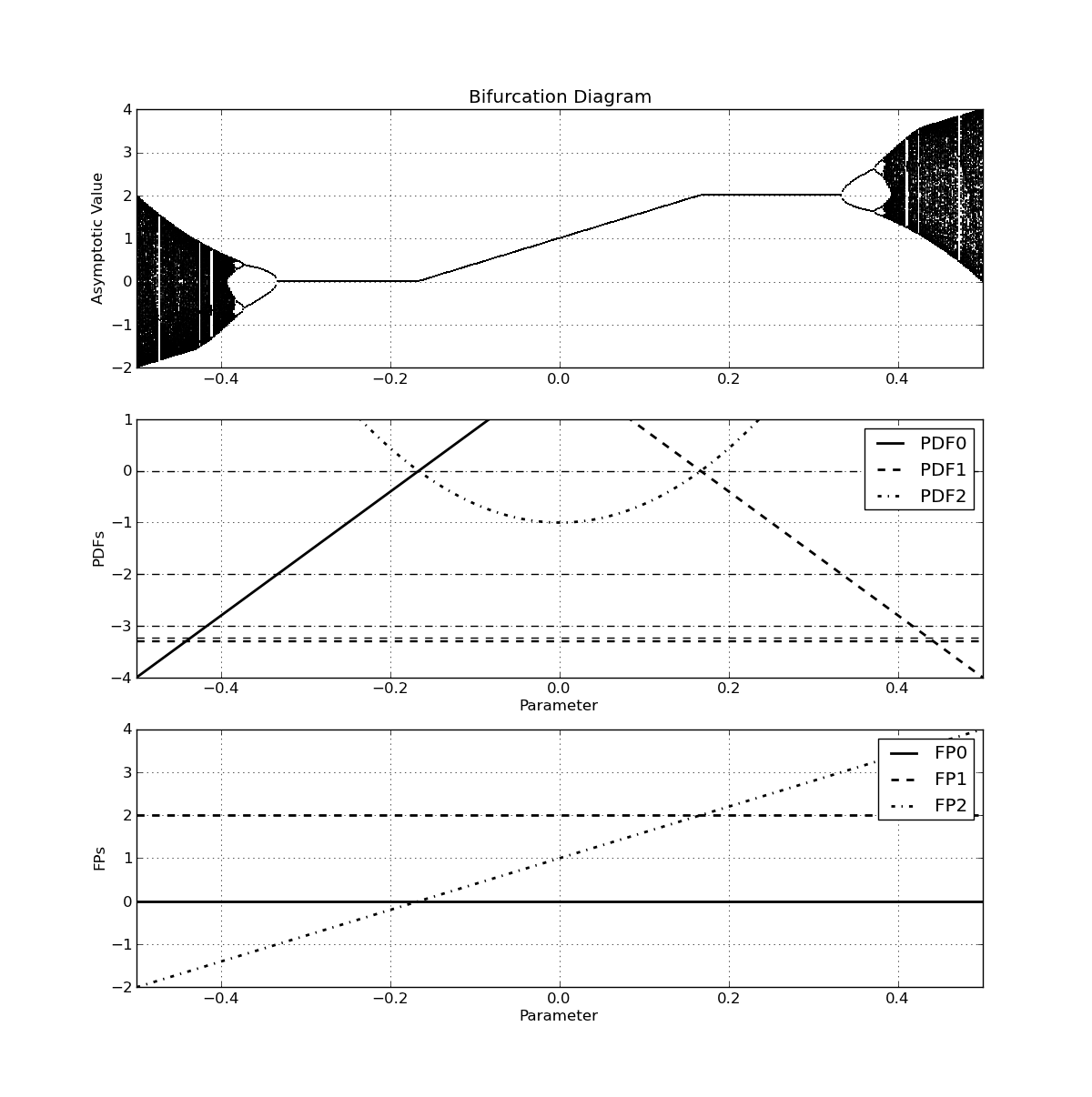}
		\caption{Bifurcation diagram (upper), stability conditions plot (middle) and fixed points plot (lower) for the constant and linear mixed parametrizations of the fixed points example of the CCM.}
		\label{fig:CCM-cstx1linx2}
	\end{figure}
	
\end{example}

\FloatBarrier

\begin{example}
	\label{ex:CCM-exponential}
	
	This example is analogous to example \ref{ex:expfp-CQM} of chapter \ref{cha:QuadRev}, where we had the $x_1$ fixed point of the CQM tend asymptotically to a value of approximately $-b_{\infty}$. To achieve the same qualitative phenomenon in the CCM, we can define the roots as
	
	\begin{equation}
		x_1(\lambda)=1.817-e^{-\lambda},\quad x_2(\lambda)=-1.817+e^{-\lambda}
	\end{equation}
	
	so that $x_2=-x_1$ and the fixed points are symmetric about the zero fixed point; their product, which gives the stability condition for the zero fixed point, will then be $-3.301489+3.634e^{-\lambda}+e^{-2\lambda}$ which will asymptotically tend to $-3.301489\approx -b_{\infty}$, as $\lambda\rightarrow\infty$; this allowed us to calculate the bifurcation values of table \ref{tab:CCM-BifurcationValues}. The lower panel of figure \ref{fig:CCM-exponential} shows the graphs of the fixed points, the middle panel shows the stability conditions for each fixed point and the upper panel shows the corresponding bifurcation diagram. This example allows us to determine the bifurcation values $c_k$ of the CCM with more precision; the values obtained are shown in table \ref{tab:CCM-BifurcationValues}.
	
	\begin{figure}
		\centering
		\includegraphics[width=0.9\columnwidth]{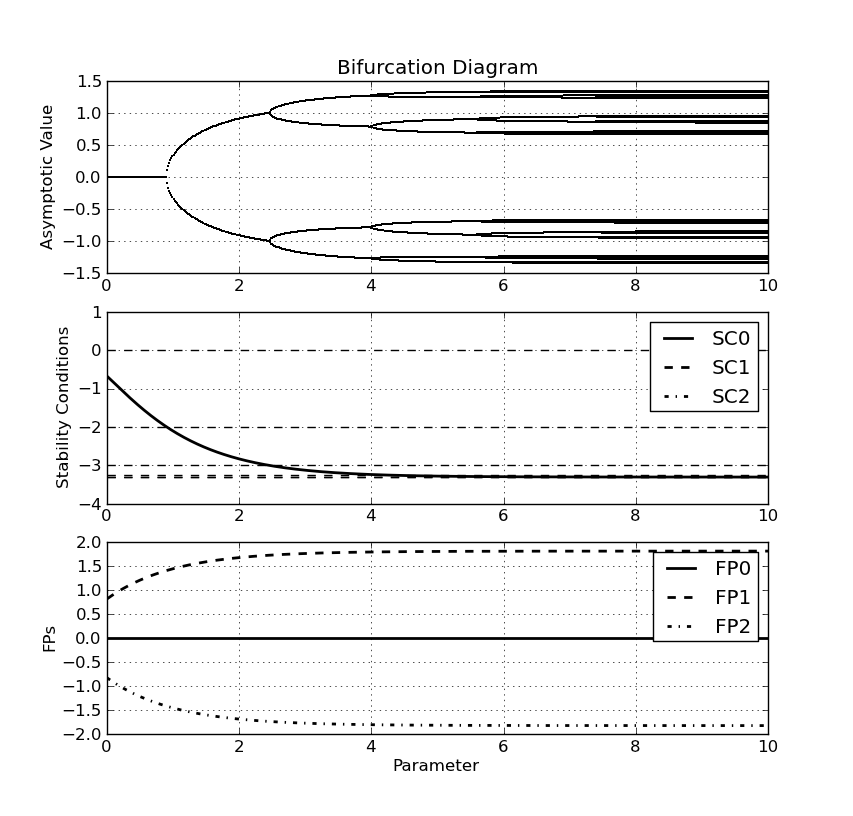}
		\caption{Bifurcation diagram (upper), stability conditions plot (middle) and fixed points plot (lower) for the exponential asymptotic parametrizations of the fixed points example \ref{ex:CCM-exponential} for the CCM.}
		\label{fig:CCM-exponential}
	\end{figure}
	
\end{example}

\chapter{\textit{n}-th degree polynomial maps}
\label{cha:Generalization}

Consider again a one-dimensional discrete dynamical system defined by
\begin{equation}
	y_{n+1}=f(y_n;\lambda)
\end{equation}

Where $f$ is a polynomial in one real variable $y$ with real fixed points and whose coefficients depend smoothly on the real parameter $\lambda$. Depending on the form of $f$ we have arrived previously to the definitions of the General, Linear Factors and Canonical Forms of the polynomial quadratic and cubic maps. Next, we will define precisely what we understand by the General and Canonical Maps of $n$-th degree.

\begin{defn}[General Polynomial Map]
	The General Polynomial Map of $n$-th degree (GPM-$n$) is defined by
	\begin{equation}
		f_n(y):= y+ (-1)^{n-1} P_{f_n}(y)
	\end{equation}
	where
	\begin{equation}
		P_{f_n}(y):=(-1)^{n-1}\sum_{i=0}^{n} a_i\,y^i.
	\end{equation}
	$P_{f_n}$ is called the \emph{Fixed Points Polynomial} associated to $f_n$.
\end{defn}

It is obvious that any $n$-th degree real polynomial in one variable can be put into the General Form by means of adjusting the value of the $a_1$ coefficient properly in the fixed points polynomial. This is the broadest class of real polynomials of finite degree.

It is clear that the roots of $P_{f_n}$ are the fixed points of $f_n$. In the case of $n$ odd, the fundamental theorem of algebra guarantees the existence of at least one real fixed point. Let $y_i\in\mathbb{C},\,i\in\{0,1,\,...,\,n-1\}$, be the $n$ roots of $P_{f_n}$, then $(y-y_i)$ is a factor of $P_{f_n}$ by the factor theorem, therefore we can rewrite $P_{f_n}$ as

\begin{equation}
	\label{eq:CPM-FPP}
	P_{f_n}(y)=M (y-y_0)\cdots (y-y_{n-1})=M\prod_{j=0}^{n-1}(y-y_j),\quad M\in\mathbb{R},
\end{equation}

and then define the

\begin{defn}[Linear Factors Form]
	\label{def:LinearFactorsForm}
	Let $f_n$ be the GPM-$n$ and $y_j$, $j\in\{0,\,...,\,n-1\}$, its $n$ fixed points. Then we can write
	\begin{equation}
		\begin{aligned}
			f_n(y)	&=y+(-1)^{n-1}\,M\,\prod_{i=0}^{n-1} (y-y_i),\\
			&=y+(-1)^{n-1}\mathrm{sgn}(M)|M|\prod_{i=0}^{n-1} (y-y_i),\\
			&=y+(-1)^{n-1}s\tilde{M}\prod_{i=0}^{n-1} (y-y_i),\\
		\end{aligned}
	\end{equation}
	with the definitions of $s$ and $\tilde{M}$ used in chapters \ref{cha:QuadRev} and \ref{cha:Cubic}. We call this the \emph{Linear Factors Form} of $f_n$.
\end{defn}

We can directly verify that for $n=2,\,3$ we obtain the corresponding Linear Factors Form of the Quadratic and Cubic Maps, described on pages \pageref{def:LFFQM} and \pageref{def:LFFCM}, respectively. Once we know the $n$ fixed points of a map $f_n$, it is straightforward to write its linear factors form. The motivation behind the $(-1)^{n-1}$ factor is that we want that, for purely aesthetic reasons, if $M>0$, the fixed points are real and $0<y_0<y_1<\cdots<y_{n-1}$, we have that $f'_n(0)\geq0$, which is thus accomplished.

We will now restrict this set of polynomials to those whose fixed points polynomials have only real roots, i.e. maps with real fixed points only, though not necessarily distinct. Let us make this precise by the following

\begin{defn}[Canonical Polynomials Set]
	The Canonical Polynomials Set, denoted $P_C[y]$, is
	\begin{equation}
		P_C[y]:=\{f\in\mathbb{R}[y]\,\vert\,\text{$P_f$ has only real roots.}\}
	\end{equation}
	where $\mathbb{R}[y]$ is the set of polynomials with real coefficients on the variable $y$. Likewise, $P_C^n[y]$ denotes $P_C[y]\bigcap\mathbb{R}_n[y]$ where $\mathbb{R}_n[y]$ is the set of polynomials of degree $n$ with real coefficients on the variable $y$.
\end{defn}

The set $P_C$ has been our main work ground for the analysis in this work and, as it turns out, its elements can be put in a much nicer form, easier to understand.

We can further reduce the complexity of this set of maps by means of the transformation

\begin{equation}
	y=T_n(x):=s\tilde{M}^{-\frac{1}{n-1}}x+y_0
\end{equation}

where $y_0$ is a real fixed point of the map in its linear product form. Notice that $T_n$ is linear, therefore it has an inverse

\begin{equation}
	x=T_n^{-1}(y)=s\tilde{M}^{\frac{1}{n-1}}(y-y_0).
\end{equation}

$T_n$ is in fact a homeomorphism. We will drop the subscript $n$ when referring to the transformation for no specific degree. Applying this transformation to $y$ we reach the

\begin{defn}[Canonical Polynomial Map]
	\label{def:CanonicalPolynomialMap}
	The Canonical Polynomial Map of $n$-th degree (CPM-$n$) is
	\begin{equation}
		g_n(x):=x+(-1)^{n-1}\,s^n\,x\,\prod_{i=1}^{n-1}(x-x_i),\quad n\geq 2,
	\end{equation}
	where
	\[x_i=s\tilde{M}^{\frac{1}{n-1}}(y_i-y_0),\]
	and $y_j$ are the $n$ fixed points of the corresponding linear factors form map of $n$-th degree, (at least) $y_0$ being real.
\end{defn}

It is clear from the definition that $x_0=0$ always. Notice also that the $x_i$ result from evaluating $T_n^{-1}$ in the corresponding $y_i$. We can easily prove that not only does the canonical map result from applying $T$, but that the canonical map is in fact $T$-conjugate to the linear factors form.

\begin{prop}
	Let $f_n$ and $T_n$ and $g_n$ as defined above, having $f_n$ at least one real fixed point; let $y_0$ be this real fixed point, without loss of generality. Then $f_n$ is $T_n$-conjugate to $g_n$.
\end{prop}
\begin{proof}
	It is clear that $T_n$ is a homeomorphism since it is linear. Then, we must only prove that $T_n\circ f_n = g_n\circ T_n$, i.e. $f_n(T_n(x))=T_n(g_n(x))$. We then have
	\begin{equation}
		\begin{aligned}
			f_n(T_n(x))	&=f_n(s\tilde{M}^{-\frac{1}{n-1}}x+y_0)\\
			&=s\tilde{M}^{-\frac{1}{n-1}}x+y_0 + (-1)^{n-1}\,s\tilde{M}\,\prod_{i=0}^{n-1} (s\tilde{M}^{-\frac{1}{n-1}}x+y_0-y_i)\\
			&=s\tilde{M}^{-\frac{1}{n-1}}x+(-1)^{n-1}\,s^2\tilde{M}\,\tilde{M}^{-\frac{1}{n-1}}\,x\,\prod_{i=1}^{n-1}(s\tilde{M}^{-\frac{1}{n-1}}x+y_0-y_i)+y_0\\
			&=s\tilde{M}^{-\frac{1}{n-1}}x+(-1)^{n-1}\,s^{n-1}\tilde{M}^{-\frac{1}{n-1}}\,x\,\prod_{i=1}^{n-1}\left[x-s\tilde{M}^{\frac{1}{n-1}}(y_i-y_0)\right]+y_0\\
			&=s\tilde{M}^{-\frac{1}{n-1}}x+(-1)^{n-1}\,s^{n-1}M^{-\frac{1}{n-1}}\,x\,\prod_{i=1}^{n-1}\left(x-x_i\right)+y_0\\
			&=s\tilde{M}^{-\frac{1}{n-1}} \left[ x+(-1)^{n-1}\,s^n x\,\prod_{i=1}^{n-1}\left(x-x_i\right)\right]+y_0\\
			&=T_n(g_n(x)).
		\end{aligned}
	\end{equation}
	Where we have used $s^2=1$ and $s=s^{-1}$.
\end{proof}

This turns out to be very useful, since we know that topological conjugacy is an equivalence relation that \emph{preserves the property of chaos} (see theorem \ref{thm:Conjugacy} in page \pageref{thm:Conjugacy}). This means that the analysis of stability and chaos (i.e. the ``dynamics'') of real polynomial maps with real fixed points is reduced to the study of the canonical polynomial maps defined above, since we can always take any polynomial in $P_C[x]$ to its canonical form by means of $T$, determine the stability properties and then go back to the original polynomial. A commutative diagram of the conjugacy is

\begin{equation}
	\begin{CD}
		y		@=	y				@>{f_n}>>	f_n(y)		@= 	f_n(y)\\
		@A{T}AA		@VV{T^{-1}}V				@A{T}AA			@VV{T^{-1}}V\\
		x		@=	x				@>>g_n>		g_n(x)		@=	g_n(x)
	\end{CD}
\end{equation}

\section{Stability and chaos in the canonical map of degree \textit{n}}

The derivative of $g_n$, recalling $x_0=0$ to simplify notation, is

\begin{equation}
	g_n'(x)=1+(-1)^{n-1}\,s^n\,\sum_{j=0}^{n-1}\prod_{i=0,\,i\neq j}^{n-1}(x-x_i).
	\label{eq:CPM-derivative1}
\end{equation}

Evaluating \eqref{eq:CPM-derivative1} in the fixed point $x_k$ we get the eigenvalue function for each $x_k$

\begin{equation}
	\begin{aligned}
		\phi_k(\lambda)=g_n'(x_k(\lambda))	&=1+(-1)^{n-1}\,s^n\,\prod_{i=0,\,i\neq k}^{n-1}(x_k(\lambda)-x_i(\lambda)) \\
		&=1+s^n\,\prod_{i=0,\,i\neq k}^{n-1}(x_i(\lambda)-x_k(\lambda)). \\
	\end{aligned}
\end{equation}

Then, the asymptotic stability condition $\vert g_n'(x_k)\vert <1$ implies that

\begin{equation}
	-2<s^n\,\prod_{i=0,\,i\neq k}^{n-1}(x_i-x_k)<0.
	\label{eq:CPM-stbcond}
\end{equation}

From \eqref{eq:CPM-stbcond} we can recover all the stability conditions for the fixed points of the Canonical Quadratic and Cubic Maps.

which leads us to the following

\begin{defn}[Product Distance Function]
	Let $g_n$ be the Canonical Polynomial Map of $n$-th degree and $x_0=0,\,x_1,\,...,\,x_{n-1}$ its $n$ fixed points, all of which depend upon the parameter $\lambda$. Let $x_k$ be a real fixed point among the latter. Then
	\begin{equation}
		D_{n,k}(\lambda):=s^n\,\prod_{i=0,\,i\neq k}^{n-1}(x_i(\lambda)-x_k(\lambda)),\quad k\in\{0,\,...,\,n-1\},\,n\geq2,
	\end{equation}
	is called the \emph{Product Distance Function (PDF)} of $x_k$.
\end{defn}

The definition is motivated by the fact that $D_{n,k}$ is a product of the distances to $x_k$ of each of the other $n-1$ fixed points and that this quantity is fundamental in determining the stability of the fixed points. These distances are positive when $x_i>x_k$ and is negative when $x_i<x_k$. We have stressed the dependence on the parameter $\lambda$ in the definition of $D_{n,k}$ so that its character as a function is clear, stemming from the corresponding dependence on $\lambda$ of the fixed points. In this way, the stability condition for the $k$-th fixed point is reduced to

\begin{equation}
	-2<D_{n,k}(\lambda)<0.
\end{equation}

Since $D_{n,k}$ must be negative in order for $x_k$ to be stable as a sufficient condition, and an \emph{odd} number of factors $(x_i-x_k)$ must be negative for the product in $D_{n,k}$ to be negative, it follows that, if $n$ is even or $M>0$, an odd number of negative factors $(x_i-x_k)$ is a sufficient condition for the hyperbolic fixed point $x_k$ to be stable; i.e. if $M>0$, an \emph{odd} number of fixed points must lie below $x_k$ and, consequently, an even or zero (respectively, odd) number of fixed points must lie above $x_k$ if $n$ is even (respectively, odd). By similar arguments, we can prove

\begin{prop}[Necessary conditions for the stability of $x_k$]
	\label{prop:GPM-NecessaryConditions}
	Let $g_n$, $D_{n,k}$ be defined as above and $x_k$ be a hyperbolic real fixed point of $g_n$. The following are necessary conditions for $x_k$ to be an asymptotically stable fixed point:
	\begin{itemize}
		\item if $M>0$ or $n$ is even, an odd number of fixed points must have values lower than $x_k$; or
		\item if $n$ is odd and $M<0$, zero or an even number of fixed points must have values lower than $x_k$.
	\end{itemize}
\end{prop}

We must remark that the above conditions are \emph{not} sufficient for a fixed point to be an attractor. The sufficient condition, however, is stated as

\begin{thm}[Sufficient condition for the stability of $x_k$]
	Let $g_n$, $D_{n,k}$ be defined as above and $x_k$ be a hyperbolic real fixed point of $g_n$. Then, a necessary and sufficient condition for $x_k$ to be an attractor is that
	\begin{equation}
		-2<D_{n,k}(\lambda)<0.
	\end{equation}
\end{thm}

Below the value of -2 there are other ``stability bands'' that lead to further period doubling bifurcations of the fixed points as they are crossed, but they must be calculated numerically and, as we have seen, depend on the degree $n$ of the polynomial.

\section{Example}

We will now give an example in which the above theory is applied to different degree polynomial maps.

\begin{example}[Quartic Maps]
	
	Using definition \ref{def:LinearFactorsForm} for $n=4$, we have that
	\begin{equation}
		f_4(y)=y-M\,(y-y_0)(y-y_1)(y-y_2)(y-y_3).
	\end{equation}
	Suppose $f_4$ has at least one real fixed point. Without loss of generality, suppose this fixed point is $y_0$. Then
	\[T_4(x)=s\tilde{M}^{-1/3}\,x+y_0.\]
	Making the substitution $y=T_4(x)$ we can verify that we get
	\[g_4(x)=x-x(x-x_1)(x-x_2)(x-x_3),\]
	where
	\[x_i=s\tilde{M}^{1/3}(y_i-y_0),\quad i\in\{0,\,1,\,2,\,3\},\]
	The stability a real fixed point $x_k$ is given by the product distance function
	\[D_{4,k}(\lambda)=\prod_{i=0,\,i\neq k}^3 (x_i(\lambda)-x_k(\lambda)),\]
	whose value must remain between minus two and zero in order for $x_k$ to be asymptotically stable; that is, if all fixed points are real,
	\begin{equation}
		\begin{aligned}
			-2 	&< D_{4,0}(\lambda)	= x_1\,x_2\,x_3 			&< 0,\\
			-2 	&< D_{4,1}(\lambda)	= -x_1(x_2-x_1)(x_3-x_1) 	&< 0,\\
			-2 	&< D_{4,2}(\lambda)	= -x_2(x_1-x_2)(x_3-x_2) 	&< 0,\\
			-2 	&< D_{4,3}(\lambda)	= -x_3(x_1-x_3)(x_2-x_3) 	&< 0,
		\end{aligned}
	\end{equation}
	for 0, $x_1,\,x_2$ and $x_3$ to be asymptotically stable fixed points, respectively.
	
	For example, let
	\begin{equation}
		\begin{aligned}
			x_1(\lambda) &= \lambda,\\
			x_2(\lambda) &= -\lambda,\\
			x_3(\lambda) &= 2\,\lambda.
		\end{aligned}
	\end{equation}
	
	The plots of these fixed points with their corresponding parametric dependence on $\lambda$ are shown in the lower panel of figure \ref{fig:CPM-Quartic}.
	
	In light of proposition \ref{prop:GPM-NecessaryConditions} we expect only $x_0$ and $x_3$ to be able to be asymptotically stable fixed points in any given range of $\lambda$. As the middle panel of figure \ref{fig:CPM-Quartic} shows, precisely $x_0$ and $x_3$ are the fixed points whose product distances cross the stability band $(-2,\,0)$ in the range of $\lambda$ being plotted. As we recall, the product distance functions are the ``stability conditions'' of the fixed points. As long as the product distances remain within the stability interval, the fixed points are attractors, as we can verify in the upper panel of figure \ref{fig:CPM-Quartic}; also in this last panel, we can see the two attracting fixed points at the beginning of the plotted range; then, first $x_3$ loses its stability and gives rise to the period doubling bifurcations cascade which leads to chaotic behavior; later, zero also loses its stability and also gives rise to period doubling and chaos.
	
	\begin{figure}
		\centering
		\includegraphics[width=0.9\columnwidth]{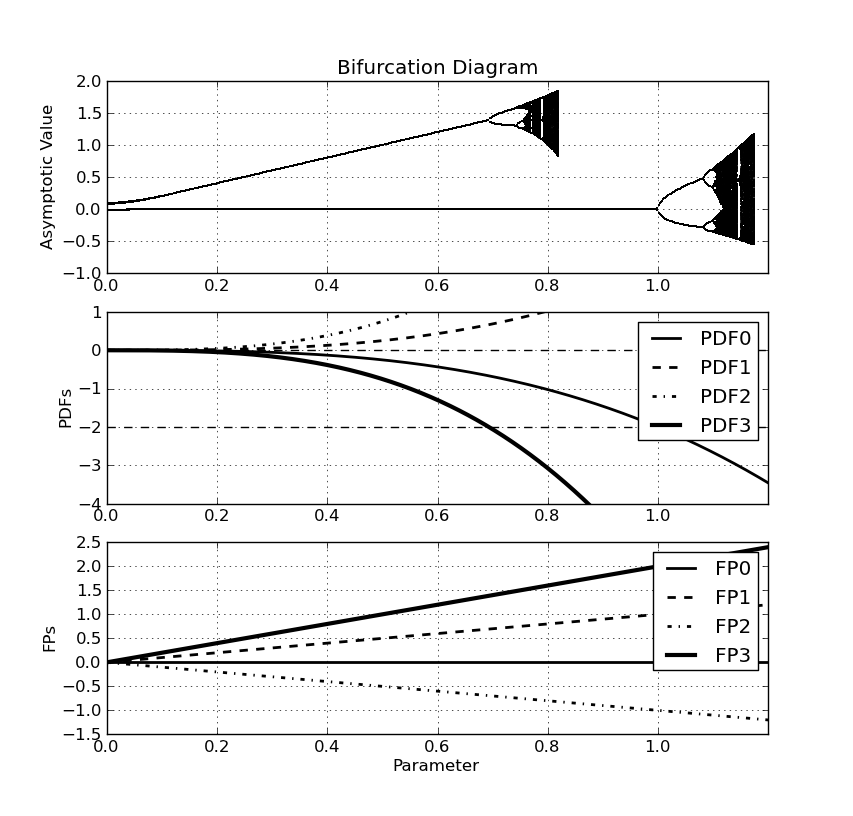}
		\caption{Bifurcation diagram (upper), Product Distance Functions (middle) and fixed points plot of the quartic polynomial map example.}
		\label{fig:CPM-Quartic}
	\end{figure}
	
\end{example}

\FloatBarrier
\chapter{An example application of the canonical map theory}
\label{cha:Applications}

We shall now consider an example application of the theory developed in chapters \ref{cha:QuadRev} through \ref{cha:Generalization} to a couple of well known examples of discrete dynamical systems. This applications demonstrate a new method of solution for analyzing the dynamics of real polynomial maps that simplifies upon the procedures found in current literature.

\section{The logistic map}

The logistic map is the most immediate and obvious example application. We had already encountered briefly the logistic map, $L_{\lambda}(x)=\lambda x(1-x)$, in subsection \ref{sub:PeriodDoubling} of chapter \ref{cha:Intro} and again in example \ref{ex:logistic} of chapter \ref{cha:Quadratic}, where we saw that the map undergoes a series of period doubling bifurcations starting at the value of $\lambda=3$, ultimately achieving a chaotic nature at $\lambda\approx 3.570$ \cite[p. 47]{Elaydi2008}. We will now proceed to utilize the theory developed in the previous chapters, particularly in chapter \ref{cha:QuadRev} and generalized in chapter \ref{cha:Generalization}, to this particular map.

We already know (or can easily calculate) that the fixed points of the logistic map are $y_0=0$ and $y_1=\frac{\lambda-1}{\lambda}$ \cite[p. 43]{Elaydi2008}. With some minor algebra, we can then check that in the corresponding Linear Factors Form of the logistic map is then

\begin{equation}
	h_\lambda(x)= x - \lambda x \left(x - \frac{\lambda-1}{\lambda}\right)
\end{equation}

where we can identify the functions of the parameter $M$, $y_0$ and $y_1$ from definition \ref{def:LFFQM}, on page \pageref{def:LFFQM}, as

\begin{align}
	s &= +1,\\
	\tilde{M}(\lambda) &= \lambda,\\
	y_0(\lambda) &= 0,\\
	y_1(\lambda) &= \frac{\lambda-1}{\lambda}.
\end{align}

The corresponding non-zero fixed point of the canonical logistic map, $x_1(\lambda)=s\tilde{M}(y_1-y_0)$ defined in \eqref{eq:x1dependence}, is then simply

\begin{equation}
	x_1(\lambda)=\lambda-1,
\end{equation}

from which we can state that the \emph{canonical logistic map} takes the explicit form

\begin{equation}\label{eq:CanonicalLogistic}
	g_{\lambda}=x-x(x-\lambda+1).
\end{equation}

In order to determine the stability properties of these fixed points, both zero and nonzero, in the canonical logistic map, it is then sufficient, as we have proved in chapters \ref{cha:QuadRev} and \ref{cha:Generalization}, to observe the behavior of the Product Distance Functions (PDFs) of these fixed points, viz.

\begin{align}
	D_{g, 0}(\lambda) &= x_1-0 = x_1 = \lambda -1,\\
	D_{g, 1}(\lambda) &= 0-x_1 = -x_1 = 1-\lambda.
\end{align}

By determining when these PDFs cross the stability bands whose boundaries are shown in table \ref{tab:CQM-BifurcationValues} we can readily determine when these fixed points are stable or unstable, when they bifurcate and when they reach any $2^n$ attracting periodic orbit for any $n$, up to crossing the $b_{\infty}$ band. This whole process is depicted in figure \ref{fig:dbif_logistic}. In particular, we can see from table \ref{tab:CQM-StabilityConditions}, and again at figure \ref{fig:dbif_logistic}, that when $-1<x_1<0$, the zero fixed point is attracting since its PDF lies within the first stability band, then exchanges stability at $x_1=0$, when this last FP becomes stable and proceeds to a period doubling cascade upon its PDF, $-x_1$, crossing the bands defined by the bifurcation values $-b_1$, $-b_2$, etc. until reaching $x_1=b_{\infty}\approx 2.569941$. This last value, by the way, agrees quite well and improves upon the approximation reported in \cite{Elaydi2008} of 3.570 for the logistic map, since with the calculations of the present work $\lambda_{\infty}=1+b_{\infty}\approx 3.569941\pm5\times10^{-7}$. Finally, since these maps are topologically conjugate and it is known that the logistic map is chaotic starting with $\lambda=4$ \cite{Elaydi2008}, we conclude that the CQM must be so starting from $x_1=3$, which we may define as $b_c$, our final ``bifurcation'' value.

\begin{figure}
	\centering
	\includegraphics[width=0.8\columnwidth]{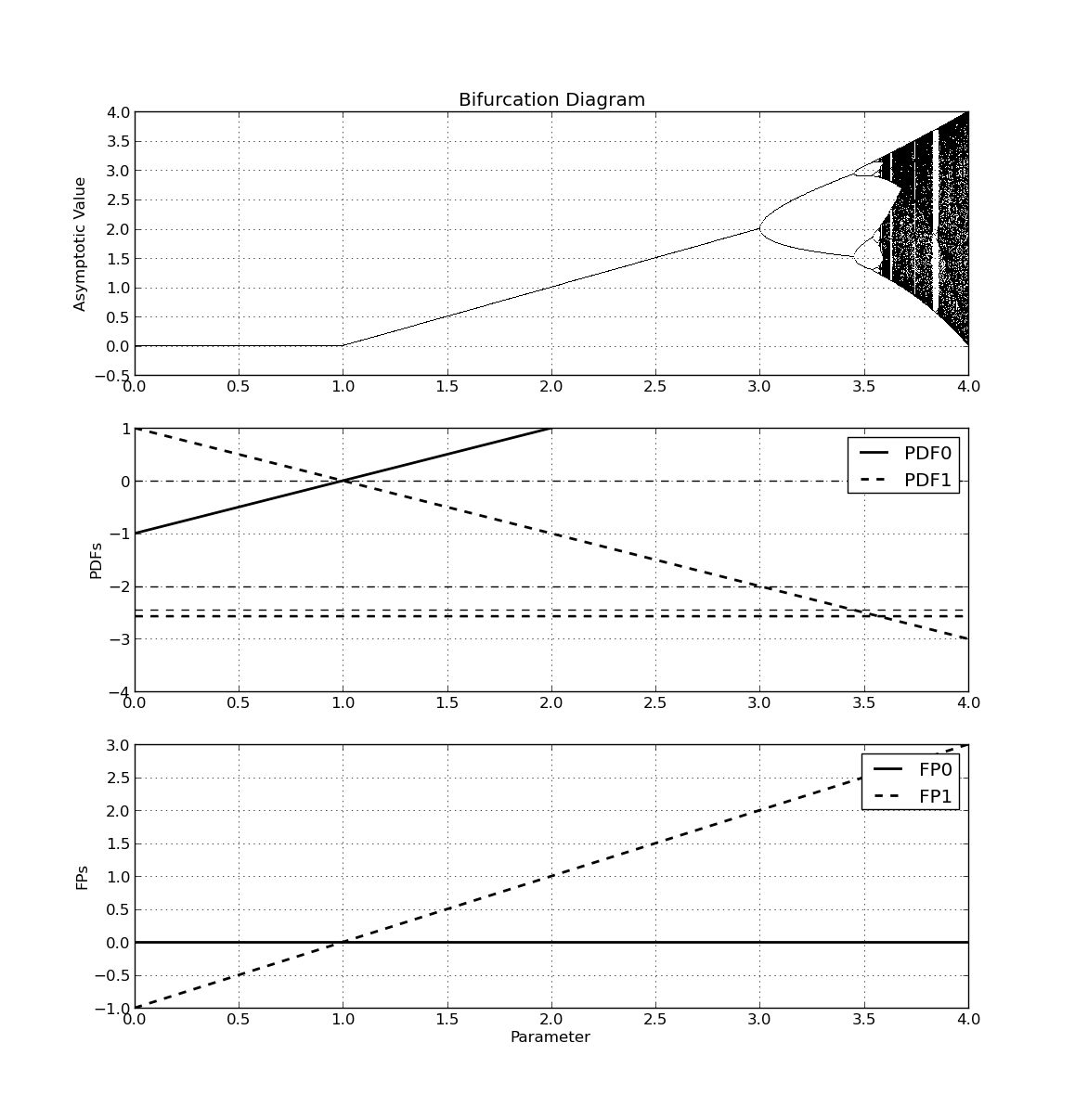}
	\caption{Bifurcation diagram for the canonical logistic map.}
	\label{fig:dbif_logistic}
\end{figure}

\section{Harvesting strategies}

The connection of the logistic map with population models is old and well known. In the book by \cite{Sandefur}, there are a few examples of uses of second degree polynomials as recurrence functions which are used to model ``harvesting'' --or hunting-- strategies of a population of animals. The main idea is that the population of animals grows whenever there is food and resources in the environment which, by account of its finite resources has a certain ``carrying capacity''; this leads to a maximum population this environment can hold, the population growing according to the logistic model. The population may then be ``harvested'' --or hunted-- at a certain rate yet to be specified and, depending on this rate, it is not hard to imagine that the final fate of the population of animals may be (i) extinction, if the rate is too high; (ii) steady population below the carrying capacity, if the rate is ``just right''; or (iii) steady population at its maximum value dictated by the carrying capacity of the environment. A final fourth possibility --perhaps more rare-- is to (iv) bring more individuals of the species from outside the system under analysis and introduce them to it, therefore making it possible for the population to surpass in number the carrying capacity of the system, but only to return naturally to the maximum value after a finite number ``time-steps''. To examine this in detail consider the system defined by the recurrence relation

\begin{equation}\label{eq:harvest1}
	\Delta y\equiv y_{n+1}-y_n=r(1-y_n)y_n.
\end{equation}

Here the population growth in any given period $\Delta y$ is proportional to both the initial population $y_n$ and the difference between this and the maximum population, normalized to the value of 1. The proportion constant is the growth rate $r$. After some simplification, we can rewrite system as

\begin{equation}\label{eq:harvestmod}
	y_{n+1}=(1+r)y_n-ry_n^2.
\end{equation}

We have yet to add the ``harvesting term(s)'', which we might do so in several ways.


If we consider a fixed rate, say each period we harvest a proportion $b$  of the population (in terms of the maximum) then

\begin{equation}
	f(y)=(1+r)y-ry^2-b
\end{equation}

is the corresponding map of this discrete dynamical system. From here we see that the Fixed Points Polynomial (FPP) of $f$ is\footnote{Remember $f_n(y)=y-(-1)^nP_f(y)$ for the $n$-th degree polynomial.}

\[
P_f(y)=ry^2-ry+b,
\]

whose FPs we can determine to be

\[
y_\pm=\frac{1\pm\sqrt{1-4b/r}}{2},
\]

from which we immediately calculate the Linear Factors Form (LFF) to be

\begin{equation}
	f(y)=y-ry(y-y_+)(y-y_-).
\end{equation}

From this form it is also straightforward to determine that in the corresponding canonical form we have

\[
x_1(b)=r(y_+-y_-)=\sqrt{r(r-4b)}.
\]

Both systems are related then by the linear transformations $y_n=y_-+x_n/r$ and $x_n=r(y_n-y_-)$. With this choice, the zero fixed point in the canonical map corresponds to $y_-$ and the nonzero fixed point to $y_+$. To analyze the stability of this system, consider $r$ to be fixed and given and $b$ to be the parameter of this family of systems. The one immediate conclusion is that, for $x_1\in\mathbb{R}$, we necessarily have that $x_1\geq0$. Now, real $x_1$ implies $b\leq r/4$. Over the $r/4$ value $x_1$ becomes complex which does not give any fixed points (but would mean ``over-harvesting''). Since ``harvesting'' cannot be negative, it is clear $b=0$ corresponds to the maximum value of $x_1=r$ and that $x_1=0$ when $b=r/4$. Remember now that to analyze the stability of $x_1$ we consider its PDF, $D_1=-x_1$.　Analogously, $D_0=x_1$. Since the maximum value of $x_1$ is $r$ and, being this the ``unconstrained growth rate'', $0<r<1$, we have that $-r<D_1<0$, therefore putting the $x_1$ in the stability range between $b_0=0$ and $-b_1=-2$ determined by the first stability band. This range is never left in any situation with physical meaning so we conclude that the nonzero fixed point, i.e. $y_+$ in the original system, is always the only asymptotically stable fixed point and the zero fixed point, i.e. $y_-$ in the original system, is always unstable. The only case to analyze with care is $b=0$ since then the two fixed points collide, but proposition \ref{prop:quadsemi} (on page \pageref{prop:quadsemi}) guarantees that in this case $x_1=0$ is semistable from the right.

In conclusion:
\begin{enumerate}
	\item When $b=r/4$ the population faces extinction asymptotically. Over this value extinction is achieved in a finite number of steps, there not being any more fixed points.
	\item When $0<b<r/4$, $x_n\rightarrow r$ and the population tends to $y_+=0.5(1+\sqrt{1-4b/r})$.
	\item When $b=0$, i.e. no harvesting, $x_n$ still tends to $x_1=0$ from the right and, correspondingly, the population tends asymptotically to $y_+=1$.
\end{enumerate}

\section{Remarks on the application of the method}

As somewhat instructive as the above examples may have been, the true simplifying power of the Canonical Map Theory is not yet fully met. Therefore we encourage the reader to apply this method to test the dynamics of any discrete dynamical systems family with a polynomial map depending on a single parameter, and evaluate the usefulness of the method by yourself. Families of maps with particular parametric dependencies and of cubic order will give the most useful applications for the time being, since the theory of higher degree polynomials is so far restricted to analyzing the stability of fixed points only.

\FloatBarrier

\chapter{Conclusions and future work}
\label{cha:Conclusions}

Finally, lets sum up the main findings of this work and give some directions on the future work that arises from it. Basically, we can summarize the findings of this work as having successfully given conditions for the stability of the fixed points of any real polynomial map with real fixed points and that depends on a single parameter. In order to do this we have defined ``canonical polynomial maps'' which are topologically conjugate to any polynomial map of the same degree with real fixed points. Then, the stability of the fixed points of the canonical polynomial maps has been found to depend solely on a special function called ``product distance'' of a given fixed point. The values of this product distance determine the stability of the fixed point and when it bifurcates to give rise to attracting periodic orbits of period $2^n$ for all $n$, and even ultimately when chaos arises through the period doubling cascade, as it passes through different ``stability bands'', although the exact values and widths of these stability bands are yet to be calculated for regions of type greater than one for higher order polynomials. The latter must be done numerically.

As specific goals achieved we can remark that in this work we have

\begin{itemize}
	\item Generalized the existing results on parametric dependency of the coefficients of real quadratic maps for a larger set of functions.
	\item Studied in detail the dynamics of general real cubic maps when the parametric dependency of its coefficients is given by continuous functions.
	\item We have obtained transformations that take general quadratic and cubic maps into one whose dynamics is known and easy to study, and that are also topologically conjugate maps, so that the properties of stability and chaos are preserved.
	\item Explicitly given the set of fixed points for the the proposed 'canonical' quadratic and cubic maps.
	\item Created a method that allows to understand the dynamics of cubic maps, that also includes quadratic maps.
	\item Generalized the latter method to $n$-th degree real polynomial maps.
	\item Proposed a method to generate a large class of regular-reversal cubic maps.
	\item Reproduced the know stability characteristics of the logistic map through the proposed method in a simpler way.
	\item Proposed a methodology that allows to be used to create discrete dynamical systems with some prescribed bifurcation diagram.
\end{itemize}

As future work we can point out the following:

\begin{itemize}
	\item The numerical values of the limits of the stability bands for still higher degree canonical polynomial maps must be calculated with numerical methods.
	\item Ultimately it is desired to obtain extensive tables of the bifurcation values for higher order polynomials.
	\item Algebraic multiplicity of the roots of the Fixed Points Polynomial, when nonhyperbolic fixed points take place, must be analyzed for the general n-th degree Canonical Polynomial Map.
	\item The case of complex fixed points must be included in detail for completeness.
	\item Once the preceding is accomplished, the possibility of extending the theory to apply it to the Taylor Polynomial of any real analytical iteration function can be analyzed.
	\item The power and simplicity of the proposed methodology will best be appreciated with 3rd or higher degree polynomials and when the implications on the Taylor polynomial of any non-linear map are understood.
\end{itemize}

\backmatter

\bibliographystyle{plainnat}
\bibliography{thesis_all}
\appendix
\noappendicestocpagenum
\addappheadtotoc

\chapter{Computational and numerical tools}
\label{app:numerical}

For revision of the algebraic calculations, scripts in the wxMaxima scripting language were used, running under Ubuntu GNU/Linux 11.10. Both wxMaxima and Ubuntu are \emph{free software}\footnote{``Free'' as in ``free speech''.}. wxMaxima is a full-featured Computer Algebra System (CAS) in which some of the plots for this work were also made. Also, for most of the plots and bifurcation diagrams, a Python script shown below was used to calculate and plot most of the bifurcation diagrams presented in this work. The Python script was also ran under Ubuntu GNU/Linux and it uses \texttt{numpy} and \texttt{matplotlib} as libraries for numerical calculations and plotting, respectively.

\begin{verbatim}
import numpy as np
import matplotlib.pyplot as plt
import time

###################### DEFINITION OF FUNCTIONS #################################
def gn(n, xi, s=1, x0=0.1, niter=512, nplot=256): #iterates the canonical cubic map
#niter:     total number of iterations
#nplot:     number of points to return (for graphing)
#x0:        initial condition
#xi:        array containing the 3 fixed points
#n:         order of the canonical map and number of f.p.'s
#s:         sgn(M) from the original polynomial map

x = np.zeros(nplot) #here we store the orbit to be returned for graphing
xtemp = x0 #initialize orbit with initial condition
ssn=np.power(-1,n-1)*np.power(s,n) #the sign parameter

#part of the orbit that will not be graphed:
for i in range(niter-nplot):
prod=1
for j in range(n-1):
prod*=xtemp-xi[j+1]
xtemp *= (1 + ssn*prod)

x[0]=xtemp

#part of the orbit that will be graphed:
for i in range(nplot-1):
prod=1
for j in range(n-1):
prod*=x[i]-xi[j+1]
x[i+1] = x[i]*(1 + ssn*prod)

return x

def pdfn(n,xi): #Product Distance Function Values Array
#xi: vector containing the n fixed points
#k: number of fixed point whose stability we want to calculate
#n: total number of fixed points
pdf=np.zeros(n)     #array to store product distance values for each f.p.

for k in range(n):
prod=1 #initalize the product
for j in range(n): #perform product
if(j != k):
prod *= xi[j] - xi[k]
pdf[k]=prod

return pdf

def xl(i, p, lamb): #define the fixed points as a function of the parameter
#i: number of fixed point to return value
#p: NP 2-d array containing the extra parameters for the fuctional forms
#   of the fixed points. 1st index is the number of fixed point, 2nd is
#   the number of relevant parameter for that fixed point
if(i>0):
#return lamb*(3-lamb) #quadratic FP for QCM
#return p[i][1]*np.sqrt(lamb) + p[i][2] #square root fixed points
return p[i][0]*lamb + p[i][1] #linear fixed points
#return (-1)**p[i,0]*p[i,1]+(-1)**(p[i,0]+1)*np.exp(-p[i,2]*lamb)
else:
return 0 #the zeroth fixed point is always zero in value

####################### DEFINITION OF PARAMETERS ###############################
print time.clock(), "Defining parameters..."

n=4         #number of fixed points and order of the canonical map
npl=256    #number of points to graph for each value of lambda
nit=512 + npl     #total number of iterations
lmin=0     #starting value of lambda for the graph
lmax=1.2      #ending value of lambda for the graph
p=np.zeros([n,3])   #parameters for the definition of the fixed points
p[1,0]=1
p[1,1]=0
#p[1,2]=-1
p[2,0]=-1
p[2,1]=0
#p[2,2]=1
p[3,0]=2
p[3,1]=0
#p=np.array([[0,0,0],[1,0,0],[-1,0,0]]) #for the linear FPs
#p=np.array([[0,0,0],[1,0,0],[-1,0,0]]) #for the linear FPs
#p=np.array([[0,0,0],[0,2,0],[6,1,0]])   #one linear FP another const. FP
lval=np.linspace(lmin,lmax,600) #range of the parameter for the bif. diagram
midl=(lmin+lmax)/2  #middle point of the lambda interval (bif.val. estimation)
uncert=abs(lmax-lmin)/2 #uncertainty for the bif. value estimation
xi=np.zeros(n) #the zero fixed point is always zero in value
#x0 = 0.05 #initial condition
pdf=np.zeros([n,lval.size])
fps=np.zeros([n,lval.size])
lines=['-','--','-.']*3
linw=[2,2,2,3,3,3,4,4,4,5,5,5]

############# PRINT ESTIMATIONS TO SCREEN ######################################
print time.clock(), "Print parameter estimates..."

#we print the values found for the bifurcation
#along with its corresponding propagated uncertainties
print "l_i \in (",lmin,", ",lmax,")"
print "l_i =",midl,"+-",uncert
#print "b_i  \in (", -p[1][1]-np.exp(-lmin),",",-p[1][1]-np.exp(-lmax),")"
#print "b_i =",-p[1][1]-np.exp(-midl),"+-",midl*np.exp(-midl)*uncert

############# BEGIN PLOT #######################################################
print time.clock(), "Beginning Fixed Points plot..."

##draw the graph of the n fixed points varying with the parameter
plt.figure(1)
plt.subplot(313)
plt.grid(True)
plt.xlabel('Parameter')
plt.ylabel('FPs')
plt.xlim(lmin,lmax)
#plt.title('Fixed Points')
#plt.ylim(-2,2)
for k in range(n):
fps[k]=[xl(k,p,lam) for lam in lval]
plt.plot(lval,fps[k,:],'k',label='FP'+str(k), ls=lines[k],linewidth=linw[k])

plt.legend()

print time.clock(), "Beginning PDFs plot..."

#draw the product distance functions of each of the n FP's
plt.subplot(312)
plt.grid(True)
#plt.xlabel('Parameter')
#plt.ylabel('PDFs')6
#plt.ylabel('Stability Conditions')
plt.ylabel('PDFs')
#plt.title('PDFs of the Fixed Points')
plt.xlim(lmin,lmax)
plt.ylim(-4,1)
#print bifurcation values boundaries
plt.plot(lval,np.zeros_like(lval),'k--')
plt.plot(lval,-2*np.ones_like(lval),'k--')
if(n==2):
plt.plot(lval,-2.449*np.ones_like(lval),'k--')
plt.plot(lval,-2.544*np.ones_like(lval),'k--')
plt.plot(lval,-2.5642*np.ones_like(lval),'k--')
#plt.plot(lval,-2.56871*np.ones_like(lval),'k--')
#plt.plot(lval,-2.56966*np.ones_like(lval),'k--')
#plt.plot(lval,-2.569881*np.ones_like(lval),'k--')
plt.plot(lval,-2.57*np.ones_like(lval),'k--')

if(n==3):
plt.plot(lval,-3*np.ones_like(lval),'k--')
plt.plot(lval,-3.236*np.ones_like(lval),'k--')
plt.plot(lval,-3.288*np.ones_like(lval),'k--')
plt.plot(lval,-3.29925*np.ones_like(lval),'k--')

for j in range(lval.size):
pdf[:,j]=pdfn(n,fps[:,j])

for k in range(n):
plt.plot(lval,pdf[k,:],'k',label='PDF'+str(k),ls=lines[k],linewidth=linw[k])
#plt.plot(lval,pdf[k,:],'k',label='SC'+str(k),ls=lines[k],linewidth=linw[k])

plt.legend()

print time.clock(), "Begin bifurcation diagram plotting..."

#draw the bifurcation diagram
plt.subplot(311)
plt.grid(True)
#plt.xlabel('Parameter')
plt.ylabel('Asymptotic Value')
plt.title('Bifurcation Diagram')
plt.xlim(lmin,lmax)
#plt.ylim(-2, 2)
m=0

for lamb in lval: #draws the bifurcation diagram
#    for j in range(n-1): #calculate the fixed points
#        xi[j+1]= xl(j+1,p,lamb)
#
#    aux=pdfn(n,xi)  #calculate PDF of the FP's
#    for k in range(n):
#        pdf[k][m]=aux[k] #store the values for graphing
#plt.subplot(311)
x = gn(n, fps[:,m], x0=0.9, niter=nit, nplot=npl)
plt.plot(lamb*np.ones(npl),x,'k,')
x = gn(n, fps[:,m], x0=-0.01,niter=nit, nplot=npl)
plt.plot(lamb*np.ones(npl),x,'k,')
m+=1

print time.clock(), "Done."

plt.show()
\end{verbatim}

\listoftables
\listoffigures
\end{document}